\pdfoutput=1
\RequirePackage{ifpdf}
\ifpdf 
\documentclass[pdftex]{sigma}
\else
\documentclass{sigma}
\fi

\numberwithin{equation}{section}

\newtheorem{Theorem}{Theorem}[section]
\newtheorem{Proposition}[Theorem]{Proposition}
\newtheorem{Lemma}[Theorem]{Lemma}
 { \theoremstyle{definition}
 \newtheorem{deff}[Theorem]{Definition}
\newtheorem{Remark}[Theorem]{Remark}
\newtheorem{Comment}[Theorem]{Comment}
}

\usepackage{bbm}
\usepackage[all]{xy}

\newcommand{\tr}{\operatorname{Tr}}
\renewcommand{\Re}{\operatorname{Re}}
\renewcommand{\Im}{\operatorname{Im}}
\newcommand{\Ad}{\operatorname{Ad}}
\newcommand{\ad}{\operatorname{ad}}

\newcommand{\e}{{\rm e}}
\newcommand{\mc}[1]{{\mathcal{#1}}}
\newcommand{\got}[1]{{\mathfrak{#1}}}
\newcommand{\db}[1]{{\mathbb{#1}}}

\newcommand{\pa}{\partial}
\newcommand{\K}{{\mathbb{K}}}
\newcommand{\R}{{\mathbb{R}}}
\newcommand{\C}{{\mathbb{C}}}
\newcommand{\N}{{\mathbb{N}}}
 
\newcommand{\gl}{{\mc{L}}}

\newcommand{\g}{{\got{g}}}
\newcommand{\m}{{\got{m}}}

\newcommand{\SL}{{\rm SL}(2,\R)}
\newcommand{\mr}[1]{{\mathrm{#1}}}
\def\ii{{\rm i}}

\newcommand{\un}{{\mathbbm{1}_n}}
\newcommand{\h}{{\got{h}}}

\begin{document}
\allowdisplaybreaks

\newcommand{\arXivNumber}{1903.10721}

\renewcommand{\PaperNumber}{096}

\FirstPageHeading

\ShortArticleName{The Real Jacobi Group Revisited}

\ArticleName{The Real Jacobi Group Revisited}

\Author{Stefan BERCEANU}

\AuthorNameForHeading{S.~Berceanu}

\Address{National Institute for Physics and Nuclear Engineering, Department of Theoretical Physics,\\
 PO BOX MG-6, Bucharest-Magurele, Romania}

\Email{\href{mailto:Berceanu@theory.nipne.ro}{Berceanu@theory.nipne.ro}}
\URLaddress{\url{http://www.theory.nipne.ro/index.php/mcp-home}}

\ArticleDates{Received May 09, 2019, in final form November 25, 2019; Published online December 07, 2019}

\Abstract{The real Jacobi group $G^J_1(\mathbb{R})$, defined as the semi-direct product of the group ${\rm SL}(2,\mathbb{R})$ with the Heisenberg group $H_1$, is embedded in a $4\times 4$ matrix realisation of the group ${\rm Sp}(2,\mathbb{R})$. The left-invariant one-forms on $G^J_1(\mathbb{R})$ and their dual orthogonal left-invariant vector fields are calculated in the S-coordinates $(x,y,\theta,p,q,\kappa)$, and a left-invariant metric depending of 4 parameters $(\alpha,\beta,\gamma,\delta)$ is obtained. An invariant metric depending of $(\alpha,\beta)$ in the variables $(x,y,\theta)$ on the Sasaki manifold ${\rm SL}(2,\mathbb{R})$ is presented. The well known K\"ahler balanced metric in the variables $(x,y,p,q)$ of the four-dimensional Siegel--Jacobi upper half-plane $\mathcal{X}^J_1=\frac{G^J_1(\mathbb{R})}{{\rm SO}(2) \times\mathbb{R}} \approx\mathcal{X}_1 \times\mathbb{R}^2$ depending of $(\alpha,\gamma)$ is written down as sum of the squares of four invariant one-forms, where $\mathcal{X}_1$ denotes the Siegel upper half-plane. The left-invariant metric in the variables $(x,y,p,q,\kappa)$ depending on $(\alpha,\gamma,\delta)$ of a five-dimensional manifold $\tilde{\mathcal{X}}^J_1= \frac{G^J_1(\mathbb{R})}{{\rm SO}(2)}\approx\mathcal{X}_1\times\mathbb{R}^3$ is determined.}

\Keywords{Jacobi group; invariant metric; Siegel--Jacobi upper half-plane; balanced metric; extended Siegel--Jacobi upper half-plane; naturally reductive manifold}
\Classification{32F45; 32Q15; 53C25; 53C22}

\section{Introduction}\label{section1}

The Jacobi group \cite{bs, ez} of degree $n$ is defined as the
semi-direct product $G^J_n=\mr{H}_n\rtimes{\rm Sp}(n,\R)_{\C}$,
where ${\rm Sp}(n,\R)_{\C}= {\rm Sp}(n,\C)\cap {\rm U}(n,n)$ and $\mr{H}_n$ denotes the
$(2n+1)$-dimensional Heisenberg group \cite{sbj,nou, Y02}. To the Jacobi
group $G^J_n $ it is associated a homogeneous manifold, called the
Siegel--Jacobi ball $\mc{D}^J_n$ \cite{sbj}, whose points are in
$\C^n\times\mc{D}_n$, i.e., a partially-bounded space. $\mc{D}_n$ denotes the
Siegel (open) ball of degree $n$. The non-compact Hermitian
symmetric space $ \operatorname{Sp}(n, \R
)_{\C}/\operatorname{U}(n)$ admits a matrix realization as a homogeneous bounded
domain \cite{helg}:
\begin{gather*}
\mc{D}_n:=\big\{W\in M (n, \C )\colon W=W^t,\,\un-W\bar{W}>0 \big\}.
\end{gather*} The Jacobi group is an interesting object in several branches of Mathematics, with important applications in Physics, see references in \cite{jac1,sbj,SB15,BERC08B,gem}.

Our special interest to the Jacobi group comes from the fact that $G^J_n$ is a coherent state (CS) group \cite{lis2,lis,mosc,mv,neeb96,neeb}, i.e., a group which has
orbits holomorphically embedded into a~projective Hilbert space, for a~precise definition see
 \cite[Definition~1]{sb6}, \cite{jac1}, \cite[Section~5.2.2]{GAB} and \cite[Remark~4.4]{gem}. To an element $X$ in the Lie algebra~$\g$
of $G$ we associated a first order differential operator $\db{X}$ on the homogenous space $G/H$, with polynomial holomorphic coefficients, see
\cite{sbl, morse,sbcag} for CS based on hermitian symmetric spaces, where the maximum degree of the polynomial is~2. In~\cite{sb6, SBAG01,last} we have advanced the hypothesis
that for CS groups the coefficients in $\db{X}$ are polynomial, and in~\cite{jac1} we have verified this for~$G^J_1$.

It was proved in \cite{sbj,nou,SB15} that the K\"ahler two-form on $\mc{D}^J_n$, invariant to the action of the Jacobi group $G^J_n$, has the expression
 \begin{gather}
- \ii\omega_{\mc{D}^J_n}(z,W) = \tfrac{k}{2}\tr (\mc{B}\wedge\bar{\mc{B}}) +\mu \tr \big(\mc{A}^t\bar{M}\wedge \bar{\mc{A}}\big), \qquad \mc{A} ={\rm d} z+ {\rm d} W\bar{\eta},\nonumber\\
\mc{B} = M{\rm d} W,\qquad M = (\un-W\bar{W})^{-1},\qquad z\in\C^n, \qquad W\in\mc{D}_n,\label{aabX}
\\ \label{etaZ}
\eta=M(z+W\bar{z}).
\end{gather}
It was emphasized \cite{nou} that the change of coordinates $(z,W)\rightarrow(\eta,W)$, called FC-transform, has the meaning of passing from un-normalized to
normalized Perelomov CS vectors~\cite{perG}. Also, the FC-transform \eqref{etaZ} is a homogeneous K\"ahler diffeomorphism from $\mc{D}^J_n$ to $\C^n\times\mc{D}_n$, in the meaning of the fundamental conjecture for homogeneous K\"ahler manifolds~\cite{DN,gpv, GV}.

We reproduce a proposition which summarizes some of the geometric properties of the Jacobi group and the Siegel--Jacobi ball \cite{SB15,GAB}, see the definitions of the
notions appearing in the enunciation below in \cite{jac1,csg,berr,GAB} and also Appendix~\ref{compl} for some notions on Berezin's quantization:
\begin{Proposition}\label{TOTU}\quad\begin{enumerate}\itemsep=0pt
\item[$(i)$] \looseness=-1 The Jacobi group $G^J_n$ is a unimodular, non-reductive, algebraic group of Harish-Chandra type.
\item[$(ii)$] The Siegel--Jacobi domain $\mc{D}^J_n$ is a homogeneous, reductive, non-symmetric manifold associated to the Jacobi group $G^J_n$ by the generalized Harish-Chandra embedding.
\item[$(iii)$] The homogeneous K\"ahler manifold $\mc{D}^J_n$ is contractible.
\item[$(iv)$] The K\"ahler potential of the Siegel--Jacobi ball is global. $\mc{D}^J_n$ is a Q.-K.~Lu manifold, with nowhere vanishing diastasis.
\item[$(v)$] The manifold $\mc{D}^J_n$ is a quantizable manifold.
\item[$(vi)$] The manifold $\mc{D}^J_n$ is projectively induced, and the Jacobi group $G^J_n$ is a CS-type group.
\item[$(vii)$] The Siegel--Jacobi ball $\mc{D}^J_n$ is not an Einstein manifold with respect to the balanced metric corresponding to the K\"ahler two-form~\eqref{aabX}, but it is one with respect to the Bergman metric corresponding to the Bergman K\"ahler two-form.
\item[$(viii)$]The scalar curvature is constant and negative.
\end{enumerate}
\end{Proposition}
The properties of geodesics on the Siegel--Jacobi disk $\mc{D}^J_1$ have been investigated in \cite{jac1,csg,berr}, while in \cite{GAB} we have considered geodesics on the Siegel--Jacobi ball $\mc{D}^J_n$. We have explicitly determined the equations of geodesics on $\mc{D}^J_n$. We have proved that the FC-transform \eqref{etaZ} is not a geodesic mapping on the non-symmetric space $\mc{D}^J_n$, see definition in~\cite{mikes}.

However, it was not yet anlayzed whether the Siegel--Jacobi ball is a naturally reductive space or not, even if its points are in $\C^n\times \mc{D}_n$, both manifolds being naturally reductive, see Definition~\ref{DEF2} and Proposition~\ref{PR5}. In fact, this problem was the initial point of the present investigation. The answer to this question has significance in our approach~\cite{csg} to the geometry of the Siegel--Jacobi ball via CS in the meaning of Perelomov~\cite{perG}. We have proved in~\cite{ber97} that for symmetric manifolds the FC-transform gives geodesics, but the Siegel--Jacobi ball is not a~symmetric space. Similar properties are expected for naturally reductive spaces \cite{BERA,BERB}.

In the standard procedure of CS, see \cite{sbj,perG}, the K\"ahler two-form on a homogenous manifold~$M$ is obtained from the K\"ahler potential $f(z,\bar{z})=\log K(z,\bar{z})$ via the recipe
\begin{gather}\label{KALP}-\ii\omega=\pa\bar{\pa}f, \end{gather}
where
$K(z,z):=(e_{\bar{z}},e_{\bar{z}})$ is the scalar product of two CS at
$z\in M$. In~\cite{SB15} we have underlined that the metric associated to the K\"ahler
two-form~\eqref{KALP} is a {\it balanced metric}, see more details in Appendix~\ref{compl}.

The real Jacobi group of degree $n$ is defined as $G^J_n(\R):={\rm Sp}(n,\R)\ltimes H_n$,
where $H_n$ is the real $(2n+1)$-dimensional Heisenberg group. ${\rm Sp}(n,\R)_{\C}$~and~$G^J_n$ are isomorphic to~${\rm Sp}(n,\R)$
and~$G^J_n(\R)$ respectively as real Lie groups, see \cite[Proposition~2]{nou}.

We have applied the partial Cayley transform from the Siegel--Jacobi ball to the
 Siegel--Jacobi upper half-plane and we have obtained the balanced
 metric on $\mc{X}^J_n$, see \cite[Proposition~3]{nou}.

However, the mentioned procedure of obtaining the invariant metric on
homogeneous K\"{a}hler manifolds works only for even
dimensional CS manifolds. For example,
starting from the six dimensional real Jacobi group
$G^J_1(\R)=\SL\ltimes H_1$, we have obtained the K\"{a}hler
invariant two-form $\omega_{\mc{X}^J_1}$ \eqref{BFR} on the Siegel--Jacobi upper
half-plane, a four dimensional homogeneous manifold attached to the
Jacobi group, $\mc{X}^J_1=\frac{G^J_1(\R)}{{\rm SO}(2)\times\R}\approx
\mc{X}_1\times\R^2$ \cite{jac1,BER77,FC,csg}, obtained previously by
Berndt~\cite{bern84,bern}, and K\"ahler~\cite{cal3,cal}.

 In the present
paper we determine the invariant metric on a five
dimensional homogeneous manifold, here called the extended Siegel--Jacobi
upper half-plane, denoted
${\tilde{\mc{X}}^J_1}=\frac{G^J_1(\R)}{{\rm SO}(2)}\approx\mc{X}_1\times\R^3$. It will
be important to find applications in Physics of the invariant metric
\eqref{linvG} on the five-dimensional manifold ${\tilde{\mc{X}}^J_1}$.

\looseness=1 In order to obtain invariant metric on odd dimensional manifolds, we are obliged to change our strategy applied previously to get the
invariant metric on homogeneous K\"ahler manifolds. Instead of the
mentioned first order differential operators on $M=G/H$ with holomorphic
polynomial coefficients $\db{X}$ associated to $X$ in the Lie algebra
$\g$ of $G$ \cite{sb6,SBAG01}, we have to use the fundamental vector field $X^*$ associated
with $X$, see Appendix~\ref{ISO}.
We have to abandon the approach in
which the Jacobi algebra is defined as the semi-direct sum
$\got{g}^J_1:= \got{h}_1\rtimes \got{su}(1,1)$, where only the
generators of $\got{su}(1,1)$ have a matrix realization, see~\cite{jac1} and the summary in Section~\ref{prJ}.

The approach of mathematicians is to consider the real Jacobi group $G^J_1(\R)$
as subgroup of ${\rm Sp}(2,\R)$. In the present paper
we follow the notation in \cite{bs, ez} for the real Jacobi group $G^J_1(\R)$,
realized as submatrices of ${\rm Sp}(2,\R)$ of the form
\begin{gather}\label{SP2R}
\left(\begin{matrix} a& 0&b & q\\
\lambda &1& \mu & \kappa\\c & 0& d & -p\\
 0& 0& 0& 1\end{matrix}\right),\qquad M=
 \left(\begin{matrix}a&b\\c&d\end{matrix}\right),\qquad \det M
 =1, \end{gather}
where
\begin{gather}\label{DEFY}Y:=(p,q)=XM^{-1}=(\lambda,\mu) \left(\begin{matrix}a&b\\c&d\end{matrix}\right)^{-1}=(\lambda d-\mu
 c,-\lambda b+\mu a)\end{gather}
is related to the Heisenberg group $H_1$.

To get the invariant metric on ${\tilde{\mc{X}}^J_1}$, we have determined the
invariant one-forms $\lambda_1,\dots,\lambda_6$ on $G^J_1(\R)$, the main tool of the present
paper, see details on the method in Appendix~\ref{S911}. Then we have determined the invariant vector fields $L^j$
verifying the relations $\langle \lambda_i\,|\,L^j\rangle=\delta_{ij}$,
$i,j=1,\dots,6$, such that $L^j$
are orthonormal with respect to the metric
${\rm d}s^2_{G^J_1(\R)}$ in the $S$-variables $(x,y,\theta,p,q,\kappa)$,
see \cite[p.~10]{bs}. This is the idea of the method of the moving
frame of E.~Cartan \cite{cart4,cart5,ev} explained in Section~\ref{S54}.

Firstly, we recover the well known two-parameter balanced metric on $\mc{X}^J_1$ as sum of squares of the invariant one-forms $\lambda_1$, $\lambda_2$, $\lambda_4$, $\lambda_5$. Then the invariant metric on ${\tilde{\mc{X}}^J_1}$ is obtained as the sum of the squares of $\lambda_1,\lambda_2,\lambda_4,\dots,\lambda_6$.

\looseness= 1 The paper is laid out as follows. In Section~\ref{prJ} we recall how we
have obtained the K\"ahler two-forms on the Siegel--Jacobi disk $\mc{D}^J_1$
and on the Siegel--Jacobi upper half-plane $\mc{X}^J_1$, specifying
the FC-transforms. Section~\ref{HS1} describes the real Heisenberg group $H_1$
embedded into ${\rm Sp}(2,\R)$: invariant one-forms, invariant
metrics in the variables $(\lambda,\mu,\kappa)$. Note that in the formula~\eqref{MTRSINV} the last parenthesis $({\rm d} \kappa
-\mu{\rm d}\lambda+\lambda{\rm d} \mu)^2$ replaces $({\rm d} \kappa)^2$ on the
Euclidean space $\R^3(\kappa,\lambda,\mu)$ and the idea of the paper is to see the
effect of this substitution in the invariant metric of the five-dimensional manifold
$\tilde{\mc{X}}^J_1$. Section~\ref{SLSR} deals with the $\SL$ group as
subgroup of ${\rm Sp}(2,\R)$
in the variables $(x,y,\theta)$, which describe the Iwasawa
decomposition.
${\rm SL}(2,\R)$ is treated as a Sasaki manifold, with the invariant
metric written down as sum of squares of the invariant
one-forms $\lambda_1,\dots,\lambda_3$ \`{a} la Milnor~\cite{ml}, while the metric on
$\mc{X}_1$ is just $\lambda_1^2+\lambda_2^2$. Invariant metrics on~$\SL$ in
other coordinates previously obtained by other authors are mentioned in Comment~\ref{COM2}. Details on the calculations referring to $\SL$ are presented also in
Appendix~\ref{FKVF}. Section~\ref{JG1} presents the real Jacobi
group $G^J_1(\R)$
 in the EZ and S-coordinates \cite{bs}. The action of the
reduced Jacobi group~$G^J(\R)_0$ on the
four-dimensional manifold $\mc{X}^J_1$ is recalled~\cite{jac1,BER77}
and the fundamental vector fields on it are obtained. Also the action
of $G^J_1(\R)$ on the 5-dimensional manifold
$\tilde{\mc{X}}^J_1$, called extended Siegel--Jacobi upper half-plane,
is established in Lemma~\ref{LEMN}. The
well known
K\"ahlerian balanced metric on the Siegel--Jacobi upper-half plane is
written down as
sum of the square of four invariant one-forms in Section~\ref{S54}. For this we have
obtained the invariant one-forms on $G^J_1(\R)$ in \eqref{LFLf}.
In Comment \ref{CM1} we discuss the connection of our previous papers~\cite{jac1,mlad,csg}
on $G^J_1(\R)$
with the
papers of Berndt~\cite{bern84,bern,bs} and K\"ahler~\cite{cal3,cal}, developed
by
 Yang \cite{Y02,Yan,Y07,Y08,Y10} for $G^J_n(\R)$. We have also
 determined the Killing vector fields as fundamental
 vector fields on the Siegel--Jacobi upper half-plane with the balanced
 metric~\eqref{METRS2}. The same procedure is used
to establish the invariant metric on the extended Siegel--Jacobi upper
half-plane, which is not a Sasaki manifold. All the results concerning the invariant metrics on
homogenous manifolds of dimensions 2--6 attached to the real Jacobi
group of degree 1 are summarized in
Theorem~\ref{BIGTH}. As a consequence, we show by direct calculation
that the Siegel--Jacobi upper half-plane is not a naturally reductive
space with respect to the balanced metric, but it is one in the
coordinates furnished by the FC-transform. In fact, this is the
answer to the starting point of our investigation referring to the
natural reductivity of $\mc{X}^J_1$. We also calculate the
g.o.\ vectors~\cite{kwv} on $\mc{X}^J_1$ applying the geodesic Lemma~\ref{PRR}.

In four appendices we recall several basic mathematical concepts used in paper. Appendix~\ref{ISO} is devoted to naturally reductive spaces \cite{atri,kn, nomizu}. We have included the notions of Killing vectors, Riemannian homogeneous spaces \cite{as}, the list of 3 and 4-dimensional naturally reductive spaces \cite{btv,bv,kw4, tv}, the famous BCV-spaces \cite{BIANCHI,cart,vr}. Appendix~\ref{compl} recalls the notion of balanced metric in the context of Berezin quantization. The Killing vectors on $S^2$, $\mc{D}_1$, $\R^2$ are presented in Appendix~\ref{sfera}. Appendix~\ref{appendix4} refers to notions on Sasaki manifolds \cite{BL,boga, sas}.

The main results of this paper are stated in Lemma~\ref{LEMN}, Remark~\ref{Rm1}, Propositions~\ref{X15}--\ref{PRLST}, and Theorem~\ref{BIGTH}.

\textbf{Notation.} We denote by $\mathbb{R}$, $\mathbb{C}$, $\mathbb{Z}$,
and $\mathbb{N}$ the field of real numbers, the field of complex numbers,
the ring of integers, and the set of non-negative integers, respectively. We denote the imaginary unit
$\sqrt{-1}$ by~$\ii$, and the Real and Imaginary part of a complex
number by $\Re$ and respectively $\Im$, i.e., we have for $z\in\C$,
$z=\Re z+\ii \Im z$, and $\bar{z}= cc(z)= \Re z-\ii \Im z$. We denote by~$|M|$ or by $\det(M)$ the determinant of the matrix~$M$. $M(n,m,\db{F})$ denotes the set
of $n\times m$ matrices with entries in the field $\db{F}$. We denote by $M(n,\db{F})$ the set $M(n,n,\db{F})$. If $A\in
M_n(\db{F})$, then
$A^t$ ($A^{\dagger}$) denotes the transpose (respectively, the
Hermitian conjugate) of~$A$. $\un$~denotes the identity matrix of
degree $n$. We consider a complex separable Hilbert space~$\got{H}$ endowed with a~scalar product
which is antilinear in the first argument,
$(\lambda x,y)=\bar{\lambda}(x,y)$, $x,y\in\got{H}$,
$\lambda\in\C\setminus 0$. We denote by ``${\rm d}$'' the differential. We use Einstein convention that repeated indices are
implicitly summed over. The set of vector fields (1-forms) are denoted
by $\got{D}^1$ ($\got{D}_1$). If $\lambda\in\got{D}_1$ and $L\in\got{D}^1$, then
$\langle \lambda\,|\,L\rangle $ denotes their pairing.
We use the symbol ``$\tr$'' to denote the trace
of a~matrix. If~$X_i$, $i=1,\dots,n$ are vectors in vector space $V$
over the field $\db{F}$, then $\langle X_1,X_2,\dots,X_n\rangle_{
\db{F}}$ denotes their span over $\db{F}$.

\section{The starting point in the coherent states approach}\label{prJ}
We recall firstly our initial approach \cite{BER05,jac1} to the Jacobi group
$G^J_1$ which we have followed in all our papers devoted to the
Jacobi group, except~\cite{BERC08B} and~\cite{gem}. The Lie algebra
attached to~$G^J_1$ is
\begin{gather*}
\got{g}^J_1:= \got{h}_1\rtimes \got{su}(1,1),
\end{gather*}
where $\got{h}_1$ is an ideal in $\got{g}^J_1$,
i.e., $\big[\got{h}_1,\got{g}^J_1\big]=\got{h}_1$,
determined by the commutation relations
\begin{subequations}\label{baza11}
\begin{gather}
\big [a,a^{\dagger}\big]=1,\label{baza1} \\
\label{baza2}
 \big[ K_0, K_{\pm}\big]=\pm K_{\pm} ,\qquad [ K_-,K_+ ]=2K_0 , \\
 \nonumber[a,K_+ ]=a^{\dagger} ,\qquad \big[ K_-,a^{\dagger}\big]=a, \\
\nonumber\big[ K_+,a^{\dagger}\big]= [ K_-,a ]= 0 ,\\
\nonumber \big[ K_0 , a^{\dagger}\big]=\frac{1}{2}a^{\dagger}, \qquad [ K_0,a ]=-\frac{1}{2}a .
\end{gather}
\end{subequations}
The Heisenberg algebra is \begin{gather*}
\got{h}_1\equiv\got{g}_{H_1}=
\langle \ii s 1+xa^{\dagger}-\bar{x}a\rangle _{s\in\R ,\,x\in\C} ,\end{gather*}
 where ${ a}^{\dagger}$ (${ a}$) are the boson creation
(respectively, annihilation)
operators which verify the canonical commutation relations
(\ref{baza1}). The Lie algebra of the group ${\rm SU}(1,1)$ is
\begin{gather*}
\got{su}(1,1)=
\langle 2\ii\theta K_0+yK_+-\bar{y}K_-\rangle _{\theta\in\R ,\, y\in\C} , \end{gather*}
where the generators $K_0$, $K_+$, $K_-$ verify the standard commutation relations~(\ref{baza2}), and we have considered the matrix realization
\begin{gather}\label{nr2}
K_0=\frac{1}{2}\left(\begin{matrix}1 &0 \\0 &
-1\end{matrix}\right),\qquad
K_+=\ii\left(\begin{matrix}0 &1 \\ 0 & 0 \end{matrix}\right),\qquad
K_-=\ii \left(\begin{matrix}0 &0 \\ 1 & 0 \end{matrix}\right) .
\end{gather}
We have determined the invariant metric on the Siegel--Jacobi upper half-plane $\mc{X}^J_1$ from the metric on
$\mc{D}^J_1$ and the FC-transforms, see
\cite{jac1,BER77,FC,SB15}. For the actions in Proposition \ref{PRFC}, where $G^J_0={\rm SU}(1,1)\ltimes \C$,
see \cite[Proposition~2]{nou} and Lemma \ref{LEMN} below.

\begin{Proposition}\label{PRFC}Let us consider the K\"ahler two-form
\begin{gather}\label{kk1}
 -\ii \omega_{\mc{D}^J_1}
(w,z)=\frac{2k}{(1-|w|^2)^2}{\rm d} w\wedge{\rm d}
 \bar{w}+\mu \frac{A\wedge\bar{A}}{1-|w|^2},\qquad A=A(w,z)={\rm d}
 z+\bar{\eta}{\rm d} w,
\end{gather}
$G^J_0$-invariant to the action on the Siegel--Jacobi disk $\mc{D}^J_1$
\begin{gather*}\left({\rm SU}(1,1)\times\C^2\ni\left(\begin{matrix}p&q\\\bar{q}&\bar{p}\end{matrix}\right),\alpha\right)\cdot
(w,z)=\left(\frac{pw+q}{\bar{q}w+\bar{p}},\frac{z+\alpha-\bar{\alpha}\omega}{\bar{q}w+\bar{p}}\right).\end{gather*}

We have the homogeneous K\"ahler diffeomorphism
${\rm FC}\colon (\mc{D}^J_1,\omega_{\mc{D}^J_1})\rightarrow(\mc{D}_1,\omega_{\mc{D}_1})\oplus
(\C,\omega_{\C})$, $-\ii \omega_{\C}={\rm d}\eta\wedge{\rm d} \bar{\eta},$
\begin{gather*}
{\rm FC}\colon \
 z=\eta-w\bar{\eta},\qquad {\rm FC}^{-1}\colon \ \eta=\frac{z+\bar{z}w}{1-|w|^2}, \end{gather*}
and \[{\rm FC}\colon \ A(w,z)\rightarrow {\rm d} \eta -w{\rm d} \bar{\eta}.\]
The K\"ahler two-form~\eqref{kk1} is invariant to the action $(g,\alpha)\times (\eta,w)= (\eta_1,w_1)$ of $G^J_0$ on $\C\times\mc{D}_1$: $\eta_1=p(\eta+\alpha)+q(\bar{\eta}+\bar{\alpha})$.

Using the partial Cayley transform
\begin{gather*}
\Phi^{-1}\colon \ v=\ii \frac{1+w}{1-w},\qquad u=\frac{z}{1-w}, \qquad w,z\in\C,\qquad |w|<1,\\
\Phi\colon \ w=\frac{v-\ii}{v+\ii},\qquad z=2\ii \frac{u}{v+\ii},\qquad v,u\in\C,\qquad \Im v>0,
\end{gather*}
we get the K\"ahler two-form
\begin{gather}
- \ii \omega_{\mc{X}^J_1}(v,u) = \frac{2k}{(\bar{v} - v)^2} + \frac{2\mu}{\ii(\bar{v} - v)}B\wedge\bar{B}, \nonumber\\
B(v,u) = A\left(\frac{v - \ii}{v+ \ii},\frac{2\ii u}{v + \ii}\right)= {\rm d} u - \frac{u - \bar{u}}{ v- \bar{v}}{\rm d} v,\label{BFR}
 \end{gather}
 $G^J(\R)_0$-invariant to the action on the Siegel--Jacobi upper
half-plane $\mc{X}^J_1$\begin{gather}\label{TSLL}\left(\SL\times\C^2\ni\left(\begin{matrix}a&b\\c&d\end{matrix}\right),\alpha\right)\cdot
(v,u)=\left(\frac{av+b}{cv+d},\frac{u+nv+m}{cv+d}\right),\qquad \alpha= m+\ii n.\end{gather}

We have the homogeneous K\"ahler diffeomorphism
\begin{gather*}
{\rm FC}_1\colon \ \big(\mc{X}^J_1,\omega_{\mc{X}^J_1}\big)\rightarrow(\mc{X}_1,\omega_{\mc{D}_1})\oplus (\C,\omega_{\C}),\\
{\rm FC}_1\colon \ 2\ii u=(v+\ii)\eta-(v-\ii)\bar{\eta}, \qquad {\rm FC}^{-1}_1 \colon \ \eta=\frac{u\bar{v}-\bar{u}v + \ii(\bar{u}-u)}{\bar{v}-v}.
\end{gather*}
The situation is summarized in the commutative diagram of the table FC-transforms
\begin{gather*}
\xymatrix{
 \mc{D}^J_1\ni (\omega;z) \ar[r]^{\rm FC} \ar[d]_{\Phi^{-1}} & (\omega;\eta)\in \mc{D}_1\times\C \ar[d]^{\Phi'^{-1}} \\
 \mc{X}^J_1 \ni (v;u) \ar[r]_{{\rm FC}_1\ } & (v;p,q)\in \mc{X}_1\times\C, }
\end{gather*}
where
\begin{gather*}
\Phi'^{-1}\colon \ \eta\rightarrow q+\ii p, \qquad \Phi'\colon \ (q,p)\rightarrow \eta=q+\ii p.\end{gather*}
\end{Proposition}

We recall that in Proposition \ref{PRFC} the parameters $k$ and $\mu$ come from representation theory of the Jacobi group: $k$ indexes the positive discrete series of ${\rm SU}(1,1)$ ($2k\in\N$), while $\mu>0$ indexes the representations of the Heisenberg group. Note that in the Berndt--K\"ahler approach the K\"ahler potential~\eqref{POT} is just ``guessed'', see Comment~\ref{CM1}.

Here we just verify the invariance of the K\"ahler two-form \eqref{BFR} to the action \eqref{TSLL}, see also Lemma~\ref{LEMN}. We use equations~\eqref{TR11}
\begin{gather}\label{TR11}
{\rm d} v_1=\frac{{\rm d}
 v}{\Lambda^2},\qquad v_1-\bar{v}_1=\frac{v-\bar{v}}{|\Lambda|^2}, \qquad \text{where} \qquad
\Lambda=cv+d,\end{gather}
and the particular case $n=1$ of equations in \cite[p.~17]{nou}
\begin{gather*}
{\rm d} u_1=\frac{{\rm d} u}{\Lambda}+\frac{nd -c(u+m)}{\Lambda^2}{\rm d} v, \qquad B_1=\frac{B}{\Lambda},
\end{gather*}
where $B$ is given in \eqref{BFR}.

\section[The Heisenberg subgroup of ${\rm Sp}(2,\R)$]{The Heisenberg subgroup of $\boldsymbol{{\rm Sp}(2,\R)}$}\label{HS1}

The composition law of the 3-dimensional Heisenberg group $H_1$ in \eqref{SP2R} is:
\begin{gather*}
(\lambda,\mu,\kappa)(\lambda',\mu',\kappa')=(\lambda+\lambda',\mu+\mu',\kappa+\kappa'+\lambda\mu'-\lambda'\mu).
\end{gather*}

We denote an element of $H_1$ embedded in ${\rm Sp}(2,\R)$ as in~\eqref{SP2R} with $M=\mathbbm{1}_2$
\begin{gather}\label{Real2} H_1\ni g = \left(\begin{matrix}
 1& 0& 0 & \mu\\
\lambda &1& \mu & \kappa\\
0 & 0& 1 & -\lambda\\
 0& 0& 0& 1
\end{matrix}\right),\qquad g^{-1} = \left(\begin{matrix} 1& 0& 0& -\mu\\
-\lambda &1& -\mu & -\kappa\\
0 & 0& 1 & \lambda\\
 0& 0& 0& 1\end{matrix}\right).\end{gather}

A base of the Lie algebra $\got{h}_1=\langle P,Q,R\rangle_{\R}$ of the
Heisenberg group $H_1$ in the realization~\eqref{Real2} in the space $M(4,\R)$ consists of the matrices
\begin{gather*}
P=\left(\begin{matrix}
0&0&0&0\\
1&0&0&0\\
0&0&0&-1\\
0&0&0&0\end{matrix}\right),
 \qquad Q=\left(\begin{matrix}
0&0&0&1\\
0&0&1&0\\
0&0&0&0\\
0&0&0&0\end{matrix}\right),\qquad R=\left(\begin{matrix}
0&0&0&0\\
0&0&0&1\\
0&0&0&0\\
0&0&0&0\end{matrix}\right),
\end{gather*}
which verify the commutation relations
\begin{gather}
\label{PQT1}[P,Q]=2R,\qquad [P,R]=[Q,R]=0.
\end{gather}

If we write \[H_1\ni g(\lambda,\mu,\kappa)=\mathbbm{1}_4+\lambda P+\mu Q+\kappa R,\] then, using the formulas, see details in Appendix~\ref{NR111},
\begin{gather*}
g^{-1}{\rm d} g=P\lambda^p+Q\lambda^q+R\lambda^r,\qquad {\rm d} g g^{-1}=P\rho^p+Q\rho^q+R\rho^r,\end{gather*}
we find the left-invariant one-forms (vector fields)
\begin{gather}\label{LEFT1}
 \begin{cases}
\lambda^p = {\rm d} \lambda,\\
\lambda^q = {\rm d} {\mu},\\
\lambda^r = {\rm d} {\kappa}- \lambda{\rm d} {\mu} +{\mu}{\rm d} \lambda, \end{cases}
\qquad \begin{cases} L^p=\pa _{\lambda}
 -{\mu}\pa_ {\kappa},\\
 L^q= \pa_{\mu}+\lambda \pa_{\kappa},\\ L^r=\pa_{\kappa},
\end{cases}\end{gather}
and the right-invariant one-forms (respectively vector fields)
\begin{gather*}
 \begin{cases}
\rho^p = {\rm d} {\lambda},\\
\rho^q = {\rm d} {\mu},\\
\rho^r = {\rm d} {\kappa}-{\mu}{\rm d} {\lambda} +{\lambda}{\rm d} {\mu}, \end{cases}
\qquad \begin{cases} R^p=\pa _{\lambda}+{\mu}\pa_{\kappa},\\
R^q=\pa_{\mu}-{\lambda}\pa_{\kappa},\\ R^r=\pa_{\kappa}.\end{cases}
\end{gather*}
We have the commutation relations
\[[L^p,L^q]=2L^r,\qquad [R^p,R^q]=-2R^r. \]
We get the right (respectively left) invariant metric on $H_1$
\begin{subequations}\label{MTRSINV}
\begin{gather}
g^R_{H_1}({\lambda},{\mu},{\kappa}) =(\rho^p)^2+ (\rho^q)^2+(\rho^r)^2={\rm d} {\lambda}^2+{\rm d} {\mu}^2 +({\rm d} {\kappa}-{\mu}{\rm d} {\lambda}
 +{\lambda}{\rm d} {\mu})^2,\\
 g^L_{H_1}({\lambda},{\mu},{\kappa}) =(\lambda^p)^2+
(\lambda^q)^2+(\lambda^r)^2= {\rm d} {\lambda}^2+{\rm d} {\mu}^2 +({\rm d} {\kappa}-{\lambda}{\rm d} {\mu} +{\mu}{\rm d} {\lambda})^2.
\end{gather}
\end{subequations}

The left (right) invariant action of $H_1$ on itself is given by
\begin{subequations}\label{actPQR}
\begin{gather}\exp
 ({\lambda}P+{\mu}Q+{\kappa}R)({\lambda}_0,{\mu}_0,{\kappa}_0) = ({\lambda}+{\lambda}_0,{\mu}+{\mu}_0,{\kappa}+{\kappa}_0+{\lambda}{\mu}_0-{\mu}{\lambda}_0),\label{actPQRL}\\
({\lambda}_0,{\mu}_0,{\kappa}_0)\exp({\lambda}P+{\mu}Q+{\kappa}R) = ({\lambda}+{\lambda}_0,{\mu}+{\mu}_0,{\kappa}+{\kappa}_0+{\lambda}_0{\mu}-{\mu}_0{\lambda}).\label{actPQRR}
\end{gather}\end{subequations}
With \eqref{actPQR}, we calculate the fundamental vector fields
\[P^*=\pa_{\lambda}+{\mu}\pa_{\kappa},\qquad Q^*=\pa_{\mu}-\lambda\pa_{\kappa}, \qquad R^*=\pa_{\kappa}.\]

\section[The ${\rm SL}(2,\R)$ subgroup of ${\rm Sp}(2,\R)$]{The $\boldsymbol{{\rm SL}(2,\R)}$ subgroup of $\boldsymbol{{\rm Sp}(2,\R)}$}\label{SLSR}

 An element $M\in \SL$ is realized as an element in ${\rm Sp}(2,\R)$ by the relation
\begin{gather}\label{ALOS}
M=
\left(\begin{matrix}a&b\\c&d\end{matrix}\right) \rightarrow
g = \left( \begin{matrix} a& 0&b
 &0\\0&1&0&0\\c&0&d&0\\0&0&0&1\end{matrix} \right)\!\in G^J_1(\R),
\qquad g^{-1} = \left(\begin{matrix} d& 0&-b
 &0\\0&1&0&0\\-c&0&a&0\\0&0&0&1\end{matrix}\right). \end{gather}
A basis of the Lie algebra $\got{sl}(2,\R)=\langle F,G,H\rangle_{\R}$ consists of the matrices in $M(4,\R)$
\begin{gather*}
F= \left(\begin{matrix} 0& 0&1
 &0\\0&0&0&0\\0&0&0&0\\0&0&0&0\end{matrix}\right), \qquad G= \left(\begin{matrix} 0& 0&0
 &0\\0&0&0&0\\1&0&0&0\\0&0&0&0\end{matrix}\right),\qquad H= \left(\begin{matrix} 1& 0&0
 &0\\0&0&0&0\\0&0&-1&0\\0&0&0&0\end{matrix}\right) .\end{gather*}
$F$, $G$, $H$ verify the commutation relations \eqref{FGHCOM}. With the representation \eqref{ALOS}, we have
\begin{gather*}
g^{-1}{\rm d} g=F\lambda^f+G\lambda^g+H\lambda^h,\qquad {\rm d} g g^{-1}=F\rho^f+G\rho^g+H\rho^h.\end{gather*}
Using the parameterization \eqref{ALOS} for $\SL$, we find
\begin{gather}\label{LEFTRIGHT}
 \begin{cases}
\lambda^f = d{\rm d} b-b{\rm d} d,\\
\lambda^g = -c {\rm d} a + a{\rm d} c,\\
\lambda^h = d {\rm d} a - b {\rm d} c=c{\rm d} b-a {\rm d} d,\end{cases}
\qquad \begin{cases} \rho^f=-b {\rm d} a +a{\rm d} b,\\
 \rho^g= d {\rm d} c - c{\rm d} d,\\ \rho^h=d {\rm d} a- c{\rm d} b.
\end{cases} \end{gather}

We use the notation of \cite[Section~1.4]{bs}. The Iwasawa decomposition $M=NAK$ of an element~$M$ as in~\eqref{ALOS} reads
\begin{gather}\label{MNAK}M=\left(\begin{matrix}1&x\\0&
 1\end{matrix}\right)
\left(\begin{matrix}y^{\frac{1}{2}}&
 0\\0& y^{-\frac{1}{2}}
 \end{matrix}\right)
\left(\begin{matrix}
\cos\theta &\sin\theta\\-\sin\theta
 &\cos\theta \end{matrix}\right),\qquad y>0.\end{gather}
Comparing \eqref{MNAK} with \eqref{ALOS}, we find
\begin{subequations}\label{SCXYT}
\begin{gather}
a = y^{1/2}\cos\theta - xy^{-1/2}\sin\theta,\qquad b = y^{1/2}\sin\theta + xy^{-1/2}\cos\theta,\\
c = -y^{-1/2}\sin\theta,\qquad d = y^{-1/2}\cos\theta,
\end{gather}
\end{subequations}
and
\begin{gather}\label{SCINV}
x=\frac{ac+bd}{d^2+c^2}, \qquad y=\frac{1}{d^2+c^2},\qquad \sin\theta =-\frac{c}{\sqrt{c^2+d^2}},\qquad \cos\theta =\frac{d}{\sqrt{c^2+d^2}}.
\end{gather}

From \eqref{SCXYT}, we get the differentials
\begin{subequations}\label{SCXYTD}
\begin{gather}
{\rm d} a = -y^{-\frac{1}{2}}\sin\theta {\rm d} x +\frac{1}{2y^{\frac{1}{2}}}\left(\cos\theta+
 \frac{x}{y}\sin\theta\right){\rm d} y-y^{\frac{1}{2}}\left(\sin\theta+\frac{x}{y}\cos\theta\right){\rm d}\theta,\\
{\rm d} b = y^{-\frac{1}{2}}\cos\theta {\rm d} x +\frac{1}{2y^{\frac{1}{2}}}\left(\sin\theta-
 \frac{x}{y}\cos\theta\right){\rm d} y+y^{\frac{1}{2}}\left(\cos\theta -\frac{x}{y}\sin\theta\right){\rm d}\theta,\\
{\rm d} c = y^{-1/2}\left(\frac{\sin\theta}{2y}{\rm d} y-\cos\theta{\rm d}\theta\right),\qquad
{\rm d} d = -y^{-1/2}\left(\frac{\cos\theta}{2y}{\rm d} y +\sin\theta{\rm d} \theta\right).
\end{gather}
\end{subequations}

Let $M,M',M_1\in\SL$ such that $M M'=M_1$, matrices of
the form \eqref{ALOS}. With \eqref{SCXYT} we calculate the explicit action
of $M\in\SL$ on $(x',y',\theta')$
\begin{subequations}\label{ALIGNN}\begin{gather} x_1+\ii y_1 =\frac{(ax'+b)(cx'+d)+ac y'^2+\ii
 y'}{\Lambda}, \qquad \Lambda=(cx'+d)^2+(cy')^2,\label{ALIGNN1}\\
\sin\theta_1 =\frac{(cx'+d)\sin\theta' - cy'\cos\theta' }{\sqrt{\Lambda}},\qquad \cos\theta_1
=\frac{cy'\sin\theta'+(cx'+d)\cos \theta'}{\sqrt{\Lambda}}\label{ALLIGN2}.
\end{gather}
\end{subequations}
With \eqref{ALIGNN}, we find out
\begin{gather*}
\frac{{\rm d} x_1}{y_1} =\frac{[(cx'+d)^2-(cy')^2]{\rm d} x'+ 2cy'(cx'+d){\rm d} y'}{y'\Lambda},\\
\frac{{\rm d} y_1}{y_1} = \frac{- 2cy'(cx'+d){\rm d} x'+[(cx'+d)^2-(cy')^2]{\rm d} y'}{y'\Lambda},\\
{\rm d} \theta_1 ={\rm d} \theta' +c\frac{cy'{\rm d} x'-(cx'+d){\rm d} y'}{\Lambda}.
\end{gather*}
Introducing \eqref{SCXYT} and \eqref{SCXYTD} into \eqref{LEFTRIGHT}, we find the left
(right)-invariant one-forms $\lambda$-s with respect to the action $M \cdot (x',y',\theta')= (x_1,y_1,\theta_1)$, $M\in\SL$ and
$(x_1,y_1,\theta_1)$ given by~\eqref{ALIGNN}
 (respectively $\rho$-s), 
\begin{gather}\label{LEFTRIGHT2}
 \begin{cases}
\lambda^f = \dfrac{{\rm d} x}{y}\cos^2\theta+\dfrac{{\rm d} y}{2y}\sin 2\theta +{\rm d}\theta,\vspace{1mm}\\
\lambda^g = -\dfrac{{\rm d} x}{y}\sin^2\theta + \dfrac{{\rm d} y}{2y}\sin 2\theta -{\rm d} \theta,\vspace{1mm}\\
\lambda^h = - \dfrac{{\rm d} x}{2y}\sin 2\theta+\dfrac{{\rm d} y}{2y}\cos 2\theta,\end{cases}
\qquad \begin{cases} \rho^f={\rm d} x
 -\dfrac{x}{y}{\rm d} y +\dfrac{x^2+y^2}{y}{\rm d} \theta,\vspace{1mm}\\
 \rho^g=-\dfrac{{\rm d} \theta}{y},\vspace{1mm}\\
 \rho^h= \dfrac{{\rm d} y}{2y}- \dfrac{x}{y}{\rm d} \theta .
\end{cases}\end{gather}
We determine the left-invariant vector fields $L^f$, $L^g$, $L^h$ on $\SL$, dual orthogonal to the left-invariant one-forms $\lambda^f$, $\lambda^g$, $\lambda^h$~\eqref{LEFTRIGHT2}
\begin{subequations}\label{MNOP}
\begin{gather}
L^f =y\cos2\theta\frac{\pa}{\pa x}+ y\sin 2\theta\frac{\pa}{\pa y}+\sin^2\theta\frac{\pa}{\pa\theta},\\
L^g = y\cos 2\theta\frac{\pa}{\pa x}+y\sin 2\theta\frac{\pa}{\pa y}-\cos^2\theta\frac{\pa}{\pa \theta},\\
L^h = -2y\sin 2\theta\frac{\pa}{\pa x}+2y \cos 2\theta\frac{\pa}{\pa y}+\sin 2\theta\frac{\pa}{\pa\theta},
\end{gather}
\end{subequations}
which verify the commutation relations \eqref{FGHCOM} of the generators $F$, $G$, $H$ of the Lie algebra $\got{sl}(2,\R)$. In fact, $L^f$, $L^g$, $L^h$ in~\eqref{MNOP}
give the Lie derivative of the matrices $F$, $G$, respectively~$H$, see, e.g., \cite[p.~114]{LANG}.

From \eqref{LEFTRIGHT2}, we also get
\begin{gather*}
\lambda^f +\lambda^g =\frac{1}{y}(\cos 2\theta{\rm d} x+ \sin2\theta {\rm d} y), \\
2\lambda^h =\frac{1}{y}(-\sin 2 \theta {\rm d} x + \cos 2 \theta{\rm d} y ), \\
\lambda^f- \lambda^g = \frac{{\rm d} x}{y} + 2{\rm d} \theta,
 \end{gather*}
and
\begin{gather*}
\frac{1}{2}\big(L^f+L^g\big) = y\cos 2\theta\frac{\pa}{\pa x}+ y\sin 2\theta\frac{\pa}{\pa y} -\frac{1}{2}\cos 2\theta \frac{\pa}{\pa \theta},\\
\frac{1}{2}\big(L^f-L^g\big) = \frac{1}{2}\frac{\pa}{\pa\theta},\\
\frac{1}{2}L^h = -y\sin 2\theta \frac{\pa}{\pa x}+y\cos 2\theta \frac{\pa}{\pa y}+\frac{1}{2}\sin 2\theta \frac{\pa}{\pa \theta}.
\end{gather*}
Taking $\alpha, \beta>0$, let us introduce the left invariant one-forms
\begin{subequations}\label{1l2l3l}
\begin{gather}
\lambda_1 =\sqrt{\alpha}\big(\lambda^f+\lambda^g\big)=\frac{\sqrt{\alpha}}{y}(\cos 2\theta{\rm d} x+ \sin 2\theta {\rm d} y), \\
\lambda_2 =2\sqrt{\alpha}\lambda^h=\frac{\sqrt{\alpha}}{y}(-\sin 2\theta {\rm d} x+\cos 2\theta {\rm d} y),\\
\lambda_3 =\sqrt{\beta}\big(\lambda^f-\lambda^g\big)=\sqrt{\beta} (\frac{{\rm d} x}{y} + 2{\rm d} \theta). \label{3ll}
\end{gather}
\end{subequations}
The parameters $\alpha$, $\beta$ introduced in \eqref{1l2l3l} will appear in the invariant metrics~\eqref{GGG} on ${\rm SL}(2,\R)$ and~\eqref{MTRTOT} on $G^J_1(\R)$, while only $\alpha$ will appear in the invariant metrics~\eqref{begG} on~$\mc{X}^J_1$ and~\eqref{linvG} on~$\tilde{\mc{X}}^J_1$.

Note that in the commutation relations of the generators
$e_1$, $e_2$, $e_3$ of the Lie algebra $\got{sl}(2,\R)$
\[[e_1,e_2]=-4\frac{\alpha}{\sqrt{\beta}}e_3,\qquad [e_2,e_3]=4\sqrt{\beta}e_1,\qquad [e_3,e_1]= \sqrt{\beta}e_2,\]
where \[e_1=\sqrt{\alpha}(F+G),\qquad e_2=2\sqrt{\alpha}H,\qquad e_3=\sqrt{\beta}(F-G),\]
there are 2 positive structure constants and one negative, as in the scheme of classification of three-dimensional unimodular Lie groups, see \cite[p.~307]{ml}.

We determine the left-invariant vector fields $L^j$ such that $\langle \lambda_i\,|\,L^j\rangle =\delta_{ij}$, $i,j=1,2,3$,
\begin{subequations}\label{1L2L3L}
\begin{gather}
L^1 =\frac{1}{2\sqrt{\alpha}}\big(L^f+L^g\big) =\frac{1}{\sqrt{\alpha}}\left( y\cos 2\theta\frac{\pa}{\pa x}+ y\sin 2\theta\frac{\pa}{\pa y} -\frac{1}{2}\cos 2\theta \frac{\pa}{\pa \theta}\right),\\
L^2 =\frac{1}{2\sqrt{\alpha}}L^h= \frac{1}{\sqrt{\alpha}}\left( -y\sin 2\theta \frac{\pa}{\pa x}+y\cos 2\theta \frac{\pa}{\pa y}+\frac{1}{2}\sin 2\theta \frac{\pa}{\pa \theta}\right),\\
L^3 =\frac{1}{2\sqrt{\beta}}\big(L^f-L^g\big) = \frac{1}{2\sqrt{\beta}}\frac{\pa}{\pa \theta}.\label{3LL}
\end{gather}
\end{subequations}
If we take in \eqref{1L2L3L} the limit $\theta\rightarrow 0$, we project the invariant vector fields of $\SL$ on the Siegel half-plane $\mc{X}_1=\big\{(x,y)\in\R^2\,|\,y>0\big\}$, and we recover the invariant vector fields which appear in Theorem~\ref{THM0}(2) equation~\eqref{GBELTR}
\begin{gather}\label{E1E2}
l^1_0=\frac{y}{\sqrt{\alpha}}\frac{\pa}{\pa x},\qquad l^2_0=\frac{y}{\sqrt{\alpha}}\frac{\pa}{\pa y}.
\end{gather}

Now we calculate the fundamental vector fields $f^*$, $g^*$, $h^*$ of manifold $\SL$ attached to the base $F$, $G$, respectively~$H$, invariant to the action
$(x,y,\theta)\cdot(x',y',\theta')=(x_1,y_1,z_1)$ given by the composition law $MM'=M_1$, applying~\eqref{expFGH}, \eqref{SCXYT}, \eqref{SCINV}:
\begin{subequations}\label{NUST1}
\begin{gather}
f^* =F^*_1=\frac{\pa}{\pa x},\\
g^* =G^*_1-y\frac{\pa}{\pa \theta}= \big(y^2-x^2\big)\frac{\pa}{\pa x}- 2xy\frac{\pa}{\pa y}-y\frac{\pa}{\pa \theta},\\
h^* =H^*_1=2\left(x\frac{\pa}{\pa x}+y\frac{\pa}{\pa y}\right),
\end{gather}
\end{subequations}
where we have denoted with a subindex 1 the fundamental vector fields~\eqref{FUNDFGH} of~$\SL$, corresponding to the action of the
group on the Siegel upper half-plane $\mc{X}_1$. In fact, $F^*_1$, $G^*_1$, $H^*_1$ are $\db{F}$, $\db{G}$, $\db{H}$ in the convention of Section~\ref{section1}. Evidently, the vector fields~\eqref{NUST1} verify the same commutation relations as $F$, $G$, $H$, with a minus sign.

Using \eqref{expFGH2} or directly with \eqref{NUST1}, we calculate the fundamental vector fields $v^*$, $h1^*$, $w^*$ of $\SL$ corresponding to
\begin{subequations}\label{NASOL}
\begin{gather} v=\sqrt{\alpha}(F+G),\qquad h1=2\sqrt{\alpha}H, \qquad w=\sqrt{\beta}(F-G),\nonumber\\
v^* =\sqrt{\alpha}\left[\big(1-x^2+y^2\big)\frac{\pa}{\pa x}-2xy\frac{\pa}{\pa y}-y\frac{\pa}{\pa\theta}\right],\\
h1^* = 4\sqrt{\alpha}\left(x\frac{\pa}{\pa x}+y\frac{\pa}{\pa y}\right),\\
w^* = \sqrt{\beta}\left[\big(1+x^2-y^2\big)\frac{\pa}{\pa x}+2xy\frac{\pa}{\pa y}+y\frac{\pa}{\pa \theta}\right].
\end{gather}
\end{subequations}

Now we consider $\SL$ as a contact manifold in the meaning of Definition~\ref{DFF11}. Firstly we define an almost contact
structure $(\Phi,\xi,\eta)$ as in Definition~\ref{D9}. We take $\eta=\lambda_3$ and $\xi=L^3$, verifying~\eqref{ACM1}. We have ${\rm d} \eta={\rm d} \lambda_3= \sqrt{\beta}\frac{{\rm d} x\wedge {\rm d} y}{y^2}$, and the condition \eqref{CNT} (with $n=1$) that $\eta$ be a contact
form is verified. The only nonzero component of the associated
two form $\hat{\Phi}$ in \eqref{DFI} is $\hat{\Phi}_{xy}=\frac{\sqrt{\beta}}{2y^2}$, i.e.,
\begin{gather}\label{831}
\hat{\Phi}=\left(\begin{matrix}\hat{\Phi}_{xx}&\hat{\Phi}_{xy}&
 \hat{\Phi}_{x\theta}\\
\hat{\Phi}_{yx}&\hat{\Phi}_{yy}&\hat{\Phi}_{y\theta}\\
\hat{\Phi}_{\theta x}& \hat{\Phi}_{\theta y}&
 \hat{\Phi}_{\theta\theta}\end{matrix}\right)=
\left(\begin{matrix} 0 &\dfrac{\sqrt{\beta}}{2y^2} & 0\\
-\dfrac{\sqrt{\beta}}{2y^2} & 0 & 0\\
0& 0& 0
\end{matrix}\right).
\end{gather} From \eqref{2PHI}
applied to the metric matrix \eqref{GGG} of $\SL$, we get
\[\Phi=\left(\begin{matrix} \Phi^x_x & \Phi^x_y &
 \Phi^x_{\theta}\\
\Phi^y_x& \Phi^y_y & \Phi^y_{\theta}\\
\Phi^{\theta}_x&\Phi^{\theta}_y&
 \Phi^{\theta}_{\theta}\end{matrix}\right)=
\left(\begin{matrix} 0& \dfrac{\sqrt{\beta}}{2\alpha} &0\\
-\dfrac{\sqrt{\beta}}{2\alpha}& 0& 0\\0 & -\dfrac{\sqrt{\beta}}{4\alpha
 y}& 0\end{matrix}\right) .\]

It is convenient to work with the matrix
\begin{gather}\label{PHIP}
\Phi'=\frac{2\alpha}{\sqrt{\beta}}\Phi=\left(\begin{matrix} 0 & 1 &
 0
 \\ -1 & 0 & 0\\
 0 &-\frac{1}{2y} &0\end{matrix}\right) .
\end{gather}
With $(\eta,\xi,\Phi)$ chosen as $\big(\lambda_3,L^3,\Phi'\big)$, equation~\eqref{ACM2} and the conditions of Theorem~\ref{THM11} for an almost contact structure for the manifold $\SL$, where $\operatorname{Rank}(\Phi) =2$, are verified.

We have \begin{gather}\label{V1V2}
\text{Ker}(\eta)=\langle V_1,V_2\rangle =\left\langle \frac{\pa}{\pa
 x}-\frac{1}{2y}\frac{\pa}{\pa \theta},\frac{\pa}{\pa y}\right\rangle ,\end{gather}
and we can write the $(1,1)$-tensor $\Phi'$~\eqref{DEFACD} as
\[\Phi'=-V_2\otimes {\rm d} x+V_1\otimes {\rm d} y.\]
If $X=A\frac{\pa}{\pa x}+B\frac{\pa}{\pa y}+C\frac{\pa }{\pa \theta}$, then $\Phi' X= BV_1-AV_2$, and the contact distribution $\mc{D}=\langle V_1,V_2\rangle $ verifies the condition of Remark~\ref{REM11}.

We also observe that $\SL$ is a homogeneous contact manifold in the sense of Definition~\ref{BOO}.

Now we construct the 4-dimensional symplectization $(C(\SL), \omega,\bar{g})$ of $\SL$,
where \begin{gather}\label{BARG}
\bar{g}(r,x,y,\theta)={\rm d} r^2+r^2g_{{\rm SL}(2,\R)}(x,y,\theta),\qquad \omega= {\rm d} \big(r^2\lambda_3\big).
\end{gather}
In order to see that the Riemann cone $(C(\SL),
\omega,\bar{g})$ of the manifold $\SL$ is normal in the
sense of Definition~\ref{DF17}, we calculate the components
\eqref{TENN} of the $(1,2)$-tensor $N^1$~\eqref{NN1}, using equations~\eqref{831} and \eqref{PHIP}. Because the tensor~\eqref{TENC} is antisymmetric in the lower indexes~$i$,~$j$, we have to calculate only the 9 components $\big(N^1\big)^i_{x,y}$, $\big(N^1\big)^i_{x,\theta}$, $\big(N^1\big)^i_{y,\theta}$, $i=x,y,\theta$, which were found to be 0. In accord with Definition~\ref{DF17}, the Riemann cone $(C(\SL),\omega,\bar{g})$ is Sasaki, and, in accord with Theorem~\ref{THM14}, it is a K\"ahler manifold.

It can be verified that the vector $\xi=L^3$ is a Killing vector for the metric \eqref{MTRSL}, and $\SL$ has a K-contact
 structure, in the sense of Definition~\ref{DF17}. In fact, with
 Remark~\ref{RMM1}, it is verified that $\frac{\pa}{\pa x}$ and
$\frac{\pa }{\pa \theta}$ are Killing vectors for the metric
\eqref{GGG} below, because none of the coordinates~$x$ and~$\theta$
appear explicitly in~\eqref{MTRSL}. For completness, if
$X=X^1\frac{\pa}{\pa x}+X^2\frac{\pa }{\pa y}+X^3\frac{\pa }{\pa \theta}$
 then the equations~\eqref{LG1} of the Killing vectors in the case of the homogeneous metric \eqref{GGG} are
\begin{subequations}\label{KILSL}
\begin{gather}
 -2(\alpha+\beta)X^2+2(\alpha+\beta)y\pa_xX^1+4\beta y^2\pa_xX^3=0,\\
 \alpha\pa_xX^2+(\alpha+\beta)\pa_yX^1+2\beta y\pa_yX^3 =0,\\
 -2\beta X^2+2\beta y \pa_xX^1+(\alpha+\beta)\pa_{\theta}X^1+
2\beta y \pa_{\theta}X^3=0,\\
 -X^2+y\pa_yX^2=0,\\
 2\beta y\pa_{y}X^1+4\beta y^2\pa_yX^3+\alpha\pa_{\theta}X^2=0,\\
 \beta\pa_{\theta}X^1+2\beta y\pa_{\theta}X^3+\pa_{\theta}X^1+2\beta y \pa_{\theta}X^3=0.
\end{gather}
\end{subequations}

In fact, we have
\begin{Proposition}\label{X15}The metric on the group $\SL$, invariant to the
action \eqref{ALIGNN}, is
\begin{gather}{\rm d} s^2_{{\rm SL}(2,\R)} (x,y,\theta) =\lambda_1^2+\lambda_2^2+\lambda_3^2 =\alpha\frac{{\rm d} x^2+{\rm d} y^2}{y^2}+\beta\left(\frac{{\rm d}
 x}{y}+2{\rm d}\theta\right)^2\nonumber\\
 \hphantom{{\rm d} s^2_{{\rm SL}(2,\R)} (x,y,\theta)}{} = \frac{(\alpha+\beta){\rm d} x^2+\alpha {\rm d}
 y^2}{y^2}+4\beta{\rm d} \theta^2 +4\frac{\beta}{y}{\rm d} x{\rm d} \theta.\label{MTRSL} \end{gather}
The matrix associated with the metric \eqref{MTRSL} is
\begin{gather}g_{{\rm SL}(2,\R)} (x,y,\theta)=\left(\begin{matrix}g_{xx}&
 0 & g_{x\theta}\\0& g_{yy} & 0\\ g_{\theta x} & 0 &
 g_{\theta\theta}\end{matrix}\right),\nonumber\\
 g_{xx}=\frac{\alpha+\beta} {y^2}, \qquad g_{yy}=\frac{\alpha}{y^2},\qquad g_{\theta\theta}= 4\beta, \qquad g_{x\theta}=\frac{2\beta}{y} .\label{GGG}
\end{gather}
The invariant vector fields $L^1$, $L^2$, $L^3$ given by \eqref{1L2L3L} are orthonormal with respect to the metric~\eqref{MTRSL}.
The Killing vector fields associated to the metric~\eqref{MTRSL}, solutions of the equations~\eqref{KILSL}, are given by \eqref{NASOL}.

$\big(L^3,\lambda_3,\Phi'\big)$ defines an almost contact structure on~$\SL$, where $L^3$, $\lambda_3$, $\Phi'$ are given respectively by \eqref{3LL}, \eqref{3ll}, \eqref{PHIP}. $\lambda_3$ is the contact structure for $\SL$, $L^3$ is the Reeb vector and the contact distribution $\mc{D}$ is given by~\eqref{V1V2}. $\big(\SL (x,y,\theta),\mc{X}_1,{\rm d} s^2_{\mc{D}}\big)$ is a sub-Riemannian manifold and
 \[{\rm d} s^2_{{\rm SL}(2,\R)}(x,y,\theta)={\rm d} s^2_{\mc{X}_1}+\lambda^2_3,\]
where ${\rm d} s^2_{\mc{X}_1}$ is the $($Beltrami$)$ K\"ahler metric~\eqref{GBELTR},
\eqref{GBETR1}, \begin{gather}\label{MBER}
{\rm d} s^2_{\mc{X}_1}= \lambda^2_1+\lambda_2^2=\alpha\frac{{\rm d}
 x^2+{\rm d} y^2}{y^2}\end{gather} on
the Siegel upper half-plane $x,y,\in\R,~y>0$. The invariant vector fields $l^1_0$, $l^2_0$ given by \eqref{E1E2} are orthonormal with respect to the metric~\eqref{MBER}.

The manifold $\SL$ admits the homogenous contact metric structure $\big(\lambda_3,L^3,\Phi',g_{{\rm SL}(2,\R)}\big)$. The group $\SL$ has the K-contact structure associated with $\xi=L^3$, and it is a Sasaki manifold with the Riemann cone $(C(\SL),\omega,\bar{g})$ with respect to the metric~\eqref{GGG}, where $\bar{g}$ and $\omega$ are given by~\eqref{BARG}.
\end{Proposition}

The last assertion in Proposition \eqref{X15} is well known, see \cite[Example~7]{ALB}.

Now we enumerate other invariant metrics on $\SL$ appearing in literature, different of the metric \eqref{MTRSL} in Proposition~\ref{X15}:

\begin{Comment}\label{COM2}An explicit invariant metric on $\SL$ in coordinates different of the coordinates $(x,y,\theta)$ appears in \cite[Theorem~2, p.~141]{kw3}, see also
Theorem~\ref{thm1}. A different form of the invariant metric on $\SL$ appears in the context of the BCV spaces, see our Remark~\ref{RM5}, where we have applied the Cayley transform to the metric appearing in Theorem~\ref{PR15}, reproduced after~\cite{JVAN} and \cite[Example~2.1.10, p.~59]{calin}. See also \cite[Proposition~2.2, p.~1072 and equation~(2.14)]{vp}.
\end{Comment}

We also give a direct proof of the some well-known facts, see (b2) in Theorem~\ref{thm1}.
\begin{Remark}\label{Rm1}The Siegel upper half-plane $\mc{X}_1$ admits a~realization as noncompact Hermitian symmetric space
\begin{gather}\label{MCEC}\mc{X}_1=\frac{{\rm SL}(2,\R)}{{\rm SO}(2)}\approx
\frac{{\rm SU}(1,1)}{{\rm U}(1)}.\end{gather}
$\mc{X}_1$ is a symmetric, naturally reductive space.
\end{Remark}
\begin{proof}We use the equivalence \eqref{ISO2}, but we look at the level of groups. We consider the case of ${\rm Sp}(n,\R)$. The group ${\rm Sp}(n,\K)$ is the group of matrices $M\in{\rm M}(2n,\K)$, where $\K$ is $\R$ or $\C$, for which
\begin{gather*}
\langle M\alpha,M\beta\rangle =\langle \alpha,\beta\rangle, \qquad \text{where} \qquad \alpha, \beta\in M(n,1,\K), \end{gather*}
and
\begin{gather*}
\langle \alpha,\beta\rangle :=\alpha^t J \beta, \qquad J= \left(\begin{matrix} 0 & \un\\ -\un &
 0 \end{matrix}\right), \end{gather*}
i.e., we have \eqref{SPNR}
\begin{gather}\label{SPNR}
{\rm Sp}(n,\R) = \left\{M = \left(\begin{matrix}a & b\\ c & d \end{matrix}\right) \in {\rm GL}(2n,\R)\,|\, M^tJM = J\right\},\qquad a,b,c,d\in\! M(n,\R).
\end{gather}

We can identify the complex linear group ${\rm GL}(n,\C)$ with the subgroup of matrices of ${\rm GL}(2n,\R)$ that commutes with~$J$, i.e., $a+\ii b\in{\rm GL}(n,\C)$ is identified with the real $2n\times 2n$ matrix~\eqref{ABBA}, see, e.g., \cite[p.~115]{kn}
\begin{gather}\label{ABBA}\left(\begin{matrix}a &b \\ -b & a\end{matrix}\right).\end{gather}

It is easy to prove, see, e.g., \cite{sieg,pedro}, that if $M\in{\rm Sp}(n,\R)$, then $M$ is similar with~$M^t$ and~$M^{-1}$. If $ M\in {\rm Sp}(n,\R)$ is as in \eqref{SPNR}, then the matrices $a,b,c,d\in M(n,\R)$ in \eqref{SPNR} verify the equivalent conditions
\begin{subequations}\label{simplecticR}
\begin{gather}
 ab^t- ba^t = 0,\qquad ad^t-bc^t =\un,\qquad cd^t-dc^t=0,\\
 a^tc-c^ta=0, \qquad a^td-c^tb=\un,\qquad b^td-d^tb=0.
\end{gather}
\end{subequations}
Note that the inverse of the matrix \eqref{SPNR} is given by \begin{gather}\label{MINV}M^{-1}=\left(\begin{matrix} d^t & -b^t\\ -c^t &
 a^t\end{matrix}\right) .\end{gather}
The matrices from ${\rm Sp}(n,\R)$ have the determinant
1.

Using the expression \eqref{MINV} it can be shown that the matrix
\[M\in{\rm Sp}(n,\R)\cap\mathrm{O}_{2n}\] has the expression~\eqref{SPNR} and
\begin{gather}\label{msimor}M=\left(\begin{matrix} a & b \\ -b & a\end{matrix}\right),
\qquad a^ta+ b^tb=\un,\qquad a^tb=b^ta.\end{gather}
If $M\in M(2n,\R)$ has the property~\eqref{msimor}, let
 \begin{gather}\label{MPR}
M':=a+\ii b\in M(n,\C),
\end{gather}
and the correspondence $M\rightarrow M'$ of~\eqref{msimor} with~\eqref{MPR}
 is a group
 isomorphism \[{\rm Sp}(n,\R)\cap\mathrm{O}_{2n}\approx{\rm U}(n).\]

We identify $\R^{2n}$ with $\C^n$ via the correspondence $\alpha =(p,q)\mapsto a$,
\begin{gather*}
a = \frac{p +\ii q}{\sqrt{2}}, \qquad p,q\in\R^n,
\end{gather*}
Following Bargmann \cite{bar70}, it is useful to introduce the transformation
\begin{gather*} \mc{W}\colon \ \R^{2n}\leftrightarrow \C^{2n}, \qquad \alpha= (p,q)\mapsto \alpha_{\C}=(a,\bar{a}),\\
\alpha_{\C}= \mc{W}\alpha, \qquad \mc{W}= 2^{-\frac{1}{2}}\left(\begin{matrix} \un & \ii \un\\ \un & -\ii \un\end{matrix} \right), \qquad \mc{W}^{-1}=
2^{-\frac{1}{2}}\left(\begin{matrix} \un & \un\\ -\ii \un & \ii \un\end{matrix} \right),
\end{gather*}
where
\[\alpha^t= \big(p^t,q^t\big),\qquad p^t=(p_1,\dots,p_n).\]

 To $M\in M(2n,\R)$ as in~\eqref{SPNR} we associate $M_{\C}\in M(2n,\C)$
\begin{gather}\label{cmpl}M_{\C}=\mc{W}M\mc{W}^{-1}=\frac{1}{2}\left(\begin{matrix}
 a+d +\ii (c-b)& a-d+\ii (b+c)\\
a-d- \ii (c+b) & a+d +\ii (b-c)\end{matrix}\right).
\end{gather}
If $\beta = R \alpha$, $R\in{\rm Sp}(n,\R)$ then
\begin{gather*}
\beta_{\C}=R_{\C}\alpha_{\C}, \qquad \text{where} \qquad R_{\C}=\mc{W}R\mc{W}^{-1}\in{\rm Sp}(n,\R)_{\C},
\end{gather*}
and $R_{\C}\in{\rm Sp}(n,\R)_{\C}={\rm Sp}(n,\C)\cap {\rm U}(n,n)$, cf.~\cite{bar70,fol,itzik}.

 From \eqref{simplecticR} we see that
$\SL \approx{\rm Sp}(1,\R)$. Next we apply \eqref{cmpl}
to $M=\left(\begin{smallmatrix}a
 &b\\c&d \end{smallmatrix}\right) \in\SL$, and we
 get $M_{\C}=\left(\begin{smallmatrix}\alpha&\beta\\\bar{\beta}&
 \bar{\alpha}\end{smallmatrix}\right)$,
 where $\alpha=\frac{1}{2}(a+d+\ii (c-b))$, $\beta=\frac{1}{2}(a-d+\ii(b+c))$, i.e., $M_{\C}\in{\rm SU}(1,1)$ because $|\alpha|^2-|\beta|^2=1$. In particular, if $a=\cos\theta$, $b=\sin\theta$, $c=-b$, $d=a$, then $\alpha =\e^{-\ii\theta}$, $\beta=0$, and $\eqref{MCEC}$ is proved.

Now we prove that $\mc{X}_1$ is a naturally reductive space verifying that the condition~\eqref{natred} is fulfilled. We take into account~\eqref{ISO2} and~\eqref{825a}. If we
consider $\got{sl}(2,\R)\ni \g=\got{m}+\got{h}$ for the homogeneous space $M=\mc{X}_1$
as in~\eqref{MCEC}, then $\got{h}=\langle H\rangle$, $\got{m}=\langle F,G\rangle=\langle F-G,F+G\rangle$, and the relation~\eqref{sum3}
$[\got{m},\got{h}]\subset\got{m}$ follows, i.e., $\mc{X}_1$ is reductive.

In order to verify the condition \eqref{natred}, we take $X=aL^1+bL^3$, $Y=a_1L^1+b_1L^3$, $Z=a_2L^1+b_2L^3$. We get $[X,Z]=(a_2b-b_2a)L^2\in\got{h}$ and $[X,Z]_{\got{m}}=0$. \eqref{natred} is trivially satisfied.
\end{proof}

We mention that naturally reductive left-invariant metrics on $\SL$ in the context of BCV-spaces have been investigated in~\cite{HALVA}.

\section[The Jacobi group $G^J_1(\R)$ embedded in ${\rm Sp}(2,\R)$]{The Jacobi group $\boldsymbol{G^J_1(\R)}$ embedded in $\boldsymbol{{\rm Sp}(2,\R)}$}\label{JG1}
\subsection{The composition law}

The real Jacobi group $G^J_1(\R)$ is the semi-direct product of the real three dimensional Heisenberg group $H_1$ with ${\rm SL}(2,\R)$. The Lie algebra of the Jacobi
group $G^J_1(\R)$ is given by $\got{g}^J_1(\R)=\langle P,Q,R,F,G,H\rangle_{\R}$, where the first three generators $P$, $Q$, $R$ of $\got{h}_1$ verify the commutation relations~\eqref{PQT1}, the generators $F$, $G$, $H$ of $\got{sl}(2,\R)$ verify the commutation relations~\eqref{FGHCOM} and the ideal $\got{h}_1$ in~$\got{sp}(2,\R)$ is determined by the non-zero commutation relations
\begin{gather}\label{MORCOM}
[P,F]=Q,\qquad [Q,G]=P,\qquad [P,H]=P,\qquad [H,Q]=Q.
\end{gather}

Let $G^J_1(\R)\ni g= (M,h)$, where $M$ is as in \eqref{ALOS}, while
$H_1\ni h=(X,\kappa)$, $X=(\lambda,\mu)\in\R^2$ and similarly for
$g'=(M',h')$.
 The composition law of $G^J_1(\R)$ is
\begin{gather}\label{COMPL}
gg'=g_1,\qquad
\text{where}\! \qquad M_1=MM',\qquad\! X_1=XM'+X', \qquad\! \kappa_1=\kappa+\kappa'+\left|\begin{matrix}XM'\\X'\end{matrix}\right|,\!\!
\end{gather}
where
\begin{gather*}
g_1 =\left(\begin{matrix}aa'+bc'& ab'+bd'\\ ca'+dc'&
 cb'+dd'\end{matrix}\right),\\
(\lambda_1,\mu_1) = (\lambda'+\lambda a'+\mu c',\mu'+\lambda b'+\mu d'),\qquad \kappa_1 = \kappa+\kappa'+\lambda q'-\mu p'.
\end{gather*}
The inverse element of $g\in G^J_1(\R)$ is given by \begin{gather}\label{INVV}
(M,X,\kappa)^{-1}=\big(M^{-1},-Y,-\kappa\big)\rightarrow~g^{-1}=\left(\begin{matrix}
 d& 0& -b&
 -\mu\\-p
 &1&-q
 &-\kappa\\-c&0&a&\lambda\\ 0&0&0&1\end{matrix}\right),
\end{gather}
where $Y$ was defined in \eqref{DEFY}.

Using the notation of \cite[p.~9]{bs}, the {\it EZ-coordinates}
(EZ~-- from Eichler and Zagier) of an element $g\in G^J_1(\R)$~\eqref{SP2R} are
$(x,y,\theta,\lambda,\mu,\kappa)$, where $M$ is related with
$(x,y,\theta)$ by~\eqref{SCXYT},~\eqref{SCINV}.

The {\it S-coordinates} (S~-- from Siegel) of $g=(M,h)\in G^J_1(\R)$ are $(x,y,\theta,p,q,\kappa)$, where $(x,y,\theta)$ are expressed as function of $M\in \SL$ by~\eqref{SCINV}.

\subsection{The action}
Let
\begin{gather}\label{TAUZ}
\tau:=x+\ii y,\qquad
 z:=p\tau+q=\xi+\ii \eta,\qquad \xi,\eta\in\R.\end{gather}
 Let
$\mc{X}^J_1\approx \mc{X}_1\times\R^2$ be the Siegel--Jacobi upper half-plane,
where $\mc{X}_1=\{\tau\in\C, \,y:=\Im \tau>0\}$ is the Siegel upper half-plane, and
$\tilde{\mc{X}}^J_1=\mc{X}^J_1\times\R$ denotes the extended Siegel--Jacobi upper half-plane. Simultaneously with
the Jacobi group $G^J_1(\R)$ consisting of elements $(M,X,\kappa)$, we consider the group $G^J(\R)_0$ of elements $(M,X)$. It should be mentioned that there is a group homomorphism $G^J_1(\R)\ni(M,X,\kappa)\mapsto (M,X)\in G^J(\R)_0$, through which
the action of $G^J_1(\R)$ on~$\mc{X}^J_1$ can be defined, see \cite[Proposition~2]{nou}. Then
\begin{Lemma}\label{LEMN} The action $G^J(\R)_0\times \mc{X}^J_1\rightarrow \mc{X}^J_1$
is given by \begin{gather}\label{AC1}
(M,X)\times (\tau',z')=(\tau_1,z_1), \qquad \text{where}\qquad \tau_1=\frac{a\tau'+b}{c\tau'+d},\qquad z_1=\frac{z'+\lambda\tau'+\mu}{c\tau'+d},\\
\label{AC11} (M,X)\times (x',y',p',q')=(x_1,y_1,p_1,q_1),
\end{gather}
where $z'=p'\tau'+q'$, $\tau'=x'+\ii y'$ as in \eqref{TAUZ},
\begin{gather}\label{AC12}
(p_1,q_1)=(p,q)+(p',q')
\left(\begin{matrix}a & b\\c & d\end{matrix}\right)^{-1}=(p+dp'-cq',q-bp'+aq'),\end{gather}
while $(x_1,y_1)$ are given by \eqref{ALIGNN1}.

The action $G^J_1(\R)\times \tilde{\mc{X}}^J_1\rightarrow
 \tilde{\mc{X}}^J_1$ is given by
\begin{gather}
(M,X,\kappa)\times (\tau',z',\kappa') =(\tau_1,z_1,\kappa_1),\nonumber\\
(M,X,\kappa)\times (x',y',p',q',\kappa') =(x_1,y_1,p_1,q_1,\kappa_1),\nonumber\\
\kappa_1 =\kappa +\kappa' +\lambda q'-\mu p',\qquad (p',q')= \left(\frac{\eta'}{y'},\xi'-\frac{x'}{y'}\eta'\right).\label{AC2}
\end{gather}
\end{Lemma}
\begin{proof}
The assertion \eqref{AC1} is expressed in \cite[p.~11]{bs}, reproduced in \cite[Remark~9.1]{jac1}. Details of the proof are given in \cite[Remark~1]{BER77}. The calculation of $\kappa_1$ in~\eqref{AC2} is an easy consequence of the composition law~\eqref{COMPL}. The expression of~$(p',q')$ in~\eqref{AC2} is a consequence of~\eqref{TAUZ}.
\end{proof}

\subsection{Fundamental vector fields}

In order to calculate the change of coordinates of a contravariant vector field under the change of variables \eqref{TAUZ} $(x,y,\xi,\eta )\rightarrow
(x,y,p,q)$, where $(p,q)=\big(\frac{\eta}{y},\xi-\frac{\eta}{y}x\big)$, we firstly
observe that the Jacobian is $\frac{\pa(x,y,\xi,\eta)}{\pa(x,y,p,q)}=-y<0$, and we get easily
 \begin{gather}\label{changePQ}\pa_x\rightarrow \pa_x
 -p\pa_q,\qquad \pa_y\rightarrow
 \pa_y-\frac{p}{y}(\pa_p-x\pa_q),\qquad \pa_{\xi}\rightarrow\pa_q,\qquad \pa_\eta\rightarrow \frac{1}{y}(\pa_p-x\pa_q).\end{gather}
In order to calculate the change of coordinates of a contravariant vector field under the change of variables \eqref{TAUZ} $(x,y,p,q,\kappa )\rightarrow
(x,y,\xi,\eta,\kappa)$, we get easily
\begin{gather}\label{changePQ2}\pa_x\rightarrow \pa_x+
 p\pa_{\xi},\qquad \pa_y\rightarrow
 \pa_y+p\pa_{\eta},\qquad \pa_{p}\rightarrow x\pa_{\xi}+y\pa_{\eta},\qquad \pa_q\rightarrow\pa_{\xi}.
\end{gather}
With \eqref{changePQ} and the action \eqref{AC1} on $\mc{X}^J_1$, and then with~\eqref{changePQ2} for the action~\eqref{AC2} on $\hat{\mc{X}}^J_1$, we get
\begin{Proposition}\label{4pr}
The fundamental vector fields expressed in coordinates $(\tau,z)$ of the Siegel--Jacobi upper half-plane $\mc{X}^J_1$ on which act the reduced Jacobi group $G^J(\R)_0$ by \eqref{AC1} are given by the holomorphic vector fields
\begin{subequations}\label{EQQ1}
\begin{gather}
F^* =\pa_{\tau}, \qquad G^*= -\tau^2\pa_{\tau} -z\tau\pa_{z}, \qquad H^*=2\tau\pa_{\tau}+z\pa_z,\\
P^* =\tau\pa_z,\qquad Q^*=\pa_z, \qquad R^* =0.
\end{gather}
\end{subequations}

Then the real holomorphic fundamental vector fields corresponding to $\tau=x+\ii y$, $y>0 $, $z=\xi+\ii \eta$ in the variables $(x,y,\xi,\eta)$ are
\begin{subequations}\label{EQQ2}
\begin{gather}
F^* =F^*_1, \qquad G^*= G^*_1 +(\eta y-\xi x)\pa_{\xi}-(\xi y+ x\eta)\pa{_\eta},\\
H^* = H^*_1 +\xi\pa_{\xi}+\eta\pa_{\eta}, \qquad P^* =x\pa_{\xi}+y\pa_{\eta},\qquad Q^*=\pa_{\xi}, \qquad R^* =0 ,
\end{gather}
\end{subequations}
where $F^*_1$, $G^*_1$, $H^*_1$ are the fundamental vector fields~\eqref{FUNDFGH} attached to the generators $F$, $G$, $H$ of $\got{sl}(2,R)$ corresponding to the
action \eqref{AC1} of $\rm{SL(2},\R)$ on $\mc{X}_1$.

If we express the fundamental vector fields in the variables $(x,y,p,q)$ where $\xi=px+q$, $\eta= p y$, we find
 \begin{subequations}\label{EQQ3}
\begin{gather}
F^* =F^*_1 -p\pa_q, \qquad G^*= G^*_1-q\pa_p, \qquad H^*=H^*_1-p\pa_p+q\pa_q,\label{EQQ31}\\
P^* =\pa_p,\qquad Q^*=\pa_{q}, \qquad R^* =0. \label{EQQ32}
\end{gather}
\end{subequations}
Now we consider the action \eqref{AC2} of $G^J_1(\R)$ on the points $(\tau,z,\kappa)$ of $\tilde{\mc{X}}^J_1$.

Instead of \eqref{EQQ1}, we get the fundamental vector fields in the variables $(\tau,z,p,q,\kappa)$
\begin{gather*}
F^* =\pa_{\tau}, \qquad G^*= -\tau^2\pa_{\tau} -z\tau\pa_{z},\qquad H^*=2\tau\pa_{\tau}+z\pa_z,\\
P^* =\tau\pa_z+q\pa_{\kappa},\qquad Q^*=\pa_z-p\pa_{\kappa}, \qquad R^* =\pa_{\kappa},\qquad p=\frac{\Im (z)}{\Im (\tau)}, \qquad q=\frac{\Im(\bar{z}\tau)}{\Im(\tau)}.
\end{gather*}
Instead of \eqref{EQQ2}, we get the fundamental vector fields in $\tilde{\mc{X}}^J_1$ in the variables $(x ,y ,\xi ,\eta ,\kappa )$
\begin{gather*}
F^* =F^*_1, \qquad G^*= G^*_1 +(\eta y-\xi x)\pa_{\xi}-(\xi y+ x\eta)\pa{_\eta},\\
H^* = H^*_1 +\xi\pa_{\xi}+\eta\pa_{\eta}, \qquad P^* =x\pa_{\xi}+y\pa_{\eta}+q\pa_{\kappa},\qquad Q^*=\pa_{\xi} -p\pa_{\kappa}, \qquad R^* =\pa_{\kappa} .
\end{gather*}
Instead of \eqref{EQQ3}, we ge the fundamental vector fields in the variables $(x,y,p,q,\kappa)$
\begin{subequations}\label{NEWPQR3}
\begin{gather}
F^* =F^*_1 -p\pa_q, \qquad G^*= G^*_1-q\pa_p, \qquad H^*=H^*_1-p\pa_p+q\pa_q,\label{EQQ31x}\\
P^* =\pa_p+q\pa_k,\qquad Q^*=\pa_q-p\pa_k,\qquad R^*=\pa_{\kappa}.
\end{gather}
\end{subequations}
\end{Proposition}

\subsection{Invariant metrics}\label{S54}
We explain the method to get invariant metrics on $G$-homoge\-neous manifolds $M$ from invariant metrics of $G$, see also Appendix~\ref{S911}.

Let $M=G/H$ be a reductive homogeneous space. If $\g$ $(\h)$ is the Lie algebra of~$G$ (respectively,~$H$), then there exists a vector space
$\m$ such that we have the vector space decomposition $\g=\m+\h$, $\m\cap\h=\varnothing$, and the tangent space at $x$, $T_xM$, can be identified with $\m$, where
$H=G_x$ is the isotropy group at~$x$, see Definition~\ref{DEF1} and Lemma~\ref{LMN}. Then let $X_i$, $i=1,\dots,n$ be a~basis of the Lie algebra
$\g$ such that
\[\m=\langle X_1,\dots,X_m\rangle ,\qquad \h=\langle X_{m+1},\dots,X_n\rangle ,\]
where $\dim \m=m$. The left-invariant one forms $\lambda_i$ on $G$ are given by
\[g^{-1}{\rm d} g=\sum_{i=1}^n\lambda_iX_i,\] and the left-invariant vector
fields $L^i$ on $G$ are determined from the relations $\langle \lambda_i\,|\,L^j\rangle =\delta_{ij}$, $i,j=1,\dots,n$. Then the invariant metric on $G$ is given by ${\rm d} s^2_G=\sum_{i=1}^n\lambda^2_i$ and $g_G(L^i,L^j)=\delta_{i,j}$, where $i,j=1,\dots,n$. Let now $L^j_0$ be the projections on $M$ of the vector fields $L^j$, $j=1,\dots,m$. We stil have $\langle \lambda_i\,|\,L^j_0\rangle =\delta_{ij}$, $i,j=1,\dots,m$ and ${\rm d} s^2_M=\sum_{i=1}^m\lambda_i^2$. The fundamental vector fields~$X^*_i$, $i=1,\dots,m$ are Killing vectors of the metric $g_M$.

Now we calculate the left-invariant one-forms on $G^J_1(\R)$
\begin{gather*}
 g^{-1}{\rm d} g=\lambda^FF+\lambda^GG+\lambda^HH+\lambda^PP+\lambda^QQ+\lambda^RR,\end{gather*}
where $g$ is as in \eqref{SP2R} and $g^{-1}$ as in \eqref{INVV}. We find the left-invariant one-forms on $G^J_1(\R)$
\begin{subequations}\label{BIGH}
\begin{gather}
\lambda^F =\lambda^f,\qquad \lambda^G=\lambda^g,\qquad \lambda^H=\lambda^h,\\
\lambda^P= {\rm d}\lambda-p{\rm d} a-q{\rm d} c=c{\rm d} q+ a{\rm d}
 p=\lambda^p-\lambda\lambda^h-\mu\lambda^g\\
\hphantom{\lambda^P}{}= -y^{-\frac{1}{2}}\sin \theta {\rm d} q +\big(y^{\frac{1}{2}}\cos\theta
 -xy^{-\frac{1}{2}}\sin\theta \big){\rm d} p,\label{LDP}\\
\lambda^Q = d {\rm d} q +b{\rm d} p=\lambda^q-p{\rm d} b-q{\rm d} d= \lambda^q-\lambda\lambda^f+\mu\lambda^h\\
\hphantom{\lambda^Q }{} = y^{-\frac{1}{2}}\cos\theta {\rm d} q
 +\big(y^{\frac{1}{2}}\sin\theta+xy^{-\frac{1}{2}}\cos\theta\big){\rm d} p,\label{LDQ}\\
 \lambda^R = {\rm d} \kappa -p{\rm d} q +q{\rm d} p=\lambda^r+\lambda^2\lambda^f-\mu^2\lambda^g-2\lambda\mu\lambda^h.\label{LDK}
\end{gather}
\end{subequations}
In \eqref{BIGH}, equations \eqref{LDP}, \eqref{LDQ}, \eqref{LDK} are expressed in the S-coordinates $(x,y,\theta, p,q,\kappa)$, $\lambda^p$, $\lambda^q$, $\lambda^r$ have the expression~\eqref{LEFT1} in the $(\lambda,\mu,\kappa)$-coordinates, while $\lambda^f$, $\lambda^g$, $\lambda^h$ have the expressions \eqref{LEFTRIGHT} in the $(a,b,c,d)$-coordinates of~$\SL$ and~\eqref{LEFTRIGHT2} are expressed in the $(x,y,\theta)$-coordinates. Also the elements $a$, $b$, $c$, $d$ of the matrix $M\in \SL $ \eqref{SP2R} are expressed in the $(x,y,\theta)$-coordinates by~\eqref{SCXYT}.

Now we calculate the left-invariant vector fields for the real Jacobi group
\begin{Proposition}\label{PrLI}
 The left-invariant vector fields $L^{\alpha}$ for the
real Jacobi group $G^J_1(\R)$ orthogonal with respect to the
invariant one-forms $\lambda^{\beta}$,
 \[\langle \lambda^{\beta}\,|\,L^{\alpha}\rangle =\delta_{\alpha,\beta},
\qquad \alpha,\beta= F, G, H, P, Q, R,\] are given by the equations
\begin{subequations}\label{LFLf}
\begin{gather}
L^F = L^f,\qquad L^G= L^g,\qquad L^H= L^h, \qquad L^P=L^P_0+L^P_+,\qquad L^Q=L^Q_0+L^Q_+,\label{LFLGHHH}\\
L^P_0 = d\frac{\pa}{\pa p}-b\frac{\pa}{\pa q}=\frac{\cos\theta}{y^{\frac{1}{2}}}\frac{\pa}{\pa p}-\frac{x\cos \theta + y \sin
 \theta}{y^{\frac{1}{2}}}\frac{\pa}{\pa q} ,\label{LFLf1}\\
L^P_+ = -(pb+qd)\frac{\pa}{\pa \kappa}=-\frac{1}{y^{1/2}}[p(x\cos \theta+ y\sin\theta)+ q\cos\theta]\frac{\pa}{\pa \kappa},\label{LPpl}\\
L^Q_0 = -c\frac{\pa}{\pa p}+ a \frac{\pa}{\pa q}=\frac{\sin\theta}{y^{\frac{1}{2}}}\frac{\pa}{\pa p}+ \frac{y\cos\theta
 -x\sin\theta}{y^{\frac{1}{2}}}\frac{\pa}{\pa q} ,\label{LFLf2}\\
L^Q_+ =(pa+qc)\frac{\pa}{\pa \kappa}
 =\frac{1}{y^{1/2}}[p(y\cos\theta-x\sin\theta)-q\sin\theta]\frac{\pa}{\pa \kappa},\label{LFQpl}\\
 L^R =\frac{\pa}{\pa \kappa}. \label{LFLf3}
\end{gather}
\end{subequations}
The invariant vector fields $L^F$, $L^G$, $L^H$, $L^P$, $L^Q$, $L^R$ verify the commutations relations \eqref{FGHCOM}, \eqref{PQT1} and \eqref{MORCOM} of the generators $F$, $G$, $H$, $P$, $Q$, $R$ of the Lie algebra $\got{g}^J_1(\R)$.
\end{Proposition}

Besides the formulas for $\lambda_1$, $\lambda_2$, $\lambda_3$ defined in~\eqref{1l2l3l}, we introduce the left-invariant one-forms:
 \begin{gather}\label{4lf5lf6lf}
\lambda_4= \sqrt{\gamma}\lambda^P,~\lambda_5=\sqrt{\gamma}\lambda^Q,\qquad \lambda_6=\sqrt{\delta}\lambda^R,\qquad \gamma,\delta>0,
\end{gather}
where $\lambda^P$, $\lambda^Q$, $\lambda^R$ are defined in~\eqref{BIGH}. Note that the parameters $\gamma$, $\delta$ introduced in~\eqref{4lf5lf6lf} will appear also in the invariant metrics~\eqref{linvG} on $\tilde{\mc{X}}^J_1$ and~\eqref{MTRTOT} on~$G^J_1(\R)$, while in the metric \eqref{begG} on $\mc{X}^J_1$ appears only~$\gamma$. The invariant metrics on~$\mc{D}^J_n$ and $\mc{X}^J_n$ depend only of two parameters $\alpha, \gamma>0$ \cite{sbj,nou,SB15,Yan,Y08}. In fact, the first time they appear in the papers of K\"ahler \cite{cal3,cal} and Berndt \cite{bern84,bern}, parameterizing the invariant metric on~$\mc{X}^J_1$.

Also, besides the left-invariant vector fields $L^1$, $L^2$, $L^3$ defined in~\eqref{1L2L3L}, we introduce the left invariant one forms
\begin{gather}\label{4L5L6L}
L^4=\frac{1}{\sqrt{\gamma}}L^P,\qquad L^5=\frac{1}{\sqrt{\gamma}}L^Q,\qquad L^6=\frac{1}{\sqrt{\delta}}L^R,
\end{gather}
where $L^P$, $L^Q$, $L^R$ are defined in \eqref{LFLf}. The vector fields $L^i$, $i=1,\dots,6$ verify the commutations relations
\begin{subequations}\label{FUFUFU}
\begin{alignat}{4}
&\big[L^1,L^2\big] =-\frac{\sqrt{\beta}}{\alpha}L^3, \qquad && \big[L^2,L^3\big]=\frac{1}{2\sqrt{\beta}}L^1, \qquad && \big[L^3,L^1\big]=\frac{1}{\sqrt{\beta}}L^2,&\\
&\big[L^1,L^4\big] = -\frac{1}{2\sqrt{\alpha}}L^5, \qquad && \big[L^1,L^5\big]=-\frac{1}{2\sqrt{\alpha}}L^4, \qquad && \big[L^1,L^6\big]= 0, &\\
&\big[L^2,L^4\big]= -\frac{1}{2\sqrt{\alpha}}L^4, \qquad && \big[L^2,L^5\big] =\frac{1}{2\sqrt{\alpha}}L^5,\qquad && \big[L^2,L^6\big]=0,& \\
&\big[L^3,L^4\big]= - \frac{1}{2\sqrt{\alpha}}L^5, \qquad && \big[L^3,L^5\big]=\frac{1}{2\sqrt{\beta}}L^4, \qquad && \big[L^3,L^6\big]= 0,& \\
&\big[L^4,L^5\big] = \frac{2\sqrt{\delta}}{\gamma}L^6, \qquad && \big[L^4,L^6\big] = 0,\qquad && \big[L^5,L^6\big]= 0.&
\end{alignat}
\end{subequations}

Similarly, we introduce
\begin{gather}\label{4L5L6L0}
L^4_0=\frac{1}{\sqrt{\gamma}}L^P_0,\qquad L^5_0=\frac{1}{\sqrt{\gamma}}L^Q_0, \qquad L^6_0=\frac{1}{\sqrt{\delta}}L^R.
\end{gather}

Recalling also Proposition~\ref{PRFC}, where we have replaced $u=pv+q$, $v=x+\ii y$ with $z=p\tau +q $, respectively $\tau= x+\ii y$, and $k=2c_1$, $\mu=\frac{c_2}{2}$, we have proved:
\begin{Proposition}\label{Pr4} The balanced metric \eqref{METRS} on the Siegel--Jacobi upper half-plane $\mc{X}^J_1$, left-invariant to the action \eqref{AC1}, \eqref{AC11}, \eqref{AC12} of reduced group $G^J(\R)_0$ is
\begin{subequations}\label{METRS}
\begin{gather}{\rm d} s_{\mc{X}^J_1}^2(\tau,z) =-c_1\frac{{\rm d} \tau{\rm d} \bar{\tau}}{(\tau-\bar{\tau})^2}+\frac{2\ii c_2}{\tau-\bar{\tau}}({\rm d}
z-p{\rm d}\tau)\times cc, \qquad p= \frac{z-\bar{z}}{\tau-\bar{\tau}},\label{METRS1}\\
{\rm d} s^2_{\mc{X}^J_1}(x,y,p,q) =
c_1\frac{{\rm d} x^2 + {\rm d} y^2}{4y^2} +\frac{c_2}{y}\big[\big(x^2+y^2\big){\rm d} p^2+{\rm d} q^2+2x{\rm d} p{\rm d} q\big]\nonumber\\
\hphantom{{\rm d} s^2_{\mc{X}^J_1}(x,y,p,q)}{} =c_1\frac{{\rm d} x^2+{\rm d} y^2}{4y^2} + c_2\frac{x^2 + y^2}{y}\left[ \left({\rm d} p + \frac{x}{x^2 + y^2}{\rm d}
 q\right)^2 + \left(\frac{y{\rm d} q}{x^2 + y^2}\right)^2 \right],\!\!\!\label{METRS2}\\
{\rm d} s_{\mc{X}^J_1}^2(x,y,\xi,\eta) = c_1\frac{{\rm d} x^2+{\rm d} y^2}{4y^2} \nonumber\\
\hphantom{{\rm d} s_{\mc{X}^J_1}^2(x,y,\xi,\eta) =}{} +\frac{c_2}{y}\left[{\rm d} \xi^2+{\rm d} \eta^2
 +\left(\frac{\xi}{y}\right)^2({\rm d} x^2+{\rm d} y^2)-2\frac{\eta}{y}({\rm d} x{\rm d} \xi+{\rm d} y{\rm d} \eta)\right]. \label{METRS3}
\end{gather}
\end{subequations}
The metric \eqref{METRS} is K\"ahler.

If we denote $\frac{c_1}{4}=\alpha$, $c_2=\gamma$, then the matrix attached to the left invariant metric~\eqref{METRS2} on~$\mc{X}^J_1$ reads
\begin{gather}
g_{\mc{X}^J_1} = \left(\begin{matrix}g_{xx} &0 &0 &0\\
0& g_{yy}& 0& 0\\
0& 0& g_{pp} & g_{pq} \\0 & 0& g_{pq}& g_{qq}
\end{matrix}\right) ,\nonumber\\
g_{xx} = g_{yy} = \frac{\alpha}{y^2},\qquad g_{pp} = \gamma\frac{x^2+y^2}{y},\qquad g_{qq} = \frac{\gamma}{y},\qquad g_{pq}=\gamma\frac{x}{y}.\label{begG}
\end{gather}
The metric \eqref{METRS2} can be written as
\[{\rm d} s^2_{\mc{X}^J_1}=\lambda_1^2+\lambda_2^2 +\lambda_4^2+\lambda_5^2.\]
The vector fields $L^j_0$ dual orthogonal to the invariant one-forms $\lambda_i$, $\langle \lambda_i\,|\,L^j_0\rangle=\delta_{ij}$, $i,j=1,2,4,5$, with
respect to the bases ${\rm d} x$, ${\rm d} y$, ${\rm d} p$, ${\rm d} q$ and $\frac{\pa}{\pa x}$, $\frac{\pa} {\pa y}$, $\frac{\pa}{\pa p}$, $\frac{\pa}{\pa q}$ are
\begin{gather*}
L^1_0=\frac{y}{\sqrt{\alpha}}\left(\cos 2\theta\frac{\pa }{\pa x}+\sin
2\theta \frac{\pa}{\pa y}\right),\qquad L^2_0=\frac{y}{\sqrt{\alpha}}\left(-\sin2\theta\frac{\pa}{\pa x}+\cos
2\theta\frac{\pa}{\pa y}\right),
\end{gather*}
and $L^4_0$, $L^5_0$ defined by \eqref{4L5L6L0}. The metric \eqref{METRS2} is orthonormal with respect to the vector fields $L^1_0$, $L^2_0$, $L^4_0$, $L^5_0$.

The fundamental vector fields given by \eqref{EQQ3} are the solutions of the equations of the Killing vector fields \eqref{SIASTA} on $\mc{X}^J_1$ in the variables $(x,y,p,q)$ corresponding to the metric~\eqref{METRS2}, invariant to the action \eqref{AC1}, made explicit in \eqref{ALIGNN} and \eqref{AC12}, \eqref{AC2}:
\begin{subequations}\label{SIASTA}
\begin{gather}
 -X^2+y\pa_xX^1=0,\\
 \pa_xX^2+\pa_yX^1=0,\\
 c_2\big[\big(x^2+y^2\big)\pa_xX^3+x\pa_xX^4\big]+\frac{c_1}{4y}\pa_pX^1=0,\\
 \frac{c_1}{4y}\pa_qX^1 +c_2\big(x\pa_xX^3+\pa_xX^4\big) =0,\\
 -X^2+y\pa_yX^2=0,\\
 c_2\left[\big(x^2+y^2\big)\pa_yX^3+\frac{x}{y}\pa_yX^4\right]+\frac{c_1}{y}\pa_pX^2=0,\\
 c_2\big[x\pa_yX^3+\pa_yX^4\big]+\frac{c_1}{4y}\pa_qX^2=0,\\
2xyX^1+\big({-}x^2+y^2\big)X^2+2y\big(x^2+y^2\big)\pa_pX^3+2xy\pa_pX^4=0,\\
yX^1-xX^2+xy\pa_pX^3+y\pa_pX^4+y\big(x^2+y^2\big)\pa_qX^3+xy\pa_qX^4=0,\\
-X^2+2xy \pa_qX^3+2y\pa_qX^4=0.
\end{gather}
\end{subequations}
\end{Proposition}

We make a ``historical'' comment
\begin{Comment}\label{CM1}
In \cite[p.~8]{bern84}, Berndt considered the closed two-form $\Omega={\rm d} \bar{{\rm d}} f$ on Siegel--Jacobi upper half-plane $\mc{X}^J_1$, $G^J(\R)_0$-invariant to the action~\eqref{AC1}, obtained from the K\"ahler potential
\begin{gather}\label{POT}
f(\tau,z)= c_1\log (\tau-\bar{\tau}) -\ii
c_2\frac{(z-\bar{z})^2}{\tau-\bar{\tau}}, \qquad c_1,c_2>0,\end{gather}
where $c_1=\frac{k}{2}$, $c_2=2\mu$ comparatively to our formula \eqref{BFR}. Formula \eqref{POT} is presented by Berndt as ``communicated to the author by K\"ahler'', where it is also given equation \eqref{METRS1}, while \eqref{METRS2} has two printing errors. Later, in Section~36 of his last paper~\cite{cal3}, reproduced also in~\cite{cal}, K\"ahler argues how to choose the potential as in~\eqref{POT};
see also \cite[Section~37, equation~(9)]{cal3}, where $c_1=\frac{\lambda}{2}$, $c_2=\ii\mu\pi$, and the metric~(8) differs from the metric~\eqref{METRS} by a~factor two, because the hermitian metric used by K\"ahler is ${\rm d} s^2=2g_{i\bar{j}}{\rm d} z_i{\rm d}\bar{z}_j$.

We also recall that in~\cite{Y07} Yang calculated the metric on $\mc{X}^J_n$, invariant to the action of~$G^J_n(\R)_0$. The equivalence of the metric of Yang with the metric obtained via CS on $\mc{D}^J_n$ and then transported to $\mc{X}^J_n$ via partial Cayley transform is underlined in~\cite{nou}. In particular, the metric~\eqref{METRS3}
 appears in \cite[p.~99]{Y07} for the particular values $c_1=1$, $c_2=4$. See also \cite{Yan,Y08,Y10}.
\end{Comment}

Now we shall establish a metric invariant to the action given in Lemma~\ref{LEMN} of $G^J_1(\R)$ on the extended Siegel--Jacobi upper half-plane $\tilde{\mc{X}}^J_1$. Because the manifold $\tilde{\mc{X}}^J_1$ is 5-dimensional, we want to see if the extended Siegel--Jacobi upper half-plane is a Sasaki manifold, as in the case of $\SL$ in Proposition~\ref{X15}. If we take as contact form $\eta=\lambda_6$, then ${\rm d} \eta =-2\sqrt{\delta}{\rm d} p\wedge{\rm d} q$, and $\eta(\wedge\eta)^2=0$. If we try to determine a contact distribution $\mc{D}=\operatorname{Ann} (\eta)$, we get $\mc{D}=\big\langle \frac{\pa}{\pa p}-q\frac{\pa}{\pa \theta},\frac{\pa}{\pa q}+p\frac{\pa}{\pa\theta}\big\rangle$. From~\eqref{679}, we find $\Phi^{\kappa}_{\lambda}=0$, and $\Phi^{\kappa}_p=q\Phi^{\lambda}_{\kappa}$, $\Phi^{\lambda}_q=-p\Phi^{\lambda}_{\kappa}$, where $\lambda =x,y,p,q,\kappa$. So $\Phi$ has $\operatorname{Rank}(\Phi)<4$. In conclusion, $(\Phi,\xi,\eta)$ chosen as above can not be an almost contact structure for the extended Siegel--Jacobi upper half-plane $\tilde{\mc{X}}^J_1$.

We obtain
\begin{Proposition}\label{Pr5}The metric on the extended Siegel--Jacobi upper half-plane $\tilde{\mc{X}}^J_1$, in the partial S-coordinates $(x,y,p,q,\kappa)$:
\begin{gather}{\rm d} s^2_{\tilde{\mc{X}}^J_1} ={\rm d}
 s^2_{\mc{X}^J_1}(x,y,p,q)+\lambda^2_6(p,q,\kappa)\nonumber\\
\hphantom{{\rm d} s^2_{\tilde{\mc{X}}^J_1}}{} =\frac{\alpha}{y^2}\big({\rm d} x^2+{\rm d}
 y^2\big)+\left[\frac{\gamma}{y}\big(x^2+y^2\big)+\delta q^2\right]{\rm d} p^2+
 \left(\frac{\gamma}{y}+\delta p^2\right){\rm d} q^2 +\delta {\rm d} \kappa^2\nonumber\\
\hphantom{{\rm d} s^2_{\tilde{\mc{X}}^J_1}=}{} + 2\left(\gamma\frac{x}{y}-\delta pq\right){\rm d} p{\rm d} q +2\delta (q{\rm d} p{\rm d}
 \kappa-p{\rm d} q {\rm d} \kappa)\label{linvG}
 \end{gather}
 is left-invariant with respect to the action given in Lemma~{\rm \ref{LEMN}} of the Jacobi group $G^J_1(\R)$.

The matrix attached to metric \eqref{linvG} is
\begin{gather}
g_{\tilde{\mc{X}}^J_1} = \left(\begin{matrix}g_{xx} &0 &0 &0&0\\
0& g_{yy}& 0& 0 & 0\\
0& 0& g'_{pp} & g'_{pq} & g'_{p\kappa}\\0 & 0& g'_{qp}& g'_{qq}
 &g'_{q\kappa}\\
0& 0& g'_{\kappa p}& g'_{\kappa q} & g'_{\kappa\kappa}
\end{matrix}\right),\nonumber\\
g'_{pq} = g_{pq}-\delta p q , \qquad g'_{p\kappa}=\delta q, \qquad g'_{q\kappa}=-\delta p, \qquad
g'_{pp} = g_{pp}+\delta q^2,\nonumber\\
g'_{qq} = g_{qq}+\delta p^2, \qquad g'_{\kappa\kappa} = \delta,\label{begGG}
\end{gather}
while $g_{xx}$, $g_{yy}$, $g_{pp}$, $g_{qq}$, $g_{pq}$ are given in the metric matrix \eqref{begG} associated with the balanced metric~\eqref{METRS2} on~$\mc{X}^J_1$.

The metric \eqref{begGG} is orthonormal with respect to the invariant vector fields $L^i_0$, $i=1,2$, $L^i$, $i=4,5,6$. The fundamental vector fields with respect to the action~\eqref{AC2} in the variables $(x,y,p,q,\kappa)$ are given by \eqref{NEWPQR3}.

The extended Siegel--Jacobi upper half-plane $\tilde{\mc{X}}^J_1$ does not admit an almost contact structure $(\Phi,\xi,\eta)$ with a contact form $\eta=\lambda_6$ and Reeb vector $\xi= \operatorname{Ker}(\eta)$.
\end{Proposition}

With \eqref{LDP}, \eqref{LDQ}, \eqref{LDK}, we find the invariant metric on the Jacobi group $G^J_1(\R)$:

\begin{Theorem}\label{BIGTH}The composition law for the real Jacobi group $G^J_1(\R)$ in the S-coordinates $(x,y,\theta,p,q,\kappa)$ is given by~\eqref{ALIGNN} for $(x,y,\theta)$, \eqref{AC11} for $(p,q)$ and \eqref{AC2} for the coordinate $\kappa$, replacing in the matrix $M\in\SL$ the values of $a$, $b$, $c$, $d$ as function of
 $(x,y,\theta)$ given by~\eqref{SCXYT}.

The left-invariant metric on the real Jacobi group $G^J_1(\R)$ in the S-coordinates $(x,y,\theta,$ $p,q,\kappa)$ is
\begin{gather}
{\rm d}s^2_{G^J_1(\R)} =\sum_{i=1}^6\lambda_i^2 =\alpha\frac{{\rm d} x^2+{\rm d} y^2}{y^2} +\beta\left(\frac{{\rm d} x}{y}+2{\rm d}\theta\right)^2 \nonumber\\
\hphantom{{\rm d}s^2_{G^J_1(\R)} =}{} + \frac{\gamma}{y}\big[{\rm d} q^2+\big(x^2+y^2\big){\rm d} p^2+2x{\rm d} p{\rm d}q\big]+\delta({\rm d} \kappa-p{\rm d} q+q{\rm d} p)^2,\label{MTRTOT}
\end{gather}
where $\lambda_1,\dots,\lambda_3$ are defined by \eqref{1l2l3l}, while $\lambda_4,\dots,\lambda_6$ are defined by \eqref{4lf5lf6lf}, \eqref{LDP}--\eqref{LDK}.
The matrix attached to the metric \eqref{MTRTOT} in the variables $(x,y,\theta,p,q,\kappa)$ reads
\[g_{G^J_1}=\left(\begin{matrix} g_{xx}& 0 &g_{x\theta} &0 &0 &0\\
0 & g_{yy} & 0 & 0 & 0 & 0\\
g_{\theta x} & 0&g_{\theta\theta}& 0 & 0 & 0\\
0 & 0 & 0 & g'_{pp} & g'_{pq}& g'_{p\kappa}\\
0 & 0 & 0 & g'_{qp} & g'_{qq}& g'_{q\kappa}\\
 0& 0& 0& g'_{\kappa p} & g'_{\kappa q} & g'_{\kappa\kappa}
\end{matrix}\right),\]
where $g_{xx}$, $g_{yy}$, $g_{x\theta}$, $g_{\theta\theta}$ are those attached to $\SL$ given by~\eqref{GGG}, $g'_{pp}$, $g'_{pq}$, $g'_{qq}$, $g'_{p\kappa}$, $g'_{q\kappa}$, $g'_{\kappa\kappa}$ are given by~\eqref{begGG}.

We have \begin{gather*}
\langle \lambda_i\,|\,L^j\rangle =\delta_{ij}, \qquad i,j=1,\dots ,6,\end{gather*}
 where $L^1,\dots, L^3$ are defined by~\eqref{1L2L3L}, while $L^4,\dots, L^6$ are defined by \eqref{4L5L6L}, \eqref{LFLf1}--\eqref{LFLf3}. The vector fields $L^i$, $i=1,\dots,6$ verify the commutations relations \eqref{FUFUFU} and are orthonormal with respect to the metric \eqref{MTRTOT}.

Depending of the values of the parametres $\alpha$, $\beta$, $\gamma$, $\delta$, we have invariant metric on the following manifolds:
\begin{enumerate}\itemsep=0pt
\item[$1)$] the Siegel upper half-plane $\mc{X}_1$ if $\beta,\gamma,\delta =0$, see Proposition~{\rm \ref{X15}},
\item[$2)$] the group $\SL$ if $\gamma,\delta=0$, $\beta\not= 0$, see Proposition~{\rm \ref{X15}},
\item[$3)$] the Siegel--Jacobi half-plane $\mc{X}^J_1$ if $\beta, \delta= 0$, see Proposition~{\rm \ref{Pr4}},
\item[$4)$] the extended Siegel--Jacobi half-plane $\tilde{\mc{X}}^J_1$ if $\beta=0$, see Proposition~{\rm \ref{Pr5}},
\item[$5)$] the Jacobi group $G^J_1$ if $\alpha\beta\gamma\delta\not= 0$.
\end{enumerate}
\end{Theorem}

We show some consequences of Theorem \ref{BIGTH}. We investigate if the homogeneous manifold~$\mc{X}^J_1$ is a naturally reductive manifold or not. The fact that $\mc{X}^J_1$ is not a~naturally reductive 4-dimensional manifold is well known, see Theorem~\ref{thm2}, but in Proposition~\ref{PRLST} below we present a~direct proof.

\begin{Proposition}\label{PRLST}
The Siegel--Jacobi upper half-plane realized as homogenous Riemannian mani\-fold $\big(\mc{X}^J_1=\frac{G^J_1(\R)}{{\rm SO}(2)\times\R}, g_{\mc{X}^J_1}\big)$ is a reductive, non-symmetric manifold, not naturally reductive with respect to the balanced metric~\eqref{METRS2}.

The Siegel--Jacobi upper half-plane $\mc{X}^J_1$ is not a g.o.\ manifold with respect to the balanced metric.

When expressed in the variables that appear in the FC-transform given in Proposition~{\rm \ref{PRFC}}, $\mc{X}^J_1$ is a naturally reductive space with the metric $g_{\mc{X}_1}\times g_{\R^2}$, where $g_{\mc{X}_1}$ is given by~\eqref{MBER} and~$g_{\R^2}$ is the Euclidean metric~\eqref{MTR}.

If
 \begin{gather}\label{XLP}
\got{g}^J_1\ni X=aL^1+bL^2+cL^3+dL^4+eL^5+fL^6,
\end{gather}
then a geodesic vector of the homogeneous manifold $\mc{X}^J_1$ has one of the following expressions given in Table~{\rm \ref{table1}}.

\begin{table}[t]\centering
\caption{Components of the geodesic vector \eqref{XLP}.}\label{table1}\vspace{1mm}
\begin{tabular}{||c|c|c|c|c|c|c||}
\hline Nr. cr. & {a} & {b} &
{c} & {d} & {e} & {f} \\
\hline 1& $0$ & $0$ & $c$ & $0$ & $0$ & $f$ \\
\hline 2& $a$ & $b$ & $0$ & $0$ & $0$ & $f$ \\
\hline 3& $rc$ & $0$ & $c$ & $\pm rc$& $0$ & $f$ \\
\hline 4& $a$ & $0$ & $-a$ & $0$ &$\epsilon\sqrt{r}a$ & $f$ \\
\hline 5& $\epsilon_1\epsilon_2 \frac{1-r}{\sqrt{r}}e$ & $\epsilon_1 e$
 & $-\frac{\epsilon_1\epsilon_2}{\sqrt{r}}e$
 & $\epsilon_2\sqrt{r}e$ & $e$ & $f$\\
\hline
\end{tabular}

Here $r=\sqrt{\frac{\alpha}{\beta}}$, $\epsilon_1^2=\epsilon_2^2=\epsilon^2=1$.
\end{table}
\end{Proposition}

\begin{proof}From the commutation relations \eqref{FGHCOM}, \eqref{PQT1}, \eqref{MORCOM}, it is seen that for $\got{g}^J_1(\R)$, we have{\samepage
\begin{gather}\label{MHM}\m=\langle F,G,P,Q\rangle, \qquad \h=\langle H,R\rangle,\end{gather} because $[\m,\h]\subset\m$, i.e., $\mc{X}^J_1$
is a reductive space, cf. Definition \ref{DEF1}.}

But $[\m,\m]\nsubseteq \h$, and $\mc{X}^J_1$ is not a symmetric manifold.

We verify~\eqref{natred} written as
\begin{gather}\label{X2X3}g ([X_1,X_3]_{\m},X_2)+g(X_1,[X_3,X_2]_{\m})=0,\qquad \forall\, X_1,X_2,X_3\in\m.\end{gather}
Instead of \eqref{MHM} we take
\[\m=\langle L^1,L^2,L^4,L^5\rangle, \qquad \h=\langle L^3,L^6\rangle ,\]
where $L^1,\dots,L^3$ ($L^4,\dots,L^6$) are defined in \eqref{1L2L3L},
(respectively~\eqref{4L5L6L}).

We take
\[X_i=a_iL^1+b_iL^2+c_iL^4+d_iL^5, \qquad i=1,2,3, \] and, with the
commutation relations \eqref{FUFUFU}, we find
\[[X_1,X_3]=\frac{1}{2\sqrt{\alpha}}\big[(d_1a_3-a_1d_3+c_1b_3-b_1c_3)L^4+ (c_1a_3-a_1c_3+b_1d_3-d_1b_3)L^5\big].\]
Taking into account that the vector fields $L^1,\dots,L^6$ are orthonormal with respect to the metric~\eqref{MTRTOT} on~$G^J_1$ as in Theorem~\ref{BIGTH}, the
condition~\eqref{X2X3} of the geodesic Lemma~\ref{PRR} reads
\begin{gather}\label{CRRT}c_3(a_2d_1-a_1d_2+b_2c_1-b_1c_2)+d_3(b_1d_2-b_2d_1+a_2c_1-a_1c_2)=0.\end{gather}
The condition \eqref{CRRT} implies that the system of algebraic equations
\begin{gather*}
a_1d_2+b_1c_2 = c_1b_2+d_1a_2,\\
a_1c_2-b_1d_2 = c_1a_2-d_1b_2,
\end{gather*}
must have a solution for any $a_i$, $b_i$, $c_i$, $d_i$, $i=1,2,3$, which is not possible, and $\mc{X}^J_1$ is not naturally reductive with respect to the balanced metric.

Due to Theorem \ref{BTHM}, the four-dimensional manifold $\mc{X}^J_1$ is not a g.o.\ manifold.

We also recall that in \cite[Propositions~3 and~4]{nou} it was proved that under the so called FC-transform, the manifold $\mc{X}^J_n$ is symplectomorph
with $\mc{X}^n\times\C^n$. The particular case of the Jacobi group of degree~1 was reproduced in Proposition~\ref{PRFC} and, in particular, $\mc{X}^J_1$ is equivalent with the symmetric space $\mc{X}_1\times \C$, which is naturally reductive, as in Theorem~\ref{thm2}.

To find the geodesic vectors on the Siegel--Jacobi upper half-plane $\mc{X}^J_1$, we look for the solution~\eqref{XLP} that verifies the condition~\eqref{BCOND} of the geodesic lemma expressed in Proposition~\ref{PRR}. Taking
\[\m\ni Y= a_1L^1+b_1L^2+d_1L^4+e_1L^5,\]
the condition \eqref{BCOND}
\begin{gather*}
a_1\left(\frac{bc}{\sqrt{\beta}}+\frac{ed}{\sqrt{\alpha}}\right)+\frac{b_1}{2}\left[-\frac{ac}{\sqrt{\beta}}+\frac{1}{\sqrt{\alpha}}\big(d^2-e^2\big)\right]\\
\qquad {}-\frac{d_1}{2\sqrt{\alpha}}(bd+ec+ae) +\frac{e_1}{2}\left[\frac{cd}{\sqrt{\beta}}+\frac{1}{\sqrt{\alpha}}(be-ad)\right]=0
\end{gather*}
must be satisfied for every values of $a_1$, $b_1$, $d_1$, $e_1$, i.e., the coefficients of the geodesic vector~\eqref{XLP} are solutions of the system of algebraic equations
\begin{gather}
 rbc+de=0,\nonumber\\
 -rac+ d^2-e^2=0,\nonumber\\
 bd+e(a+c)=0,\nonumber\\
rcd+be-ad=0.\label{KLJH}
\end{gather}
The solutions of the system \eqref{KLJH} are written in Table~\ref{table1}.
\end{proof}

\appendix

\section{Naturally reductive spaces}\label{ISO}

\subsection{Fundamental vector fields}
A {\it homogeneous space} is a manifold $M$ with a transitive action of a Lie group $G$. Equivalently, it is a manifold of the form $G/H$, where $G$ is a Lie group and $H$ is a closed subgroup of $G$, cf., e.g., \cite[p.~67]{aa}.

Let $(M,g)$, $(N,g')$ be Riemannian manifolds. An {\it isometry} is a~diffeomorphism $f\colon M\rightarrow N$ that preserves the metric, i.e., $g_p(u,v)=g'_{f(p)}({\rm d} f_p(u),{\rm d} f_p(v))$, $\forall\, p\in M$, $\forall\, u,v\in TM_p$. If $(M,g)$ is a~Riemannian manifold, the set $I(M,g)$ (or $I(M)$) of all isometries $M\rightarrow M$ forms a group called {\it the isometry group} of~$M$.

A {\it Riemannian homogenous space} is a Riemannian manifold $(M,g)$ on which the isometry group $I(M)$ acts transitively. A~Riemannian manifold $(M,g)$ is a $G$-{\it homogenous} (or {\it homogenous under a~Lie group}~$G$) if $G$ is a closed subgroup of $I(M,g)$ which acts transitively on $M$, cf.\ \cite[p.~178]{bes}.

Let $G$ be a Lie group of transformations acting on the manifold $M$, cf.\ \cite[Chapter~II, Section~3, p.~121]{helg}. In \cite[p.~122]{helg} it is introduced the notion of vector field on~$M$ {\it induced by the one parameter subgroup} $\exp tX$, $t\in\R$, $X\in\g$, denoted $X^+$, where $\g$ is the Lie algebra of~$G$.
In \cite[Section~5, p.~51]{kn1}, in the context of principal fibre bundle $P(M,G)$ over $M$ with structure group the Lie group~$G$, it is introduced the same notion under the name {\it fundamental vector field associate to} $X\in\g$, denoted $X^*$, see also \cite[Proposition~4.1, p.~42]{kn1}.

Let $M=G/H$ be a homogeneous $n$-dimensional manifold and let us suppose that $G$ acts transitively on the {\it left} on $M$, $G\times M\rightarrow M\colon g\cdot
x=y$, where $y=(y_1,\dots,y_n)^t$. Then $g(t)\cdot x =y(t)$, where $g(t)=\exp(tX)$, $t\in\R$, generates a curve in $M$ with $y(0)=x$ and $\dot{y}(0)=X$. The fundamental vector field attached to $X\in\got{g}$ at $x\in M$ is defined as
\begin{gather*}
X^*_x:=\frac{{\rm d}}{{\rm d} t}y(t)\Big|_{t=0}=\frac{{\rm d}}{{\rm d} t}(\exp(tX)\cdot x)\Big|_{t=0}.
\end{gather*}
We write the fundamental vector field attached to $X\in\g$ as
\begin{gather*}
X^*_x=\sum_{i=1}^n(X^*_i)_x\frac{\pa }{\pa z_i}, \qquad (X^*_i)_x=\frac{{\rm d} y_i(t)}{{\rm d} t}\Big|_{t=0}. \end{gather*}

Now, because $[X^*,Y^*]=-[X,Y]^*$, see, e.g., in \cite[Theorem~3.4, p.~122]{helg}, it is observed
\begin{Lemma}\label{MIC} If the generators $X_1,\dots,X_n$ of a Lie algebra $\got{g}$ verify the commutations relations
\begin{gather}\label{XIXJ}
[X_i,X_j]=c^k_{ij}X_k,\end{gather} then the
associated fundamental vector fields verify the commutation relations
\begin{gather*}
[X_i^*,X_j^*]=-c^k_{ij}X_k^*.\end{gather*}
\end{Lemma}

Note that if the action of $G$ on $M$ is on the right as in \cite[p.~51]{kn1}, then
\begin{gather*}
[X_i^*,X_j^*]=c^k_{ij}X_k^*.\end{gather*}

\subsection{Killing vectors}\label{KLVF}
A vector field $X\in \got{D}^1(M)$ on a Riemannian manifold $(M,g)$ is called an {\it infinitesimal isometry} or a {\it Killing vector field} if the local 1-parameter group of local transformations by $X$ in a~neighbourhood of each point of $M$ consists of local isometries, see also in \cite[Proposition~3.2, p.~237]{kn1}, i.e.,
 \begin{gather}\label{LX}L_Xg=0,\qquad X\in\got{D}^1(M),\end{gather} where $L_X$ is the Lie derivative on $M$.

We recall below in Lemma \ref{YANO} the {\it{Killing equations}} \eqref{kill}, see, e.g., \cite[Theorem 1.3, p.~5]{YK} or \cite[equation~(40$'$), p.~247]{vr}. We use the tensor notation as in \cite{vr,YK}.

Let us consider a $n$-dimensional Riemannian manifold $(M,g)$ and a~vector field with the {\it contravariant} components $X^i$, $i=1,\dots,n$:
\begin{gather}\label{SUMXX}X=\sum_{i=1}^nX^i\frac{\pa }{\pa x^i}.\end{gather}
If $\nabla$ denotes the covariant derivative, we have the standard formulas
\begin{subequations}\label{PAL}
\begin{gather}
\nabla_{\mu}g_{\lambda\chi} :=\pa_{\mu}g_{\lambda\chi}-\Gamma^{\rho}_{\mu\lambda}g_{\rho\chi}-\Gamma^{\rho}_{\mu\chi}g_{\lambda\rho},\label{PAL1}\\
\nabla_{\mu}X^{\chi} :=\pa_{\mu}X^{\chi}+\Gamma^{\chi}_{\mu\lambda}X^{\lambda},\label{PAL2}\\
\nabla_{\mu}X_{\chi}:=\pa_{\mu}X_{\chi}-\Gamma^{\lambda}_{\mu\chi}X_{\lambda},\label{PAL3}\\
X_{\mu} : =g_{\mu\lambda}X^{\lambda\label{PAL4}},\\
L_Xg_{\lambda\chi}: =X^{\mu}\pa_{\mu}g_{\lambda\chi}+g_{\rho\chi}\pa_{\lambda}X^{\rho} +g_{\lambda\rho}\pa_{\chi}X^{\rho}.\label{LG}
\end{gather}
\end{subequations}
\begin{Lemma}\label{YANO} Let $(M,g)$ be a $n$-dimensional Riemannian manifold with a Riemannian $($metric$)$ connection. The field~$X$ is a~Killing vector field if and only if its covariant components $X_{\mu}$, $\mu =1,\dots,n$ verify the Killing equations
\begin{gather}\label{kill}
\nabla_{\lambda}X_{\chi}+\nabla_{\chi}X_{\lambda} := \frac{\pa X_{\chi}}{\pa x^{\lambda}}+\frac{\pa X_{\lambda}}{\pa x^{\chi}}-2\Gamma^{\rho}_{\lambda\chi}X_{\rho}=0, \qquad \lambda, \chi =1,\dots, n. \end{gather}
\end{Lemma}

\begin{Remark}\label{RMM1}If the coordinate $x^{\alpha}$ is not present in the expression of metric tensor $g_{\lambda\chi}$, $\lambda,\chi=1,\dots,n$, then $\frac{\pa}{\pa x^{\alpha}}$ is a Killing vector field for the metric $g_{\lambda\chi}$.
\end{Remark}

With \eqref{LG}, the condition \eqref{LX} of a vector field~\eqref{SUMXX} to be a~Killing vector field is that its {\it contravariant} components to verify the equations
\begin{gather}\label{LG1}X^{\mu}\pa_{\mu}g_{\lambda\chi}+g_{\mu\chi}\pa_{\lambda}X^{\mu}+g_{\lambda\mu}\pa_{\chi}X^{\mu}=0, \qquad \lambda,\chi,\mu =1,\dots, \dim{M}=n.\end{gather}
The system \eqref{LG1} of $n(n+1)/2$ equations of a Killing vector field $X^1(x),\dots,X^n(x)$ is overdetermined, and {\it no-nonvanishing solution is guaranteed, in general}. The set $\iota(M)$ of all Killing vector fields on $n$-dimensional manifold $M$ forms a Lie algebra of dimension not exceeding $\frac{n(n+1)}{2}$ and $\dim (\iota(M))=\frac{n(n+1)}{2}$ is obtained only for spaces of constant curvature, see \cite[Theorem~3.3, p.~238]{kn1}. For example, maximal solution is
obtained for the (pseudo)-Euclidean spaces $E^{r,n-r}$, for the sphere~$S^n$ or the real projective space $RP^n=S^n/(\pm I)$, see, e.g., \cite[Theorem~1, p.~308]{kn1}, \cite[p.~251]{vr} and \cite[Section~4.6.6, p.~83]{mari}. We have $\iota(E^{r,s})=\got{so}(r,s)\ltimes \R^{r+s}$. The Euclidean group~$E^n$ of~$\R^n$ has dimension~$n(n+1)/2$, where $n$ degrees of freedom correspond to translations, the other $n(n-1)/2$ correspond to rotations, see also Proposition~\ref{P9} and Remark~\ref{Rem11} below.

The following remark is very important for the determination of Killing vector fields on Riemannian homogeneous manifolds, see, e.g., see
 \cite[p.~4]{btv} or \cite[Proposition~2.2, p.~139]{koda}:
\begin{Remark}\label{RemF} If $(M,g)$ is a Riemannian homogeneous space $M=G/H$ endowed with a~$G$-invariant Riemannian metric $g$, then each $X\in\got{g}$ generates a
 one-parameter subgroup of the group $I(M)$ of isometries ({\it motions}) of~$M$ via $p\rightarrow (\exp tX)\cdot p$. Hence {\it the fundamental vector field~$X^*$ on a~Riemannian homogeneous manifold is a~Killing vector}. {\it For Riemannian homogeneous spaces} $M=G/H$, $\dim (I(M))=\dim (G)$.
\end{Remark}

\subsection{Reductive homogeneous spaces}

The following notions are standard, see \cite[pp.~121, 123 and 125]{helg} or \cite[p. 155]{kn1} and \cite[p.~187]{kn}; see also~\cite{ale,En,Oni}.

The set of elements $G_x$ of a given group $G$, acting on a set~$M$ as group of transformations that leaves the point~$x$ fixed, is called {\it isotropy group}, also called {\it stationary group} or {\it stabilizer}. If~$G$ is a Lie group and~$H$ is a closed subgroup, then the coset space $G/H$, in particular, $H=G_x$ is taken with the analytic structure given in \cite[Theorem~4.2, p.~123]{helg}. For $x\in G$, the diffeomorphism of $G/H$ into itself is $\tau(x)\colon yH\rightarrow xyH$. The natural representation of the isotropy group of a differentiable transformation group in the tangent space to the underling manifold is called {\it isotropy representation}. If $G$ is the group of differentiable transformations on the manifold~$M$ and~$G_x$ is the corresponding isotropy subgroup at the point $x\in M$, then the isotropy representation ${\rm Is}_x\colon
G_x\to{\rm GL} (T_xM)$ associates to each $h\in G_x$ the differential ${\rm Is}_x(h):=({\rm d}\tau (h)) _{\lambda (H)}$ of the transformation $h$ at~$x$, where $\lambda\colon G\rightarrow G/H$ is the canonical projection. The image of the isotropy representation, ${\rm Is}_x(G_x)$, is called the {\em linear isotropy group} at~$x$.

If $G$ is a Lie group with a countable base acting transitively and smoothly on $M$, then {\em the tangent space $T_xM$ can be naturally identified with the space $\g/\g_x$}, where $\g\supset \g_x$ are respectively the Lie algebras of the groups $G\supset G_x$. The isotropy representation ${\mbox{\rm{Is}}}_x$ is now identified with the representation $G_x\to {\mbox{\rm{ GL}}}(\g/\g_x)$, induced by the restriction of the adjoint representation ${\rm Ad}_G$ of $G$ to $G_x$. See details below in Lemma~\ref{LMN}.

\begin{deff}[cf.\ Nomizu \cite{nomizu}]\label{DEF1} A homogeneous space $M=G/H$ is {\em reductive} if the Lie algebra $\g$ of $G$ may be decomposed into a vector space direct sum of the Lie algebra $\h$ of $H$ and an $\Ad (H)$-invariant subspace $\m$, that is
\begin{subequations}
\begin{gather}\label{sum1}
\g = \h + \m, \qquad \h\cap\m =\varnothing,\\
\label{sum2} \Ad(H)\m \subset \m.
\end{gather}
Condition (\ref{sum2}) implies
\begin{gather}\label{sum3}
[\h ,\m ]\subset \m
\end{gather}
\end{subequations}
and, conversely, if $H$ is connected, then (\ref{sum3}) implies~(\ref{sum2}). Note that $H$ is always connected if~$M$ is simply connected. The decomposition~(\ref{sum1}) verifying (\ref{sum2}) is called a~{\em $H$-stable decomposition}.
\end{deff}

\begin{Lemma}\label{LMN} If a homogeneous space $M$ is reductive, then $T_xM$ can be identified with $\m$, while ${\mbox{\rm{Is}}}_x$ can be identified with the representation $h\mapsto (\Ad_G h)|_{\m}$. In this case, the isotropy representation is faithful if~$G$ acts effectively.
\end{Lemma}

So let us denote by $x(s)$ the $1$-parameter subgroup of $G$ generated by $X\in\m$ and let $x^*(s)=\lambda (x(s))$ be the image of $x(s)$ by the projection $\lambda$ of $G$ onto $G/H$:
\[x^*(s)=x(s).o=(\exp sX^*).o.\]
Identifying $X^*$ with $X\in\m$, we can write down
\begin{gather*}
x^*(s)=\exp(sX).o.
\end{gather*}
The invariant tensor fields on a homogeneous space $M$ are in one-to-one correspondence with the tensor fields on $T_xM$ that are invariant with respect to the isotropy representation. In particular, $M$ has an invariant Riemannian metric if and only if $T_xM$ has a~Euclidean metric that is invariant under the linear isotropy group.

In accord with \cite[Proposition 3.1 and Corollary 3.2, p.~200]{kn} and \cite[p.~78]{aa}:
\begin{Proposition}\label{KIKS}
Let $M=G/H$ be a homogenous space where $G$ is a~Lie group acting effectively on~$M$, which is reductive.

The one-to-one correspondence between $G$-invariant indefinite Riemannian metrics~$g$ on $M=G/H$ and $\Ad (H)$-invariant non-degenerate symmetric bilinear forms $B$ on~$\m$
\begin{gather*}
B(X,Y)=B\big({\rm Ad}^{G/H}(h)X, \Ad^{G/H}(h)Y\big),\qquad \forall\, X,Y\in\m, \ h\in H,
\end{gather*}
is given by
\begin{gather}\label{BXY}
B(X,Y)=g(X^*,Y^*)_o, \qquad \text{for} \quad X,Y\in\m,
\end{gather}
\end{Proposition}
Explicitly, the $\Ad(H)$ invariance of the symmetric non-degenerate form $B$ in \eqref{BXY} means, see, e.g., \cite[p.~201]{kn}
\begin{gather}\label{INVBXY}
B(X,[Z,Y])+B([Z,X],Y)=0,\qquad X,Y\in\got{m},\qquad Z\in\got{h}.
\end{gather}
Usually it is asked that the group of isometries $G$ acts {\em effectively} on~$M$, cf.~\cite{atri}.

The {\em canonical connection}, see \cite[p.~192]{kn}, or {\it canonical affine connection of second type}, see~\cite{nomizu}, on the reductive space $M=G/H$ verifying (\ref{sum1}), (\ref{sum2}), is the unique $G$-invariant affine connection on~$M$ such that for any vector field $X\in\m$ and any frame $u$ at the point $o$, the curve $(\exp tX)u$ in the principal fibration of frames over $M$ is horizontal. The canonical connection is complete and the set of its geodesics through $o$ coincides with the set of curves of the type $(\exp tX)o$, where $X\in\m$, see also \cite[Proposition~2.4 and Corollary 2.5, p.~192]{kn}. In a reductive space there is a unique $G$-invariant affine connection with zero torsion having the same geodesics as the canonical connection, cf.\ \cite[Theorem~2.1, p.~197]{kn}. This connection is called in~\cite{kn} {\em natural torsion-free connection} on $M=G/H$ relative to the decomposition~(\ref{sum1}), or {\it canonical affine connection of the first kind} in~\cite{nomizu}.

\subsection{Naturally reductive spaces}
For the next definition and \eqref{natred} below, see in \cite[Chapter~II, Section~13, {\it metric connections}]{nomizu}, \cite[p.~202]{kn} and~\cite{atri},
\begin{deff}\label{DEF2} A homogeneous Riemannian or pseudo-Riemannian space $M=G/H$ is {\it naturally reductive} if it is reductive, i.e., it verifies (\ref{sum1}), (\ref{sum2}), and
\begin{gather}\label{natred}
B(X,[Z,Y]_{\m})+B([Z,X]_{\m},Y)=0,\qquad X,Y, Z\in\m ,
\end{gather}
where $B$ is the non-degenerate symmetric bilinear form on $\m$ induced by the Riemannian (pseudo-Riemannian) structure on $M$ under the natural identification of the spaces $\m$ and $M_o$, as in~\eqref{BXY}.
\end{deff}
If $M=G/H$ is a naturally reductive Riemannian or pseudo-Riemannian space verifying (\ref{sum1}), (\ref{sum2}), and (\ref{natred}), then the natural torsion-free connection coincides with the corresponding Riemannian or pseudo-Riemannian connection on $M$ \cite{ale}.

Based on \cite[Theorem~5.4]{as}, \cite[Chapter~X, Section~3]{kn}, \cite[Theorem~6.2, p.~58]{tv} and \cite[Proposition~1, p.~5]{btv}, it is formulated the following
\begin{Proposition}\label{PR5}
Let $(M,g)$ be a homogeneous Riemannian manifold. Then $(M,g)$ is a naturally reductive Riemannian homogenous space if and only if there exists a connected Lie subgroup~$G$ of $I(M)$ acting transitively and effectively on~$M$ and a reductive decomposition~\eqref{sum1}, such that one of the following equivalent statements hold:
\begin{enumerate}\itemsep=0pt
\item[$(i)$] \eqref{natred}, or
\begin{gather*}
g([X,Z]_{\m},Y)+g(X,[Z,Y]_{\m})=0\qquad \forall\, X,Y,Z \in \m,
\end{gather*} is verified;
\item[$(ii)$] the Levi-Civita connection of $(M,g)$ and the natural torsion-free connection with respect to the decomposition~\eqref{sum1} are the same;
\item[$(iii)$] $(*)$ is true, i.e., every geodesic in $M$ is the orbit of a one-parameter subgroup of $I(M)$ generated by some $X\in\m$.
\end{enumerate}
\end{Proposition}

It is not always easy to decide whether a given homogenous Riemannian space is naturally reductive~\cite{ilka}. The Riemannian manifold $M=G/H$ might be naturally reductive although for any reductive decomposition $\got{g}=\got{h}+\got{m}$ none of the statements in Proposition~\ref{PR5} holds, because that might exist another appropriate subgroup
$\tilde{G}\subset I(M)$ such that $M=\tilde{G}/\tilde{H}$ and with respect to such decomposition the conditions of Proposition~\ref{PR5} are satisfied, see, e.g., \cite[p.~5]{btv}. In accord with \cite[Proposition~2, p.~5]{btv}, {\it a necessary and sufficient condition that a~complete and simply connected manifold be naturally
reductive is that there exists a homogeneous structure~$T$ on~$M$ with $T_*v=0$, for all tangent vectors~$v$ of~$M$}.

Ambrose and Singer found the condition for a Riemannian manifold be locally homogeneous~\cite{as}.

\subsection[Naturally reductive spaces of dimension $\le 4$]{Naturally reductive spaces of dimension $\boldsymbol{\le 4}$}

The connected homogeneous Riemannian $V_n$ naturally reductive spaces of dimension $n\le 6$ are classified.

 For two dimensional manifolds, {\it because the homogeneous manifolds $V_2$ have constant curvature, they are locally symmetric spaces}, see, e.g., in
\cite[Theorem~4.1, Section~4]{tv}.
\begin{Theorem}\label{THM0} The only homogenous structure on $\R^2$ and $S^2$ is given by $T=0$, cf.\ {\rm \cite[Corollary~4.2]{tv}}.

Let $(M,g)$ be a connected and simply connected surface. Then $(M,g)$ admits a~homogenous structure $T\not= 0$ if and only if $(M,g)$ is isomorphic to the hyperbolic plane, cf.\ {\rm \cite[Theorem~4.3]{btv}}.

Up to an isomorphism, $\mc{H}^2$ has only two homogenous structures, namely:
\begin{enumerate}\itemsep=0pt
\item[$1.$] $T=0$, corresponding to the symmetric case $\mc{H}^2={\rm SO}_0(1,2)/{\rm U}(1)$, where ${\rm SO}_0(1,2)=\SL\!/\!{\pm}I$ is the connected component of the identity of the Lorentz group, see also~\eqref{ISO2}.

\item[$2.$] $T_XY=g(X,Y)\xi-g(\xi,Y)X$, $\xi=\xi^1E_1+\xi^2E_2$, $E_1= \frac{1}{r }y^1\frac{\pa}{\pa y^1}$, $E_2= \frac{1}{r }y^1\frac{\pa}{\pa y^2}$,
 \begin{gather}\label{GBELTR}
g=r^2\big(y^1\big)^{-2}\big(\big({\rm d} y^1\big)^2+\big({\rm d} y^2\big)^2\big),\end{gather}
$X,Y\in\got{D}^1(M)$, $r>0$. This homogenous structure corresponds to the Lie algebra $\got{g}$ with the product $(y_1,y_2)(y'_1,y'_2)=(y_1y'_1,y_1y'_2+y_2)$, i.e., the semi-direct product of the multiplicative group $\R_0^+$ and the additive group~$\R$.
\end{enumerate}
\end{Theorem}
The case $n=3$ was considered by Kowalski~\cite{kw3}. The proof of Theorems~\ref{thm1} and~\ref{thm2} below is based on the Ambrose and Singer theorem in the formulation of \cite[Section~2]{tv} and the classification of 3-dimensional unimodular Lie groups with left-invariant metrics of Milnor~\cite{ml}.

The following theorem is \cite[Theorem~6.5, p.~63]{tv}, \cite[Theorem 2]{bv} or \cite[Theorem~5.2]{ilka}:
\begin{Theorem}\label{thm1} A three-dimensional complete, simply connected naturally reductive
Riemannian manifold $(M, g)$ is either:
\begin{enumerate}\itemsep=0pt
\item[$(a)$] a symmetric space realized by the real forms: $ \R^3$, ${S}^3$ or the Poincar\'e half-space $\mc{H}^3$, and $S^2\times \R$, $\mc{H}^2\times\R$, or
\item[$(b)$] a non-symmetric space isometric to one of the following Lie groups with a suitable left-invariant metric:
\begin{enumerate}\itemsep=0pt
\item[$(b1)$] ${\rm SU} (2)$,
\item[$(b2)$] $\widetilde{{\rm SL}}(2,\R)$, the universal covering of $\SL$, with a special left-invariant metric,
\item[$(b3)$] the $3$-dimensional Heisenberg group $H_1$, where the Heisenberg group has a~left-inva\-riant metric.
\end{enumerate}
\end{enumerate}
\end{Theorem}
The Poincar\'e half-space $\mc{H}^n$ is the set $(x_1,\dots,x_n)\in\R^n$, $x_1>0$, with the metric propor\-tio\-nal with
\begin{gather}\label{poinc}
{\rm d} s^2:=x_1^{-2}\sum_{i=1}^n({\rm d} x_i)^2 .
\end{gather}
For the left invariant Riemannian metrics which appear in Theorem~\ref{thm1}, see \cite[Theorem 2]{kw3} and~\cite{vp}. For $H_1=\R^3[x,y,z]$ a
left-invariant metric is
\begin{gather*}
{\rm d} s^2_{H_3}=\frac{1}{b}\big({\rm d} x^2+ {\rm d} z^2+({\rm d} y-x{\rm d} z)^2\big), \qquad b\in \R_+.
\end{gather*}
Note that in \cite[Theorem 1, p.~6]{btv} appear only the non-symmetric naturally reductive spaces of dimensions~3: ${\rm SU}(2)\cong S^3$, $\widetilde{{\rm SL}}(2,\R)$ and ${\rm Nil}_3$. The metrics of these spaces are particular cases of the 7-families of BCV-spaces that appear in Theorem~\ref{PR15}, because the naturally reductive spaces are a particular class of homogenous spaces.

The case of four-dimensional manifolds was treated by Kowalski and Vanhecke, see \cite[Theo\-rem~1, p.~224]{kw4} or \cite[Theorem~2, p.~6]{btv}:
\begin{Theorem}\label{thm2} Let $(M,g)$ be a four-dimensional simply connected naturally reductive Riemannian manifold. Then $(M,g)$ is either symmetric or it is a~Riemannian product of the naturally reductive spaces of dimension~$3$ of type $(b)$ appearing in Theorem~{\rm \ref{thm1}} times~$\R$. In the last cases, $(M,g)$ is not locally symmetric.
\end{Theorem}

\subsection[$V_2$ and $V_3$ spaces with transitive group]{$\boldsymbol{V_2}$ and $\boldsymbol{V_3}$ spaces with transitive group}

The determination of the groups $G_3$ of isometries with three parameters of a two-dimensional space $V_2$ with positive definite metric was done by Bianchi~\cite{BIANCHI}. In Proposition~\ref{P9} below we follow Vranceanu, see \cite[Chapter V, Section~14, p.~288]{vr}. The generators of~$G_3$ in \cite[equation~(90)]{vr} considered by Vranceanu, in our notation \eqref{KILPL}, verifies the commutation relations
\begin{gather}\label{EC90V}
[X,Y]=-\epsilon Z, \qquad [Y,Z]=-k X,\qquad [Z,X]=-Y, \qquad \epsilon=\pm 1.
\end{gather}
Below we also write down $V_2$ as a homogenous manifolds.

${\rm E}(2)$ is the group of rigid motions of the Euclidean 2-space, denoted $M(2)$ in \cite[p.~195]{vil}, see also \cite[Section~8.5]{vil}.
\begin{Proposition}\label{P9}\quad
\begin{enumerate}\itemsep=0pt
\item[$1.$] If $k=0$, $\epsilon=1$, then the invariant metric of $V_2$ is given by \eqref{MTR},
 \begin{gather}\label{MTR}
 {\rm d} s^2= {\rm d} x^2+{\rm d} y^2,\end{gather}
and $V_2$ is the Euclidean space $E^2=E(2)/O(2)$. The Euclidean group is $E(2)=\R^2\rtimes O(2)$.
\item[$2.$] If $k=0$, $\epsilon=-1$, then the invariant metric of $V_2$ is
\[{\rm d} s^2= {\rm d} x^2-{\rm d} y^2,\]
on the pseudo-euclidean space $V_2=E^{1,1}=E(1,1)/O(1,1)$, $E(1,1)= \R^2\rtimes O(1,1)$.

\item[$3.$] A space $V_2$ with group $G_3$ always admits a simply transitively subgroup, except when the generators~\eqref{EC90V} of the structure group for $\epsilon= 1$ and $k>0$, when the stereographic projection of the sphere from the south pole $(0,0,-R)$ to plane tangent in the north pole $(0,0,R)$ has the expression \eqref{RIEMM}, where $m=\frac{k}{4}$, and the generators \eqref{X1X2} are
\begin{subequations}
\begin{gather*}
X = \sqrt{k}\left(z\frac{\pa }{\pa x}-x\frac{\pa }{\pa z}\right),\qquad
Y = \sqrt{k}\left(-y\frac{\pa }{\pa z}+z\frac{\pa }{\pa y}\right),\qquad
Z = y\frac{\pa }{\pa x}-x\frac{\pa }{\pa y},
\end{gather*}
\end{subequations}
i.e., rotations around the axes $x$, $y$, $z$. We have $V_2=S^2\cong \db{CP}^1\cong {\rm SU}(2)/{\rm U}(1)$. The metric of a space $V_2$ with simply transitive abelian group may be written as
\begin{gather}\label{SNUS}{\rm d} s^2=\e^{2\lambda v}{\rm d} u^2+ \epsilon {\rm d} v^2,\qquad k=-\lambda^2.\end{gather}
If $\epsilon=1$ and $k>0$, then \eqref{SNUS} can be written down as
\[{\rm d} s^2={\rm d} u^2 +{\rm d} v^2+\frac{k(u{\rm d} v+v{\rm d} u)^2}{1-k\big(u^2+v^2\big)},\qquad k=\frac{1}{R^2}.\]
\item[$4.$] If $\epsilon =1$ and $k<0$, then the metric on $V_2$ is $($the Beltrami$)$ metric
\begin{gather}\label{GBETR1}{\rm d} s^2=\frac{{\rm d} \xi^2 +{\rm d}\eta^2}{\lambda^2\eta^2}, \qquad k=-\lambda^2,\qquad \eta>0,\end{gather}
see also \eqref{poinc}. $V_2$ is of the type of a Siegel disk $V_2=\mc{D}_1 \equiv{\rm SU} (1,1)/{\rm U}(1)$ or, equivalently, Siegel upper half-plane $\mc{H}_1$.
\end{enumerate}
\end{Proposition}
For \eqref{GBETR1}, see \cite[equation~(2)]{sieg} or \cite[Theorem~3, p.~644]{hua}.

The formulation of the following proposition is extracted from \cite{ino}:
\begin{Proposition}\label{PR10} If $(V_3,g)$ is a homogenous space of dimension $3$, then $\dim (I(V_3))= 6$, $4$ or $3$.
\begin{enumerate}\itemsep=0pt
\item[$1.$] If $ \dim (I(V_3))=6$, then $V_3$ is of the type of the real space forms, i.e., the real Euclidean space~$E^3$, the sphere $S^3(\kappa)$, or the hyperbolic space
$\mc{H}^3(\kappa)$.
\item[$2.$] If $\dim (I(V_3))=4$, then $V_3$ is either a~Riemannian product $\mc{H}^2(\kappa)\times \R$ or $S^2(\kappa)\times\R$, or one of the following Lie groups with left invariant metric: ${\rm SU}(2)$, $\widetilde{{\rm SL}}(2,\R)$ or $H_1$, see~{\rm \cite{ml}}.
\item[$3.$] If $\dim (I(V_3))=3$, then $V_3$ is a general $3$-dimensional Lie group with left-invariant metric, e.g., the Lie group ${\rm Sol}_3$, i.e., the group with the composition law:
\[(x_1,y_1,t_1)(x_2,y_2,t_2)=\big(x_1+\e^tx_2,y_1+\e^{-t}y_2,t_1+t_2\big)\] and
the left-invariant metric \[{\rm d} s^2=\e^{-2t}{\rm d} x^2 + \e^{2t}{\rm d} y^2+ {\rm d} t^2.\]
\end{enumerate}
\end{Proposition}
The above classification contains the eight model geometries of Thurston~\cite{th}: $E^3$, $\mc{H}^3$, $S^3$, $\mc{H}^2\times \R$, $S^2\times \R$, $\tilde{{\rm SL}}(2,\R)$, $H_1$ and ${\rm Sol}_3$.

Cartan classified all 3-dimensional spaces $V_3$ with a 4-dimensional isometry group $G_4$ in~\cite{cart}, see also~\cite{BIANCHI} and~\cite{vr}. See also~\cite{vp} for a modern presentation of Cartan approach.

The Bianchi--Cartan--Vranceanu (BCV) spaces are $V_3$ spaces with $\dim (I(V_3))=4$ together with $E^3$ and $S^3(\kappa)$, while the hyperbolic space $\mc{H}^3(\kappa)$ appearing in Theorem~\ref{thm1}~-- a~symmetric naturally reductive~-- is missing in the list of BCV-spaces.

For $\kappa,\tau\in\R$, it is defined the open subset of $\R^3$
\begin{gather}\label{M3}
{\rm BCV}(\kappa,\tau): = \left\{ (x,y,z) \in\R^3\,|\,D = D(x,y,z;\kappa)>0,
 \ \text{where} \ D:=1+\frac{\kappa}{4}\big(x^2+y^2\big)\right\},\!\!\!\end{gather}
equipped with the metric
\begin{gather}\label{BCV}
{\rm d} s^2_{\rm BCV}(x,y,z;\kappa,\tau)= \frac{{\rm d} x^2 +{\rm d} y^2}{D^2}+\left({\rm d} z +\tau\frac{y{\rm d} x-x{\rm d} y}{D}\right)^2.\end{gather}

Following \cite[Section~2.5]{JVAN} and \cite[Example~2.1.10, p.~59]{calin}, the BCV spaces are described as in
\begin{Theorem}\label{PR15}
All $3$-dimensional homogenous spaces $V_3$ with isometry group $G_4$ are locally isomorphic with the BCV-spaces. The BCV family also includes
two real space forms, with isometry group~$G_6$, see Proposition~{\rm \ref{PR10}}. The full classification of these spaces is as follows:
\begin{enumerate}\itemsep=0pt
\item[$1)$] if $\kappa=\tau=0$, then ${\rm BCV}(\kappa,\tau)\cong {E}^3 $;
\item[$2)$] if $\kappa=4\tau\not=0$, then ${\rm BCV}(\kappa,\tau)\cong S^3\big(\frac{\kappa}{4}\big)\setminus \{\infty\}$;
\item[$3)$] if $\kappa>0$ and $\tau=0$, then ${\rm BCV}(\kappa,\tau)\cong {S}^2(\kappa)\setminus\{\infty\})\times \R$;
\item[$4)$] if $\kappa<0$ and $\tau=0$, then ${\rm BCV}(\kappa,\tau)\cong \mc{H}^2(\kappa)\times \R$;
\item[$5)$] if $\kappa>0$ and $\tau \not= 0$, then ${\rm BCV}(\kappa,\tau)\cong {\rm SU}(2)\setminus \{\infty\}$;
\item[$6)$] if $\kappa<0$ and $\tau \not= 0$, then ${\rm BCV}(\kappa,\tau)\cong \widetilde{{\rm SL}}(2,\R)$;
\item[$7)$] if $\kappa = 0$ and $\tau\not= 0$, then ${\rm BCV}(\kappa,\tau)\cong {\rm Nil}_3$.
\end{enumerate}
Here the Poincar\'e $($Siegel$)$ disc is
\[\mc{H}^2(\kappa)\cong \left\{(x,y)\in \R^2\,|\, D<0, \, {\rm d} s^2= \frac{{\rm d} x^2+{\rm d} y^2}{D^2}\right\}.\]
An orthonormal frame of vectors on ${\rm BCV}(\kappa,\tau)$ is given by
 \begin{gather}\label{VF1}
e_1 =D\frac{\pa}{\pa x} - \tau y \frac{\pa}{\pa z},\qquad e_2 =D\frac{\pa}{\pa y}+ \tau x \frac{\pa}{\pa z},\qquad e_3 =\frac{\pa}{\pa z} ,
\end{gather}
verifying the commutation relations
\[[e_1,e_2]=\frac{\kappa}{2}(-ye_1+xe_2)+2\tau e_3,\qquad [e_2,e_3]=[e_3,e_1]=0.\]
The dual $1$-forms $\omega_i$, $\langle \omega^i\,|\,e_j\rangle =\delta_{ij}$, $i, j=1,2,3$, to the orthonormal vector fields \eqref{VF1} are
\begin{gather}\label{1frm}
\omega^ 1=\frac{{\rm d} x}{D}, \qquad \omega^ 2=\frac{{\rm d} y}{ D}, \qquad \omega^3 ={\rm d} z +\tau\frac{y{\rm d} x-x{\rm d} y}{D},
\end{gather}
and we write down \eqref{BCV} as
\begin{gather*}
{\rm d} s^2_{\rm BCV}=\sum_{i=1}^3\omega^i\otimes\omega^i.
\end{gather*}

Let $\mc{D}$ be a distribution generated by $e_1$, $e_2$. The intrinsic $($extrinsic$)$ ideal is given by $\mc{J}=\langle\omega_3\rangle$ $($respectively, $\mc{I}=\langle \omega_1,\omega_2\rangle)$.

If $\tau\not= 0$, the distribution is step $2$ everywhere and $\omega^3$ is a contact form. If we consider the sub-Riemannian metric
\begin{gather*}{\rm d} s^2_{\mc{D}}=\sum_{i=1}^2\omega^i\otimes\omega^i,
\end{gather*}
then the BCV-space is a sub-Riemannian manifold $\big({\rm BCV},\mc{D},{\rm d} s^2_{\mc{D}}\big)$.
\end{Theorem}

\begin{Remark}Note that the BCV metrics appearing in Cases 1, 2, 5, 6, 7 are metrics on the corresponding naturally reductive spaces of Theorem~\ref{thm1}. Note that naturally reductive space~$\mc{H}^3$ in Theorem~\ref{thm1}, corresponding to the isometry group of dimension~6, is not a BCV space.
\end{Remark}
See \cite{fer} for a generalization of BCV spaces to 7 dimensions.

Applying the Cayley transform, we can formulate Theorem~\ref{PR15} on the Siegel upper half-plane, instead on the Siegel disk $D(x,y,z;\kappa)$ defined by~\eqref{M3}. We get
\begin{Remark}\label{RM5}With the Cayley transform
\begin{gather*}
\sqrt{-\frac{\kappa}{4}}\zeta=\frac{v-\ii}{v+\ii},
\end{gather*}
we get for \eqref{M3}
\begin{gather*}
D(x,y,z;\kappa)=1+\frac{4}{\kappa}\big(x^2+y^2\big)=4\frac{\Im v}{|v+\ii|^2}>0,
\end{gather*}
where $\zeta:=x+\ii y$. If $v:=\alpha+\ii \beta$, $E:=\alpha^2+(\beta+1)^2$, then the left invariant one-forms \eqref{1frm} in the new variables are
\begin{gather*}
\omega^1 =\frac{\sqrt{-\kappa}}{4}\frac{\big({-}\alpha^2+\beta^2+2\beta+1\big){\rm d}\alpha-2\alpha(\beta+1){\rm d}\beta}{\beta E},\\
\omega^2 =\frac{\sqrt{-\kappa}}{4}\frac{(\alpha^2-\beta^2-2\beta-1){\rm d}\alpha+2\alpha{\rm d}\beta}{\beta E},\\
\omega^3 ={\rm d} z+\frac{2\tau}{\kappa}\frac{\big(\alpha^2-\beta^2+1\big){\rm d}\alpha+2\alpha\beta{\rm d}\beta}{\beta E}.
\end{gather*}
Instead of the family of metrics \eqref{BCV}, we get in Theorem~\ref{PR15}
\begin{gather*}
{\rm d} s^2_{\rm BCV}(\alpha,\beta,z;\kappa,\tau)=-\frac{1}{\kappa}\frac{{\rm d} \alpha^2+{\rm d} \beta^2}{\beta^2}+\big(\omega^3\big)^2.
\end{gather*}
 \end{Remark}

\subsection{G.o.~spaces}\label{go}
The natural reductivity is a special case of spaces with a more general property than~$(*)$, see~\cite{kwv}:\\
$(**)$ {\it Each geodesic of} $(M,g)=G/H$ {\it is an orbit of a one parameter group of isometries} $\{\exp tZ\}$, $Z\in\got{g}$.

\begin{deff}\label{DEF3}
A vector $X \in\got{g}\setminus \{0\}$ is called a {\it geodesic vector} if the curve $\gamma(t)=(\exp tX)(p)$ is a geodesic.
\end{deff}

Riemannian homogeneous spaces with property $(**)$ are called {\it g.o.~spaces} (g.o.\ = geodesics are orbits). All naturally reductive spaces are g.o.\ manifolds.

Kowalski and Vanhacke \cite{kwv} have proved that
\begin{Proposition}[geodesic lemma]\label{PRR} On homogeneous Riemannian manifolds $M=G/H$ a~vector $X \in\got{g}\setminus \{0\}$ is geodesic if and only if
\begin{gather}\label{BCOND}
B([X,Y]_{\m}, X_{\m})=0, \qquad \forall\, Y\in \m.
\end{gather}
\end{Proposition}

It is known, cf.~\cite{kwv}:
\begin{Theorem}\label{BTHM}Every simply connected Riemannian g.o.~space $(G/H,g)$ of dimension $n\le 5$ is a naturally reductive Riemannian manifold.
\end{Theorem}

Kowalski and Szenteke \cite{ks} proved that
\begin{Theorem}Any homogeneous Riemannian manifold admits at least one homogeneous geodesic through every point $o\in M$.
\end{Theorem}
More details on g.o.\ spaces and examples are given in \cite{zd}.

\section{Balanced metrics and Berezin quantization}\label{compl}

 In our approach to Berezin quantization on K\"ahler manifold $M$ of complex dimension $n$, see, e.g.,~\cite{SB15}, we considered the K\"ahler two-form
\begin{gather*}
\omega_M(z)=\ii\sum_{\alpha,\beta=1}^n h_{\alpha\bar{\beta}} (z) {\rm d} z_{\alpha}\wedge {\rm d}\bar{z}_{\beta}, \qquad h_{\alpha\bar{\beta}}= \bar{h}_{\beta\bar{\alpha}}= h_{\bar{\beta}\alpha}.
\end{gather*}
We have considered homogenous K\"ahler manifolds $M=G/H$, where the $G$-invariant K\"ahler two-form is deduced from a K\"ahler potential~$f$
\begin{gather*}
h_{\alpha\bar{\beta}}= \frac{\pa^2 f}{\pa {z}_{\alpha}\pa \bar{z}_{{\beta}}} .
\end{gather*}

We have applied Berezin recipe to quantization \cite{ber73,ber74,berezin,ber75}, where the K\"ahler potential is obtained from the scalar product of two Perelomov CS-vectors $e_z$, $z\in M$ \cite{perG}
 \begin{gather*}
f(z,\bar{z})=\ln K_M(z,\bar{z}), \qquad K_M(z,\bar{z})=(e_{\bar{z}},e_{\bar{z}}),
\end{gather*}
i.e., \eqref{KALP}.

This choice of $f$ corresponds to the situation where the so called $\epsilon$-function, see \cite{cahII, raw,Cah},
\begin{gather*}
\epsilon(z) := \e^{-f(z)}K_M(z,\bar{z}),
\end{gather*}
is constant. The corresponding $G$-invariant metric is called {\it balanced metric}. This denomination was firstly used in~\cite{don} for compact manifolds, then it was used in \cite{arr} for noncompact manifolds and also in~\cite{alo} in the context of Berezin quantization on homogeneous bounded domain, and we have used it in the case of the partially bounded domain $\mc{D}^J_n$ -- the Siegel--Jacobi ball~\cite{SB15}.

We recall that in \cite{cahII, raw,Cah} Berezin's quantization on homogenous K\"ahler manifolds via CS was globalized and extended to non-homogeneous manifolds in the context of geometric (pre-)quantiza\-tion~\cite{Kos,woo}. To the K\"ahler manifold $(M,\omega)$, it is also attached the triple $\sigma =(\gl,h,\nabla)$, where~$\gl$ is a holomorphic (prequantum) line bundle on~$M$, $h$~is the Hermitian metric on~$\gl$ and~$\nabla$ is a connection compatible with metric and the K\"ahler structure~\cite{SBS}. The connection~$\nabla$ has the expression $\nabla=\pa +\pa \ln \hat{h} +\bar{\pa}$. The manifold is called {\it quantizable} if the curvature of the connection $F(X,Y)=\nabla_X\nabla_Y-\nabla_Y\nabla_X-\nabla_{[X,Y]}$ has the property that $F=-\ii \omega_M $, or $\pa\bar{\pa} \log \hat{h} =\ii \omega_M$, where $\hat{h}$ is a local representative of~$h$, taken $\hat{h}(z)=K^{-1}_M(z,\bar{z})$. Then $\omega_M$ is integral, i.e., the first Chern class is given by
\begin{gather*}
c_1[\mc{L}]=\frac{\ii}{2\pi}F=\frac{\omega}{\pi},\end{gather*}
and we have \eqref{KALP}.

\section[Killing vectors on $S^2$, $\mc{D}_1$ and $\R^2$]{Killing vectors on $\boldsymbol{S^2}$, $\boldsymbol{\mc{D}_1}$ and $\boldsymbol{\R^2}$}\label{sfera}

\subsection[Killing vectors on $S^2$]{Killing vectors on $\boldsymbol{S^2}$}

We consider on the sphere $S^2$
\begin{gather*}
\big\{x\in\R^3\,|\,x_1^2+x_2^2+x_3^2=R^2\big\}, \qquad R>0,\end{gather*}
the spherical coordinates, as in Fig.~\ref{fig1}, where $R=1$, $0\le \theta\le \pi$, $0\le \varphi\le 2\pi$ and
\begin{gather}\label{pesfera}
x_1 = R \sin\theta\cos \varphi,\qquad
x_2 =R \sin\theta\sin \varphi,\qquad
x_3 = R \cos \theta.
\end{gather}
The metric on $S^2$ is
\[{\rm d} s^2_{S^2}(\theta,\varphi) = \g_{\theta\theta}{\rm d} \theta^2+g_{\theta\varphi}{\rm d}
\theta{\rm d} \varphi
+ g_{\varphi\varphi}{\rm d} \varphi^2, \]
where
\[g_{\theta\theta}=1, \qquad g_{\theta\varphi}=0,\qquad
g_{\varphi\varphi}=\sin^2\theta ,\qquad
g^{\theta\theta}=1, \qquad g^{\theta\varphi}=0,\qquad
g^{\varphi\varphi}=\frac{1}{\sin^2\theta} ,\]
i.e.,
\begin{gather}\label{dissf}
{\rm d} s^2_{S^2}(\theta,\varphi)={\rm d}\theta^2+\sin^2\theta{\rm d} \varphi^2.
\end{gather}

\begin{figure}[t]\centering
\includegraphics{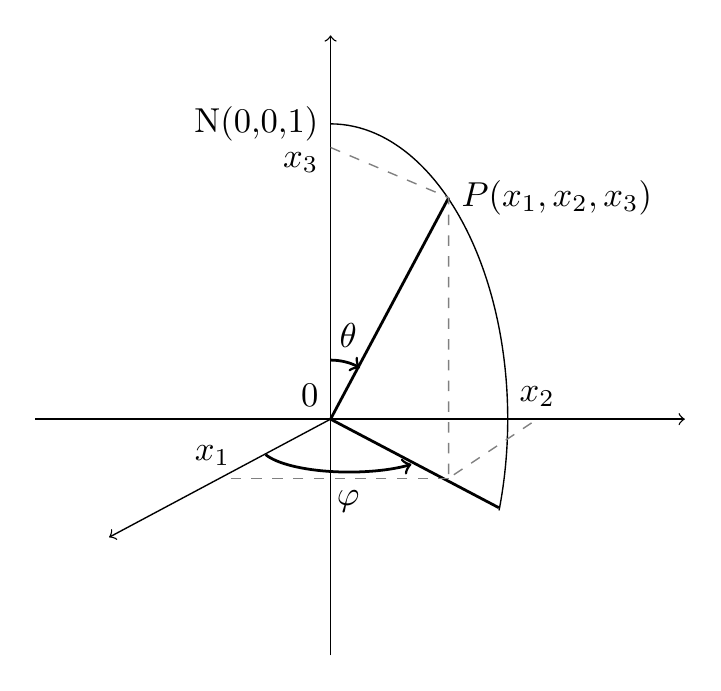}
\caption{Spherical coordinates.}\label{fig1}
\end{figure}

\begin{figure}[t]\centering
\includegraphics{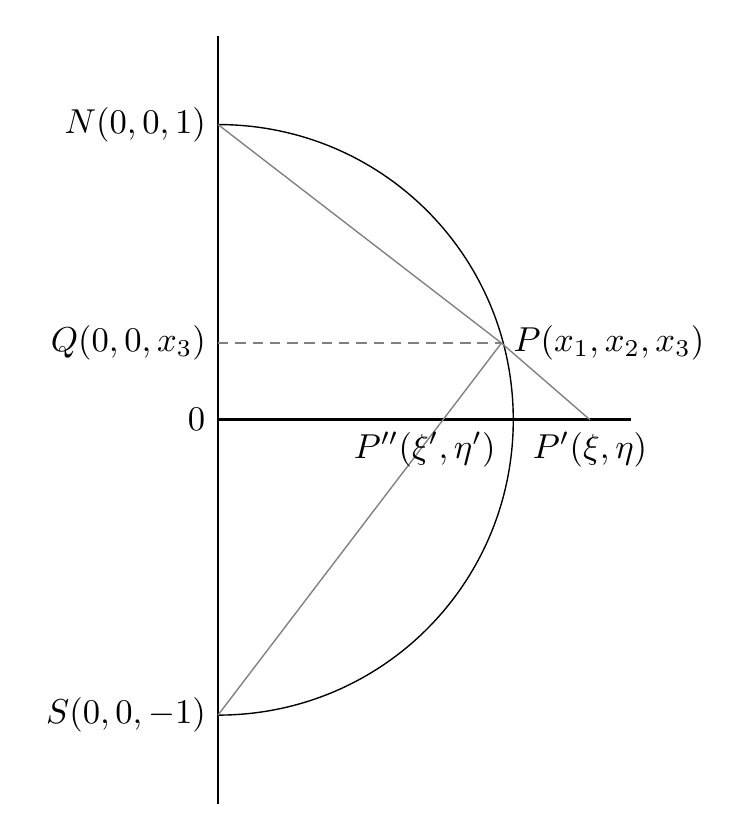}
\caption{Stereographic projection.}\label{fig2}
\end{figure}

We consider a unitary sphere with spherical coordinates~\eqref{pesfera} measured from the origin $O(0,0,0)$. The north (south) pole has coordinates $N(0,0,1)$ (respectively, $S(0,0,-1)$). We take a point $P(x_1,x_2,x_3)$ on the sphere $S^2$ and let $P'(\xi,\eta)$ ($P''(\xi',\eta')$) be the intersection of the line~$NP$ (respectively~$SP$) with the plane $x_3=0$, see Fig.~\ref{fig2}. The triangles $\Delta QNP$ and $\Delta ONP'$ (respectively $\Delta QPS$ and $\Delta OP''S$) are similar, and we have
\begin{gather}\label{ch1}\frac{1-x_3}{1}=
\frac{x_1}{\xi}=\frac{x_2}{\eta},\qquad \frac{1+x_3}{1}= \frac{x_1}{\xi'}=\frac{x_2}{\eta'}. \end{gather}The change of coordinates
$(x_1,x_2,x_3)\rightarrow (\xi,\eta)$ (respectively, $(x_1,x_2,x_3)\rightarrow (\xi'\eta')$) is given by the formulas
\[(\xi,\eta)=\left(\frac{x_1}{1-x_3}, \frac{x_2}{1-x_3}\right), \qquad (\xi',\eta')=\left(\frac{x_1}{1+x_3}, \frac{x_2}{1+x_3}\right), \]
The change of coordinates $(\xi,\eta)\rightarrow (x_1,x_2,x_3)$ (respectively, $(\xi',\eta')\rightarrow (x_1,x_2,x_3)$) is
\begin{gather*}
x_1 =\frac{2\xi}{1+\xi^2+\eta^2},\qquad x_2=\frac{2\eta}{1+\xi^2+\eta^2},\qquad x_3=\frac{-1+\xi^2+\eta^2}{1+\xi^2+\eta^2},\\
 x_1 =\frac{2\xi'}{1+(\xi')^2+(\eta')^2},\qquad x_2=\frac{2\eta'}{1+(\xi')^2+(\eta')^2},\qquad
 x_3=\frac{1-(\xi')^2-(\eta')^2}{1+(\xi')^2+(\eta')^2}.
\end{gather*}

Let $\C\ni z:=\xi-\ii \eta$, $z':=\xi'+\ii \eta'$. Then $zz'=1$. With \eqref{ch1}, \eqref{pesfera}, we find
\begin{gather}\label{ZSF}z=\cot\frac{\theta}{2}\e^{-\ii \varphi}.\end{gather}

Introducing \eqref{ZSF} into the metric on the Riemann sphere
\begin{gather*}
{\rm d} s^2(z)=4\frac{{\rm d} z{\rm d} \bar{z}}{\big(1+|z|^2\big)^2}\end{gather*}
corresponding to the K\"ahler two-form \eqref{OSF}{\samepage
\begin{gather}\label{OSF}\omega=2j \frac{\ii {\rm d} z \wedge {\rm d}
 \bar{z}} {\big(1+|z|^2\big)^{2}},\end{gather} where, if we take $2j=4$, we get again \eqref{dissf}.}

The equations \eqref{kill} of the covariant components $(X_{\theta},X_{\varphi})$ of the Killing vectors on the sphere~$S^2$ read
\begin{gather*}
X_{\theta,\theta} =0,\qquad
X_{\varphi,\varphi} +2\sin\theta\cos\theta X_{\theta}= 0,\qquad
X_{\theta,\varphi} +X_{\varphi,\theta}-2\cot\theta X_{\vartheta}=0.
\end{gather*}
But \begin{gather*}
X_{\theta} =g_{\theta\nu}X^{\nu}=X^{\theta}=u,\qquad
X_{\varphi}= g_{\nu\varphi}X^{\nu} =g_{\varphi\varphi}X^{\varphi}=\sin^2\theta v,
\end{gather*}
where $\big(X^{\theta},X^{\varphi}\big)$ are the contravariant components of the Killing vector fields on the sphere~$S^2$. The equations of the contravariant components of the Killing vector $(u,v):=\big(X^{\theta},X^{\varphi}\big)$ become
\begin{gather*}
 \frac{\pa u}{\pa \theta} =0,\qquad \sin^2\theta\frac{\pa v}{\pa \varphi} +\sin\theta\cos\theta u=0,\qquad \frac{\pa u}{\pa \varphi} +\frac{\pa v}{\pa\theta}-2\cot\theta\sin^2\theta v=0.
\end{gather*}
We find
\begin{Remark}\label{SPHC} There are three linearly independent Killing vectors on the sphere~$S^2$
\begin{gather}\label{XYZ}X=\frac{\pa}{\pa \varphi}, \qquad Y=\sin \varphi \frac{\pa}{\pa \theta}
+\cos\varphi \cot\theta\frac{\pa}{\pa \varphi},\qquad Z=\cos \varphi \frac{\pa}{ \pa\theta}-\sin\varphi\cot\theta\frac{\pa}{\pa \varphi},\end{gather}
 which verify the commutation relations
\begin{gather*}
[X,Y]=Z,\qquad [Z,X]=Y,\qquad [Y,Z]=X. \end{gather*}
The Killing vector fields \eqref{XYZ} in spherical coordinates $(\theta,\varphi)$ on the sphere $S^2$ in the stereographic coordinates $(\xi,\eta)$ are
\begin{gather}X = -\eta\frac{\pa}{\pa\xi} + \xi\frac{\pa}{\pa\eta},\qquad Y = \frac{1}{2}\left[2\xi\eta\frac{\pa}{\pa\xi} + \big(1 - \xi^2+\eta^2\big)\frac{\pa}{\pa\eta}\right],\nonumber\\
Z = -\frac{1}{2}\left[\big(1+\xi^2-\eta^2\big)\frac{\pa}{\pa\xi} + 2\xi \eta\frac{\pa}{\pa\eta}\right].\label{KILLXI}\end{gather}
\end{Remark}
\eqref{KILLXI} are equations in \cite[p.~292]{vr}: our $(X,Y,Z)$ in \eqref{KILLXI} correspond to $\big({-}Z,\frac{1}{2}Y,-\frac{1}{2}X\big)$, with $m=1$ to formulas of Vranceanu, where the Riemann metric on the Riemann sphere~$S^2$ is
\begin{gather}\label{RIEMM}{\rm d} s^2= \frac{{\rm d} \xi^2+{\rm d} \eta^2}{\big[1+m\big(\xi^2+\eta^2\big)\big]^2}.\end{gather}

\subsection[Killing vectors on the Siegel disk $\mc{D}_1$]{Killing vectors on the Siegel disk $\boldsymbol{\mc{D}_1}$}
The metric on the Siegel disk $|z|<1$ is
\begin{gather}\label{MCD}ds^2=4\frac{{\rm d} z{\rm d}\bar{z}}{B^2},\qquad B=1-|z|^2,\qquad z=\xi-\ii\eta.\end{gather}
The equations \eqref{LG} of the Killing vectors $u\frac{\pa}{\pa \xi}+v\frac{\pa}{\pa \eta}$ corresponding to the metric \eqref{MCD}, which are obtained as solution of the equation $L_Xg=0$, are
\begin{gather*}
 \frac{2}{B}(\xi u+\eta v)+\frac{\pa u}{\pa \xi}=0,\qquad 
 \frac{2}{B}(\xi u+\eta v)+\frac{\pa v}{\pa \eta}=0,\qquad 
 \frac{\pa v}{\pa \xi}+ \frac{\pa u}{\pa \eta} =0. 
\end{gather*}
We find for $\mc{D}_1$
\begin{Remark}\label{KVSD}The Killing vectors on the Siegel disk $\mc{D}_1$ corresponding to the metric \eqref{MCD} are
\begin{gather}
X_1 =\frac{1}{2}\big(\xi^2-\eta^2-1\big)\frac{\pa }{\pa
 \xi}+\xi\eta\frac{\pa}{\pa\eta},\qquad
Y_1 =\xi\eta\frac{\pa}{\pa
 \xi}+\frac{1}{2}\big(\eta^2-\xi^2-1\big)\frac{\pa}{\pa \eta},\nonumber\\
Z_1= \eta\frac{\pa}{\pa\xi}-\xi \frac{\pa}{\pa\eta}.\label{OKILL}
\end{gather}
The Killing vectors \eqref{OKILL} on the Siegel disk $\mc{D}_1$ verify the commutation relations
\begin{gather*}
[X_1,Y_1]=-Z_1,\qquad [Y_1,Z_1]=X_1,\qquad [Z_1, X_1]= Y_1.
\end{gather*}
\end{Remark}

\subsection[Fundamental vector fields as Killing vector fields on $\mc{D}_1$ and $\mc{X}_1$]{Fundamental vector fields as Killing vector fields on $\boldsymbol{\mc{D}_1}$ and $\boldsymbol{\mc{X}_1}$}\label{FKVF}

We recall some general facts about Hermitian symmetric spaces, see, e.g., \cite{sbl, Wolf1,Wolf2}.

Let
\begin{itemize}\itemsep=0pt
\item $X_n=G_n/K$: Hermitian symmetric space of noncompact type.
\item $X_c$: compact dual form of $X_n$, $X_c=G_c/K$.
\item $G_n$: largest connected group of isometries of $X_n$, a centerless semisimple Lie group.
\item $G_c$: compact real form of $G_n$.
\item $G^c=G^c_n=G^c_c=G$: complexification of $G_c$ and $G_n$.
\item $K$: maximal compact subgroup of $G_n$.
\item $\got{g}_n$, $\got{g}$, $\got{g}_c$, $\got{k}$: Lie algebras of $G_n$, $G$, $G_c$, $K$ respectively.
\item $\got{g}_n=\got{k}+\got{m}_n$, sum of $+1$ and $-1$ eigenspaces of the Cartan involution $\sigma$.
\item $\got{g}=\got{g}^c_n=\got{k}^c+\got{m}$: complexification, where $\got{m}=\got{m}^c_n$.
\item $\got{g}_c=\got{k}+\got{m}_c$: compact real form of $\got{g}_n$, where $\got{m}_c=\ii \got{m}_n$.
\end{itemize}

We consider the simple Lie algebra $\got{sl}(2,\C)=\langle F,G,H\rangle_{\C}$, whose generators verify the commutation relations
\begin{gather}\label{FGHCOM}
[F,G]=H,\qquad [G,H]=2 G,\qquad [H,F]=2F.
\end{gather}
We consider the following matrix realization of the $\got{sl}(2,\C)$ algebra
\begin{gather}\label{FGH}F=e_{12}=\left(\begin{matrix} 0 & 1\\ 0 &
 0\end{matrix}\right), \qquad G= e_{21}=\left(\begin{matrix} 0 & 0\\ 1 &
 0\end{matrix}\right), \qquad H=e_{11}-e_{22}=\left(\begin{matrix} 1 & 0\\ 0 &
 -1 \end{matrix}\right).\end{gather}
To the complex Lie algebra $A_1=\got{sl}(2,\C)$ are associated the compact real form $\got{sl}(2,\C)_c=\got{su}(2)$ and the non-compact
real forms $\got{su}(1,1)$ and $\got{sl}(2,\R)$, see \cite[pp.~186,~446]{helg}, \cite{sbl,Wolf1,Wolf2}, and we have
\begin{subequations}
\begin{gather}
\got{su}(2) =\langle \ii H, -F+G,\ii(F+G)\rangle_{\R},\label{SU2}\\
\got{su}(1,1) =\langle \ii H, \ii (F-G),F+G\rangle _{\R},\label{SU11}\\
\got{sl}(2,\R) =\langle H, -F+G,F+G\rangle _{\R}.\label{SL2R}
\end{gather}
\end{subequations}
We have also the isomorphisms between the compact real forms
\begin{gather*}
\got{su}(2)\sim\got{so}(3)\sim\got{sp}(1),
\end{gather*}
and the non-compact real forms
\begin{gather}\label{ISO2}
\got{sl}(2,\R)\sim\got{su}(1,1)\sim \got{so}(2,1)\sim\got{sp}(1,\R).\end{gather}
We have also the relations
\begin{subequations}\label{FMULTE}
\begin{gather}\got{su}(1,1) =\got{g}_n=\got{k}+\got{m}_n, \qquad \got{k}=\ii H, \qquad \got{m}_n=\langle \ii (F-G), F+G\rangle _{\R},\label{825a}\\
\got{su}(2) =\got{g}_c=\got{k}+\got{m_c},\qquad \got{m}_c=\langle F-G,\ii (F+G)\rangle _{\R}, \\
\got{sl}(2,\C) =\got{g}=\got{su}(1,1)^c=\got{su}(2)^c=\got{k}^c+\got{m}^c_n,\qquad
\got{k}^c=\langle H\rangle _{\C}, \qquad \got{m}_n=\langle F,G\rangle _{\C}.
\end{gather}
\end{subequations}
We calculate the fundamental vector fields for the real noncompact group ${\rm SU}(1,1)$. Let us denote the elements of the Lie algebra $\got{su}(1,1)$ as
\begin{gather}
G_1 := \ii H = \ii\left(\begin{matrix} 1& 0\\ 0 &
 -1\end{matrix}\right) ,\qquad G_2 := \ii(F - G)= \ii\left(\begin{matrix}
 0 & 1\\ -1 & 0\end{matrix}\right) ,\nonumber\\
 G_3 := F + G = \left(\begin{matrix} 0& 1\\ 1 & 0\end{matrix}\right) .\label{G123} \end{gather}
Note the commutation relations
\begin{gather}\label{COMG}
[G_1,G_2]= -2G_3,\qquad [G_2,G_3]= 2 G_1, \qquad [G_3,G_1]=-2 G_2.
\end{gather}
If we make the notation $G_i=2G'_i$, $i=1,2,3$, then the commutation relations \eqref{COMG} became
\begin{gather}\label{COMG2}
[G'_1,G'_2]= -G'_3,\qquad [G'_2,G'_3]= G'_1, \qquad [G'_3,G'_1]=- G'_2.
\end{gather}

We obtain, see also \cite[p.~294]{vil},
\begin{gather*}
\e^{tG_1} =\left(\begin{matrix}\e^{\ii t}& 0\\ 0 &\e^{-\ii t}
 \end{matrix}\right),\qquad
\e^{tG_2} =\left(\begin{matrix}\cosh t& \ii \sinh t\\ -\ii \sinh t
 &\cosh
 t \end{matrix}\right),\qquad
\e^{tG_3} =\left(\begin{matrix}\cosh t & \sinh t\\ \sinh t
 &\cosh t \end{matrix}\right).
\end{gather*}
We get
\begin{gather}\label{PTT}
\frac{\pa \big(\e^{tG_1}.w\big)}{\pa t}\Big|_{t=0} = 2\ii w,\qquad
\frac{\pa \big(\e^{tG_2}.w\big)}{\pa t}\Big|_{t=0} = \ii \big(1+w^2\big),\qquad
\frac{\pa \big(\e^{tG_3}.w\big)}{\pa t}\Big|_{t=0} = 1-w^2.\!\!\!
\end{gather}
With \eqref{PTT}, we get the corresponding holomorphic fundamental vector fields on the Siegel disk $\mc{D}_1=\frac{{\rm SU}(1,1)}{{\rm U}(1)}$:
\begin{gather}\label{PTT1}
G^*_1 = 2\ii w\frac{\pa}{\pa w},\qquad
G^*_2 = \ii \big(1+w^2\big)\frac{\pa}{\pa w},\qquad
G^*_3 = \big(1-w^2\big)\frac{\pa}{\pa w}.
\end{gather}
If we introduce $w=\xi-\ii \eta$, we write \eqref{PTT1} as
\begin{gather*}
G^*_1 = Z_1+\ii \left(\xi\frac{\pa}{\pa \xi}+\eta\frac{\pa}{\pa \eta}\right),\qquad
G^*_2 = Y_1+\frac{\ii}{2}\left[\big(1+\xi^2-\eta^2\big)\frac{\pa}{\pa
 \xi}+2\xi\eta\frac{\pa}{\pa \eta}\right],\\
G^*_3 = -X_1 +\frac{\ii}{2}\left[2\xi\eta\frac{\pa}{\pa
 \xi}+\big(1-\xi^2+\eta^2\big)\frac{\pa}{\pa \eta}\right] ,
\end{gather*}
where $X_1$, $Y_1$, $Z_1$ are the Killing vector fields of the Siegel disk $\mc{D}_1$ calculated in~\eqref{OKILL}.

We also have the relations, see also \cite[p.~353]{vil}
\begin{gather}\label{expFGH}\e^{tF}=\left(\begin{matrix} 1 & t\\ 0 &
 1 \end{matrix}\right),\qquad \e^{tG}=\left(\begin{matrix}
 1 & 0\\ t &
 1 \end{matrix}\right),\qquad \e^{tH}=\left(\begin{matrix}
 \e^t & 0\\ 0 &
 \e^{-t} \end{matrix}\right),\\
 \label{expFGH2}
\e^{t(F+G)}=\left(\begin{matrix}\cosh t&\sinh t\\ \sinh t & \cosh
 t\end{matrix}\right),\qquad \e^{t(F-G)}=\left(\begin{matrix}\cos
 t&\sin t\\ -\sin t & \cos t\end{matrix}\right),\\
\e^{tF}\cdot \tau=\tau +t,\qquad \e^{tG}\cdot \tau
=\frac{\tau}{1+t\tau},\qquad \e^{tH}\cdot \tau= \e^{2t}\tau,\nonumber\\
 \frac{{\rm d} }{{\rm d} t}\e^{tF}\cdot \tau\Big|_{t=0}= 1,\qquad \frac{{\rm d} }{{\rm d}
t}\e^{tG}\cdot \tau\Big|_{t=0}=-\tau^2,\qquad \frac{{\rm d} }{{\rm d} t}\e^{tH}\cdot
\tau\Big|_{t=0}=2\tau,\nonumber\\
F^*= \pa_{\tau},\qquad G^*= -\tau^2\pa_{\tau},\qquad H^*=2\tau\pa_{\tau}.\nonumber
\end{gather}
If we put $\tau=x+\ii y$, we find the fundamental vector fields on the homogenous manifold $\mc{X}_1$, see Theorem~\ref{THM0}(1) and~\eqref{MCEC}
\begin{gather}\label{FUNDFGH}
F^*_1=\frac{\pa }{\pa x},\qquad G^*_1=\big(y^2-x^2\big)\frac{\pa }{\pa x}-2 xy\frac{\pa}{\pa y},\qquad H^*_1=2\left(x\frac{\pa}{\pa x}+y\frac{\pa }{\pa y}\right).
\end{gather}
In the convention of Section~\ref{section1}, the vector fields $F^*_1$, $G^*_1$, $H^*_1$ are $\db{F}$, $\db{G}$, $\db{H}$.

If \begin{gather}\label{MRA}A=\left(\begin{matrix}a&b\\c&d\end{matrix}\right)\in \SL,
\end{gather}
then, with formula \eqref{DEFAD},
\begin{gather}\label{DEFAD}
\Ad(g)X=g Xg^{-1},\qquad g\in G, \quad X\in \g,
\end{gather} we find easily
\begin{gather*}
\Ad (A)F = a^2F-c^2G-acH,\\
\Ad (A)G = -b^2F+d^2G+bdH,\\
\Ad (A)H = -2abF+2cdG+(ad+bc)H.
\end{gather*}
We find out that in the base \eqref{FGH}
\begin{gather}\label{ADDAA}\Ad(A)= \left(\begin{matrix} a^2 & -c^2 & -ac \\ -b^2 & d^2
 & bd\\ -2ab & 2cd & ad+bc\end{matrix}\right),
\end{gather}
and $\det(\Ad)=1$,

Now let us consider an element $X\in \got{sl}(2,\R)$
\begin{gather}\label{ECXX}
X=aH+bF+cG=\left(\begin{matrix}a &b \\c &-a\end{matrix}\right).\end{gather}
Then we find
\begin{gather}\label{3ad}
[X,H] = -2bF+2cG,\qquad [X,F] = 2aF-cH,\qquad [X,G] =-2a G +bH.
\end{gather}
With \eqref{3ad} we find in the base $H$, $F$, $G$ the expression of $\ad(X)$ for $X$ given by~\eqref{ECXX}
\begin{gather}\label{addX}
\ad (X) = \left(\begin{matrix} 0 & -2b & 2c\\ -2c & 2a &0\\ b & 0&
 -2a\end{matrix}\right),\end{gather}
and \[\tr \ad =0.\]
\eqref{addX} implies, see also \cite[p.~551]{helg}:
\begin{gather}\label{BXX}K(X,X)=\tr(\ad X\circ\ad X)=8\big(a^2+bc\big)=4 \tr (XX).\end{gather}

As in Remark \ref{REMCUL}, we consider \[X=aX_1+bX_2+cX_3\in\got{su}(2),\] where,
as in \eqref{SU2}, \[X_1=\ii H,\qquad X_2=-F+G,\qquad X_3=\ii (F+G).\]
Then
\[\ad X=\left(\begin{matrix} 0 & -2c& 2b\\ 2c& 0& -2a\\ -2b&
 2a & 0\end{matrix}\right) ,\]
and \begin{gather}\label{KXYS}K(X,Y)=-4(aa'+bb'+cc'),\end{gather} i.e., the Killing form for ${\rm SU}(2)$ is $K(X,Y)=4\tr(XY)$.

Note that for $\got{su}(2)$, we have $\m=\langle X_2,X_3\rangle $.

 Putting together \eqref{G123}--\eqref{BXX} and \eqref{KXYS}, we have proved
\begin{Remark}\label{REMCUL}With \eqref{G123}, \eqref{FGH}, \eqref{nr2}, we get
\begin{gather}\label{3stele}
K_0=-\frac{\ii}{2}G_1,\qquad K_+=\frac{1}{2}(G_2+\ii G_3),\qquad K_-=-\frac{1}{2}(G_2-\ii G_3).
\end{gather}
Introducing in \eqref{PTT1} and \eqref{3stele}, we get the holomorphic fundamental vector fields
\[K^*_0= w\frac{\pa}{\pa w}, \qquad K^*_+= \ii\frac{\pa}{\pa w}, \qquad K^*_-=-\ii w^2\frac{\pa}{\pa w}, \qquad w\in\C,\qquad |w|<1. \]
Note that the vector fields $\Re G*_i$, $i=1,2,3$ verify the commutation relations \eqref{COMG2} with the sign~$-$, i.e., the (real) Killing vector fields $Z_1$, $Y_1$, $-X_1$ on $\mc{D}_1$ are the real part of the fundamental vector fields $G'^*_1$, $G'^*_2$, $G'^*_3$, corresponding to the metric~\eqref{MCD}.

The fundamental vector fields $F^*_1$, $G^*_1$, $H^*_1$ associated to the generators $F$, $G$, $H$ \eqref{FGH} of~$\got{sl}(2,\C)$, given by~\eqref{FUNDFGH}, verify
the commutation relations~\eqref{FGHCOM} with a minus sign. They are Killing vector fields corresponding to the Killing equation
\[ -X^2+y\pa_xX^1=0,\qquad \pa_xX^2+\pa_yX^1=0,\qquad -X^2+y\pa_yX^2=0 \]
associated to the metric \begin{gather*}
c_1\frac{{\rm d} x^2+{\rm d}
 y^2}{4y^2},\qquad c_1>0\end{gather*} on the Siegel upper half-plane $\mc{X}_1$, $x,y\in\R$, $y>0$.

If $ A \in\SL$ has the expression \eqref{MRA}, then the expression of $\Ad(A)$ with respect to the base $F$, $G$, $H$ \eqref{FGH} is \eqref{ADDAA}, and the group $\SL$ is unimodular.

The $\ad$ matrix in the base $H$, $F$, $G$ of $\got{sl}(2,\R)$ is given by~\eqref{addX}. The Killing form for~$\got{sl}(2,\R)$ is
\begin{gather}\label{ADS}K(X,Y)=4 \tr (XY).\end{gather}
The Killing form \eqref{ADS} is $\SL$-invariant and
 verifies \eqref{INVBXY}. Note that
\begin{gather}\label{BFGHB}
K(H,H)=K(F,G)=4,\end{gather}
and $K(X,Y)=0$ for all $X,Y\in\got{sl}(2,\R)$ different of the choice in \eqref{BFGHB}.

The Killing form for the compact group ${\rm SU}(2)$ is $K(X,Y)=4\tr(XY)$, and $\m=\langle X_2,X_3\rangle $.
 \end{Remark}

 \subsection[Killing vectors on $\R^2$]{Killing vectors on $\boldsymbol{\R^2}$}

The Perelomov's coherent state vectors (Glauber's coherent states) for {\it the oscillator group} are, see, e.g.,~\cite{sb6},
\[e_z:=\e^{z a^{\dagger}}e_0 , \]
and the scalar product is
\begin{gather}\label{SCH}(e_{\bar{z}},e_{\bar{z}'})=\e^ {z\bar{z}'}.
\end{gather}
The scalar product \eqref{SCH} of Glauber coherent states on~$\C$ implies the metric on $\R^2$ \eqref{MTR} ${\rm d} s^2_{\R^2}={\rm d} x_1^2+{\rm d} x_2^2$, where we have considered $z=x_1+\ii x_2$.

Let as consider a vector field on $\R^2$
\begin{gather}\label{X1X2}X=X^1\frac{\pa}{\pa x_1}+X^2\frac{\pa}{\pa x_2}.\end{gather}
We formulate a remark, see also in \cite[Section~4.6.7, p.~83]{mari}:
\begin{Remark}\label{Rem11} The Killing vectors on $\R^2$ associated with the metric~\eqref{MTR} are
\begin{gather}\label{KILPL}
A X+ B Y +C Z,\end{gather}
where
\begin{gather*} X=-x_2\frac{\pa}{\pa x_1}+ x_1\frac{\pa}{\pa x_2},\qquad
Y=\frac{\pa}{\pa x_1},\qquad Z= \frac{\pa}{\pa x_2},
\end{gather*}
verifying the commutation relations
\begin{gather*}
[X,Y]=-Z,\qquad [Y,Z]=0,\qquad [Z,X]= -Y.\end{gather*}
$-X$ is a rotation around $(0,0)\in\R^2$. $Y$ ($Z$) represents a translation around the $x_1$ (respectively~$x_2$) axis. The Killing vectors \eqref{KILPL} can be put
into correspondence with matrix representation~\eqref{bazae2}
\begin{gather}\label{bazae2}
a_1=\left(\begin{matrix}0& 0 & 1\\0 &0
 & 0\\0 &0
 &
 0\end{matrix}\right),\qquad a_2=\left(\begin{matrix}0& 0 & 0\\0 &0
 & 1\\0 &0
 &
 0\end{matrix}\right),\qquad a_3= \left(\begin{matrix}0& 1& 0\\-1 &0
 & 0\\0 &0 & 0\end{matrix}\right),\end{gather} of the
Lie algebra $\got{e}(2)$ in the representation \eqref{ge2}
\begin{gather}\label{ge2}
g=\left(\begin{matrix} \cos\theta & -\sin\theta & a\\\sin\theta
 &\cos\theta & b \\ 0& 0& 1\end{matrix}\right),\qquad
 \theta\in[0,2\pi),\qquad (a,b)\in \R^2,\end{gather}
of the group $E(2)$.

The Lie algebra of the Killing vectors of $\R^2$ with the Euclidean metric \eqref{MTR} is $\iota(\R^2)=\R^2\rtimes\got{so}(2)$, and the Euclidean group~$E(2)$ of the plane $\R^2$ is $E(2)=\R^2\rtimes O(2)$.
\end{Remark}

\section{Sasaki manifolds}\label{appendix4}
\subsection{Contact structures}\label{CSS}

\subsubsection{Maurer--Cartan equations}\label{S911}
Let $G$ be a Lie group with Lie algebra $\got{g}$, which has the generators $X_1,\dots,X_n$ verifying the commutation relations~\eqref{XIXJ}. To $X\in\got{g}$ we associate the left-invariant vector $\tilde{X}$ on $G$ such that $\tilde{X}_e=X$, see \cite[p.~99]{helg}.

Let $\omega_1,\dots,\omega_n$ be the 1-forms on $G$ determined by the equations $\langle \omega_i\,|\,\tilde{X}_j\rangle =\delta_{ij}$, $i,j=1,\dots,n$. Then we have the
{\it Maurer--Cartan} equations, see, e.g., \cite[Proposition~7.2, p.~137]{helg}:
\begin{gather*}
{\rm d} \omega_i=-\frac{1}{2}\sum_{j,k=1}^nc^i_{jk}\omega_j\wedge\omega_k,
\end{gather*}
where $c^i_{jk}$ are the structure constants \eqref{XIXJ}.

If $G$ is embedded in ${\rm GL}(n)$ by a matrix valued map $g=(g)_{ij}$, $i,j=1,\dots,n $, then let $\lambda$ ($L$) denote a left-invariant one-form (vector field) on $G $ and
$\rho$ ($R$) a right-invariant one-form (respectively, vector field) on $G$. We have the relations
\begin{subequations}\label{NR111}
\begin{gather}
 g^{-1}{\rm d} g =X_i\lambda_i,\qquad {\rm d} g g^{-1}= X_i\rho_i,\label{LILA}\\
\langle \lambda_a\,|\,L_b\rangle =\delta_{ab},\qquad \langle \rho_a\,|\,R_b\rangle =\delta_{ab},\\
{\rm d} \lambda_a =-\frac{1}{2}c^a_{bc}\lambda_b\wedge \lambda_c,\qquad {\rm d} \rho_a= \frac{1}{2}c^a_{bc}\rho_b\wedge\rho_c,\\
[L_a,L_b] =c^c_{ab}L_c,\qquad [R_a,R_b]=-c^c_{ab}R_c.
\end{gather}
\end{subequations}

\subsubsection{Almost contact manifolds}
Following Sasaki \cite{sas} and \cite[Definition 6.2.5]{boga}, we use
\begin{deff}\label{D9}
 Let $M_{m}$ be a $m=(2n+1)$-dimensional manifold. $M_{m}$ has a {\it $($strict$)$ almost contact structure} $(\Phi,\xi,\eta)$ (or $(\xi,\eta,\Phi)$) if there exists a~$(1,1)$-tensor field $\Phi$, a contravariant vector field ({\it Reeb vector field}, or {\it characteristic vector field})~$\xi$, and a one-form~$\eta$
\begin{gather}\label{DEFACD}
\Phi=\Phi^i_j\frac{\pa}{\pa x^i}\otimes{\rm d} x^j,\qquad \xi=\xi^i\frac{\pa}{\pa x^i},\qquad \eta=\eta_i{\rm d}
x^i,\end{gather} verifying the relations
\begin{subequations}\label{ACM}
\begin{gather}\langle \eta|\xi\rangle =1, \qquad \text{or} \qquad \eta\lrcorner \xi=1, \qquad \text{or}\qquad \eta\xi=1,\qquad \text{or}\qquad \xi^i\eta_i=1,\label{ACM1}\\
 \Phi^2X =-X+\eta(X)\xi\qquad \text{or}\qquad \Phi^2 =-\mathbbm{1}_m+\xi\otimes\eta,\qquad \text{or}\qquad \Phi^i_j\Phi^j_k=-\delta^i_k+\xi^i\eta_k, \!\!\!\label{ACM2}
\end{gather}
\end{subequations}
where we have used the convention \begin{gather*}\xi^t = \big(\xi^1, \dots, \xi^m\big)\in
M(1,m,\R),\qquad \eta = (\eta_1,\dots,\eta_m) \in M(1 , m , \R ),\\
\Phi=\big(\Phi^i_j\big)\in M(m,\R).\end{gather*}
\end{deff}
Manifolds $M$ with a structure $(\Phi,\xi,\eta)$ as in Definition~\eqref{ACM} are called {\it almost contact mani\-folds}.

Sasaki has proved, see \cite[Theorem 1.1]{sas} and \cite[equation~(5.16)]{sh}:
\begin{Theorem}\label{THM11}
For an almost contact structure $(\Phi,\xi,\eta)$, the following relations hold
\begin{gather}\label{679}
\Phi^i_j \xi^j = 0, \qquad\! \text{or} \qquad\! \Phi \xi=0,\qquad\! \Phi^i_j\eta_i = 0, \qquad\! \text{or} \qquad\! \eta \Phi =0,\qquad\!
\operatorname{Rank} \big(\Phi^i_j\big) = 2n.\!\!\!
\end{gather}
Let $ M_{2n+1}$ be a differentiable manifold with almost contact structure $(\Phi,\xi,\eta)$. Then there exists a~positive Riemannian metric $g$ such that
\begin{gather*}
g(\xi,X) =\eta(X), \qquad \text{or} \qquad \eta_i = g_{ij}\xi^j, \qquad \text{or} \qquad \eta^t =g\xi,\\ 
g(\Phi X,\Phi Y) = g(X,Y) -\eta(X)\eta(Y),\qquad \text{or} \qquad \Phi^i_h g_{ij}\Phi^j_k = g_{hk} - \eta_h\eta_k, \qquad \text{or}\\
 \Phi^tg\Phi=g-\eta^t\otimes\eta. 
\end{gather*}
\end{Theorem}
 If we put
\begin{gather}\label{2PHI}\hat{\Phi}_{ij}:=g_{ih}\Phi^h_j, \qquad \text{or} \qquad \hat{\Phi}:=g\Phi,\end{gather}
then $ \hat{\Phi}_{ij}=-\hat{\Phi}_{ji}$.

$\hat{\Phi}_{ij}$ is called the {\it associated skew-symmetric tensor of the almost contact metric structure}, see also~\eqref{DFI} below.

\subsubsection{Contact structures}
 Following \cite{boy}, we define
\begin{deff}\label{DFF11}
Let $M_{2n+1}$ be a $C^{\infty}$-manifold of dimension $(2n+1)$. A {\it contact structure} can be given by {\it a codimension one subbundle}~$\mc{D}$ {\it of the tangent bundle} $TM$ {\it which is as far from being integrable as possible}.

Alternatively, the codimension one subbundle $\mc{D}$ of $TM$ can be given as {\it the kernel of a~smooth} 1-form~$\eta$~-- the {\it contact form}, $\mc{D}:=\operatorname{Ker}(\eta)$~-- which satisfies the condition
\begin{gather}\label{CNT}
\eta\wedge ({\rm d} \eta )^n\not=0,
\end{gather}
and from \eqref{CNT} it follows that the distribution $\mc{D}$ is not integrable.
\end{deff}$\mc{D}$ is called {\it the contact distribution} of the {\it strict contact manifold} $(M,\eta)$. {\it A~contact structure } on $M$ is an equivalence class of such 1-forms, where $\eta'\sim \eta$ if there is a nowhere vanishing function on $M$ such that $\eta'=f\eta$, cf.\ \cite[Definition~6.1.7]{boga}.

The tangent space of $M$ has the orthogonal decomposition, see, e.g., \cite[p.~9]{Ja},
\begin{gather*}
TM=\mathcal{D}\oplus \langle \xi\rangle, \qquad \text{where}\qquad \mathcal{D}=\operatorname{Ann}(\eta)=\{X\in \text{TM}\,|\,\eta(X)=0\}.\end{gather*}

\begin{Remark}\label{REM11}The codimension one subbundle $\mc{D}=\operatorname{Ker} (\eta)$ of $TM$ has an almost complex structure $J=\Phi|_{\mc{D}}$.
\end{Remark}

Boothby and Wang \cite{boo} have defined
\begin{deff}\label{BOO} A contact manifold $M$ is said to be {\it homogeneous} if there is a connected Lie group $G$ acting transitively and effectively as a group of differentiable homeomorphisms on~$M$ which leave $\eta$ invariant.
\end{deff}
If the 1-form $\eta$ has the expression given in~\eqref{DEFACD}, then
\begin{gather}\label{DFI}
{\rm d} \eta=\hat{\Phi}_{ij}{\rm d} x^i\wedge {\rm d} x^j, \qquad \text{where} \qquad -2 \hat{\Phi}_{ij}=\pa_i\eta_j-\pa_j\eta_i.
\end{gather}
Note that in \cite[equation~(3.4)]{sas} the minus sign was omitted.

Sasaki has proved, see \cite[Theorem 3.1]{sas}:
\begin{Theorem}\label{TH31} Let $M_{2n+1}$ be a differentiable manifold with $\eta$ the contact form. Then we can find an almost contact metric structure $(\Phi,\xi,\eta,g)$ such that
\begin{gather*}
{\rm d} \eta (X,Y)= g(X,\Phi(Y)),
\end{gather*}
i.e., \eqref{2PHI} is verified with $\hat{\Phi}$ given by \eqref{DFI}.
\end{Theorem}
$M_{2n+1}$ from Theorem \ref{TH31} is said to be a {\it contact $($Riemannian$)$ manifold} associated with $\eta$.

\subsection{Structures on cones}\label{SCC}
Following \cite[p.~201]{boga}, we define
\begin{deff}\label{D14} Let $(M,g)$ be a smooth Riemannian manifold and let us consider the cone $C(M):=M\times \R^+$ endowed with the Riemannian metric
\begin{gather*}
\bar{g}={\rm d} r^2 + r^2g, \qquad r\in \R^+.\end{gather*}
 $(C(M),\bar{g})$ is called the {\it Riemannian cone} (or {\it metric cone}) on~$M$.
\end{deff}

Let $M$ be endowed with the almost contact structure $(\Phi,\xi,\eta)$. Let us define a section $\bar{\Phi}$ of the endomorphism bundle of the $TC(M)=TM\oplus T\R^+$ as
\begin{gather}\label{Ecxxy}
\bar{\Phi}Y=\Phi Y+\eta(Y)\Psi,\qquad \bar{\Phi}\Psi=-\xi,\qquad \text{where} \qquad Y\in TM,\qquad \Psi=r\frac{\pa}{\pa r}\in T\R^+.\end{gather} Then
\begin{Remark}\label{REM9}In the notation \eqref{Ecxxy}, $\bar{\Phi}$ defines an almost complex structure on $TC(M)$.
\end{Remark}

Let \begin{gather*}
\omega:= {\rm d}\big(r^2\eta\big).\end{gather*} In accord with \cite[Proposition 6.5.5]{boga}, we have a {\it symplectization} (or {\it symplectification}) of~$M$:
\begin{Proposition} There is one-to-one correspondence between the contact metric structures on $(M,\xi,\eta,g,\Phi)$ and the almost K\"ahler structures $(C(M),\omega,\bar{g},\bar{\Phi})$.
\end{Proposition}
According to \cite[Definitions 6.4.7, 6.5.7 and 6.5.13]{boga}:
\begin{deff}\label{DF17}
An almost contact structure $(\xi,\eta,\Phi)$ is {\it normal} if the corresponding structure $\bar{\Phi}$ on $C(M)$ is integrable. A~normal contact metric structure $\mc{S}=(M,\xi,\eta,\Phi,g)$ is called a~{\it Sasakian} structure. $M$ has a~$K$-{\it contact structure} if $\xi$ is a Killing vector for $g$.
\end{deff}

Following \cite[p.~47]{BL}, let us introduce
\begin{deff}Let $h$ be a tensor field of type $(1,1)$. Then the Nijenhuis torsion $[h,h]$ of $h$ is the tensor field of type $(1,2)$ given by
\begin{gather*}
[h,h](X,Y)=h^2[X,Y]+[hX,hY]-h[hX,Y]-h[X,hY].\end{gather*}
\end{deff}
Let us define the $(1,2)$-tensor
\begin{gather}\label{NN1}N^1:=[\Phi,\Phi]+2{\rm d} \eta\otimes\xi.\end{gather}
According with \cite[Theorem 6.5.9]{boga}:
 \begin{Theorem}\label{THM14} An almost contact structure $(\xi,\eta,\Phi)$ on $M$ is normal if and only if $N^1=0$. Then $(C(M),\bar{g},\omega,\bar{\Phi})$ is K\"ahler.
\end{Theorem}
\begin{Lemma}\label{LLL}The components of the tensor \eqref{NN1} are given by
\begin{gather}\label{TENC}
\big(N^1\big)^i_{jk}=\Phi^h_j\frac{\pa \Phi^i_k}{\pa x^h}-\Phi^h_k\frac{\pa \Phi^i_j}{\pa x^h} +\Phi^i_h\left(\frac{\pa \Phi^h_j}{\pa x^k}-\frac{\pa \Phi^h_k}{\pa x^j}\right)+2\hat{\Phi}_{jk}\xi^i.
\end{gather}
\end{Lemma}
\begin{proof}In the calculation below we use the expressions
\begin{subequations}\label{724}
\begin{gather}\Phi X =(\Phi X)^i\frac{\pa}{\pa x^i}= \Phi^i_jX^j\frac{\pa}{\pa x^i},\\
[X,Y] =[X,Y]^j\frac{\pa}{\pa x^j}= \left(X^i\frac{\pa Y^j}{\pa x^i}-Y^i\frac{\pa X^j}{\pa x^i}\right) \frac{\pa}{\pa x^j},\\
N^1(X,Y) =\big(N^1\big)^i_{jk}X^jY^k\frac{\pa}{\pa x^i}.
\end{gather}
\end{subequations}
Let us introduce the notation \begin{gather*} A:=\Phi^2[X,Y],\qquad B:=[\Phi X,\Phi Y],\qquad C:=\Phi [\Phi X,Y], \\
 D:= \Phi[X,\Phi Y], \qquad E:= 2{\rm d} \eta\otimes \xi (X,Y).\end{gather*}
With \eqref{724}, we get for $A,\dots,E$ the expressions
\begin{gather} A =\Phi(\Phi[X,Y])^i\frac{\pa}{\pa x^i}=\Phi
^a_b(\Phi[X,Y])^b\frac{\pa }{\pa x^i}= \Phi^a_b\Phi^b_k[X,Y]^k\frac{\pa }{\pa x^a}\nonumber\\
\hphantom{A}{} = \big({-}\delta^a_k+\xi^a\eta_k\big)[X,Y]^k\frac{\pa}{\pa x^a},\label{EA}\\
 B = [\Phi X,\Phi Y]^a\frac{\pa }{\pa x^a}=\left[(\Phi X)^c\frac{\pa (\Phi
 Y)^a}{\pa x^c}-(\Phi Y)^c\frac{\pa(\Phi X)^a}{\pa
 x^c}\right]\frac{\pa}{\pa x^a}\nonumber\\
\hphantom{B}{} =\left[X^dY^b\left(\Phi^c_d\frac{\pa \Phi ^a_b}{\pa x^c}-\Phi^c_b\frac{\pa
 \Phi^a_d}{\pa x^c}\right)+\Phi^c_d\Phi^a_b\left(X^d\frac{\pa Y^b}{\pa x^c}-Y^d\frac{\pa
 X^b}{\pa x^c}\right)\right]\frac{\pa}{\pa x^a},\label{EB}\\
\label{EC}
C = \Phi^a_j[\Phi X,Y]^j\frac{\pa}{\pa x^a}= \Phi^a_j\left(\Phi^i_kX^k\frac{\pa Y^j}{\pa x^i}-Y^i\frac{\pa(\Phi^j_h X^h)}{\pa x^i}\right)\frac{\pa}{\pa x^a},\\
-C-D =\left[\Phi^a_j\Phi^i_k\left(Y^k\frac{\pa X^j}{\pa x^i}-X^k\frac{\pa Y^j}{\pa
 x^i}\right)+\big({-}\delta^a_k+\xi^a\eta_k\big)[Y,X]^k\right.\nonumber\\
 \left.\hphantom{-C-D =}{} +Y^iX^k\Phi^a_j\left(\frac{\pa \Phi^j_k}{\pa x^i}-\frac{\Phi^j_i}{\pa x^k}\right)\right]\frac{\pa}{\pa x^a},\label{ECD}
\\ \label{ECE}
E=2\hat{\Phi}_{ij}X^iY^j\xi^a\frac{\pa}{\pa x^a}.
\end{gather}
Introducing the values of $A$, $B$, $D+C$ and $E$ obtained
in equations \eqref{EA}, \eqref{EB}, \eqref{ECD}, respectively
\eqref{ECE}, we get for $A+B-C-D+E$ the values given in \eqref{TENC}.
\end{proof}

Note that formula given in \cite[pp.~7--10]{sas}
\begin{gather}\label{TENN}
\big(N^1\big)^i_{jk}=\Phi^h_k\big(\pa_h\Phi^i_j-\pa_j\Phi^i_h\big)-\Phi^h_j\big(\pa_h\Phi^i_k-\pa_k\Phi^i_h\big)+\big(\pa_j\xi^i\big)\eta_k-\big(\pa_k\xi^i\big)\eta_j
\end{gather}
is wrong. The same wrong formula appears also in \cite[equation~(3.7)]{tan2}.

The Heisenberg group $H_1$ is a Sasaki manifold \cite{boy}.

\subsection*{Acknowledgements}
This research was conducted in the framework of the
ANCS project programs PN 16 42 01 01/2016, 18 09 01 01/2018, 19 06
01 01/2019. I had the idea to apply Lemma~\ref{PRR} after the talk of Professor Zdan\v{e}k Du\v{s}ek at the 1st International Conference on Differential Geometry
(April 11--15, 2016, Fez, Morocco). I am grateful to Professor Zdan\v{e}k for his
correspondence in the first stages of the preparation of this paper. I also would
like to thank Professor Mohamed Tahar Kadaoul Abbassi for the hospitality during the
Fez conference and the partial financial support. I would like to thank to Professor
G.W.~Gibbons for answering to an e-mail. I am grateful to Professor
M.~Visinescu for initiating me in the world of Sasaki manifolds.
Thanks are also addressed to Professor R.D.~Grigore for
suggestions in some calculations. I am grateful to Professors Dmitri
Alekseevsky and Vicente Cort\'es for criticism and suggestions on
the first version of this paper. The author thanks the unknown referees'
who through their recommendations contributed to the improvement of
the text of the paper. The author thanks Drs. I.~Berceanu and M.~Babalic
for help in preparation of the text.

\pdfbookmark[1]{References}{ref}
\LastPageEnding

\end{document}